\documentclass[a4paper,12pt]{amsart}
\usepackage[utf8]{inputenc}
\usepackage[T1]{fontenc}
\usepackage[UKenglish]{babel}
\usepackage[margin=20mm]{geometry}
\usepackage{color}
\usepackage{bbm,bm}

\usepackage{subfigure}

\usepackage{graphicx}

\usepackage{amsmath,amssymb,amsfonts,amsthm}
\usepackage{mathrsfs,eucal,dsfont}
\usepackage{verbatim,enumitem}

\usepackage{hyperref,url}
\usepackage{amssymb, amsmath, amsthm, amscd}

\usepackage{eucal,mathrsfs,dsfont}
\usepackage{color}

\renewcommand{\leq}{\le}
\renewcommand{\geq}{\ge}

\newcommand{\e}{\text{e}}

\newcommand{\I}{\mathds 1}

\def\d{{\rm d}}
\def\<{\langle}
\def\>{\rangle}

\newtheorem{theorem}{Theorem}[section]
\newtheorem{lemma}[theorem]{Lemma}
\newtheorem{proposition}[theorem]{Proposition}
\newtheorem{corollary}[theorem]{Corollary}
\numberwithin{equation}{section}

\theoremstyle{definition}
\newtheorem{definition}[theorem]{Definition}

\newtheorem{remark}[theorem]{Remark}

\begin{document}
\allowdisplaybreaks
\title[Resistances in one-dimensional critical long-range percolation]
{\bfseries  The polynomial growth of effective resistances in one-dimensional critical long-range percolation}

\author{Jian Ding \qquad  Zherui Fan  \qquad  Lu-Jing Huang}

\thanks{\emph{J. Ding:}
School of Mathematical Sciences, Peking University, Beijing, China.
  \texttt{dingjian@math.pku.edu.cn}}
\thanks{\emph{Z. Fan:}
School of Mathematical Sciences, Peking University, Beijing, China.
  \texttt{1900010670@pku.edu.cn}}

\thanks{\emph{L.-J. Huang:}
School of Mathematics and Statistics \& Key Laboratory of Analytical Mathematics and Applications (Ministry of Education), Fujian Normal University, Fuzhou, China.
  \texttt{huanglj@fjnu.edu.cn}}


\date{}
\maketitle

\begin{abstract}
We study the critical long-range percolation on $\mathds{Z}$, where an edge connects $i$ and $j$ independently with probability $1-\exp\{-\beta\int_i^{i+1}\int_j^{j+1}|u-v|^{-2}\d u\d v\}$ for $|i-j|>1$ for some fixed $\beta>0$ and with probability 1 for $|i-j|=1$. Viewing this as a random electric network where each edge has a unit conductance, we show that the effective resistances from 0 to $[-n,n]^c$ and from the interval $[-n,n]$ to $[-2n,2n]^c$ (conditioned on no edge joining $[-n,n]$ and $[-2n,2n]^c$)  both grow like $n^{\delta(\beta)}$ for some $\delta(\beta)\in (0,1)$.

\noindent \textbf{Keywords:} Long-range percolation, effective resistance, polynomial growth.

\medskip

\noindent \textbf{MSC 2020:} 60K35, 82B27, 82B43

\end{abstract}
\allowdisplaybreaks

\tableofcontents

\section{Introduction}\label{intro}

Long-range percolation (LRP), initially introduced by \cite{S83,ZPL83}, was proposed to study the effects of long-range interactions on phase transitions. It is a percolation model on $\mathds{Z}^d$, in which each pair of vertices may be connected by a bond with probability (roughly speaking) proportional to $\beta |i - j|^{-s}$ for some parameters $\beta, s > 0$.
Over the years, LRP has been a subject of extensive study. One of the key early findings is the existence of multiple phase transitions regarding the infinite cluster as the parameters $s$ and $\beta$ vary.
For instance, in one dimension, \cite{S83} showed that percolation does not occur when $s>2$ unless the nearest-neighbor connecting probability $p_1 = 1$, while \cite{NS86} established that an infinite cluster exists if $s<2$.
Furthermore, it was found that percolation occurs at the critical case of $s=2$ when $\beta$ is sufficiently large and $p_1$ is close to 1.
This was further refined by \cite{AN86} (see also \cite{DGT24}), which showed the critical behavior of $\beta$. A comprehensive overview of these developments can be found in \cite{DFH23}.

Another intriguing direction involves the behavior of random walks on the infinite cluster of LRP. While important progress has been made on the recurrence properties and scaling limits of random walks in the non-critical case (see \cite{CS13,BT24}), much remains to be understood about the random walk behavior in the critical case, particularly in low-dimensional settings (see Section \ref{relatew} below for further details).

Motivated by this, we aim to study the growth properties of effective resistances in the critical LRP model, which is closely related to various aspects of random walks (see Section \ref{relatew}). We will focus on the one-dimensional critical LRP model where the connecting probabilities exhibit strong scaling invariance, making it a particularly convenient choice. Specifically, we set $p_1=1$, ensuring the connectivity trivially.
For $i,j\in \mathds{Z}$ with $|i-j|>1$, let the edge $\langle i,j\rangle$ occur independently with probability
\begin{equation}\label{LRP}
p_{i,j}:=1-\exp\left\{-\beta\int_i^{i+1}\int_j^{j+1}\frac{1}{|x-y|^2}\d x\d y\right\},
\end{equation}
where $\beta>0$ serves as a parameter of this LRP model.
This model will also be referred to as a $\beta$-LRP model, with $\mathcal{E}$ denoting for the edge set. Additionally, for ease of notation, we will also use $\sim$ to denote edges, that is, $i\sim j$ implies $\langle i,j \rangle \in \mathcal{E}$.

We now assign unit conductance to each edge in the $\beta$-LRP model, which allows us to view the $\beta$-LRP model as a random electric network on $\mathds{Z}$.
To introduce effective resistance, we define the notion of a flow as follows.
Let $f$ be a function defined on $\mathds{Z}^2$. We call $f$ a flow on the $\beta$-LRP model if it satisfies
\begin{equation*}
f_{ij}=
\begin{cases}
-f_{ji},\quad & \langle i,j\rangle\in \mathcal{E}, \\
0,\quad &\text{otherwise}
\end{cases}
\end{equation*}
for all $i,j\in \mathds{Z}$.
Thus, the net flow out of a point $i$ is defined as $f_{(i)}:=\sum_{j\neq i} f_{ij}$.
Moreover, given two finite disjoint subsets $I_1,I_2\subset \mathds{Z}$, we say that $f$ is a unit flow from $I_1$ to $I_2$ if it is a flow satisfying
$$
\sum_{i\in I_1}f_{(i)}=1,\quad\sum_{i\in I_2}f_{(i)}=-1\quad  \text{and}\quad f_{(j)}=0\quad \text{for all }j\notin I_1\cup I_2.
$$
We then define the effective resistance between two disjoint subsets $I_1, I_2\subset \mathds{Z}$ as
\begin{equation}\label{eff-resis}
R(I_1,I_2)=\inf_{J_1, J_2}\inf\left\{\frac{1}{2}\sum_{i\sim j}f^2_{ij}:\ f\ \text{is a unit flow from $J_1$ to } J_2\right\},
\end{equation}
where the first infimum is taken over all finite subsets $J_1 \subset I_1$ and $J_2 \subset I_2$.
We refer to the quantity in the bracket of \eqref{eff-resis} as the energy (of the flow $f$).
In particular, when $I_1=\{i\}$ is a singleton, we denote $R(i,I_2)$ as $R(I_1,I_2)$.

Our main result presents that the effective resistances from 0 to $[-n,n]^c$ and from the interval $[-n,n]$ to $[-2n,2n]^c$ (conditioned on no edge joining $[-n,n]$ and $[-2n,2n]^c$)  both grow like $n^\delta$ for some $\delta\in(0,1)$, as detailed in the following theorem.

\begin{theorem}\label{mainr1}
For all $\beta>0$, there exists an exponent $\delta=\delta(\beta)\in (0,1)$ {\rm(}depending only on $\beta${\rm)} such that for all $n\in \mathds{N}$,
$$
\mathds{E}\left[R(0,[-n,n]^c)\right]\asymp_P R(0,[-n,n]^c)\asymp_P n^\delta
$$
and conditioned on $[-n,n]\nsim[-2n,2n]^c$,
\begin{equation*}
 \mathds{E}\left[R([-n,n],[-2n,2n]^c )\ | [-n,n]\nsim[-2n,2n]^c\right]\asymp_P R([-n,n],[-2n,2n]^c )\asymp_P n^\delta,
\end{equation*}
where $A(n)\asymp_P B(n)$ means that for all $\varepsilon >0$, there exist constants $0<c<C<\infty$ such that $\mathds{P}[cB(n)\leq A(n)\leq CB(n)]>1-\varepsilon$ for all $n\in \mathds{N}$.
\end{theorem}

As a corollary for the proof of Theorem  \ref{mainr1} below, we can obtain the following lower tail of  the resistance.

\begin{corollary}\label{cor-lt}
For all $\beta>0$ and sufficiently small $\varepsilon\in(0,1)$, there exists a constant $q>0$ {\rm(}depending only on $\beta${\rm)} such that for all $n\in \mathds{N}$,
\begin{equation*}
\mathds{P}\left[R(0,[-n,n]^c)\geq \varepsilon n^{\delta}\right]\geq 1-\varepsilon^{q},
\end{equation*}
where $\delta>0$ is the constant defined in Theorem \ref{mainr1}.
\end{corollary}

\begin{remark}
(1) Note that the fact $\delta<1$ follows from \cite{Baumler23a} and the fact $\delta > 0$ follows from \cite{DFH24+}. All the other content of Theorem \ref{mainr1} is the contribution of the present work, including the existence of $\delta$.

(2) It is a natural question to derive the scaling limit of the resistance, from which one may hope to derive the scaling limit of the random walk using the method of resistance form as in \cite{Croydon18}.
However, this is challenging even in light of the proof for the scaling limit of the chemical distance as in \cite{DFH23}, since the resistance is a metric of rather different nature (for example, it is not a length metric, which is an important axiom in \cite{DFH23}, inherited from \cite{GM21,DG23}).

(3) As kindly suggested by Takashi Kumagai, our current understanding on effective resistances already allow to derive a few random walk estimates using techniques in \cite{CCK22}. We chose not to do it here since on the one hand this paper is already rather long, and on the other hand we feel it may be more efficient to derive properties of random walks after a full development on resistances (ideally, after deriving the scaling limit, although it appears to be a major challenge currently). As a result, we refrain from deriving estimates on random walks at this point.
\end{remark}

\subsection{Related work}\label{relatew}

The effective resistance serves as a powerful tool for analyzing random walks on graphs in general. It is particularly relevant to various aspects of random walks, including recurrence/transience, the heat kernel and mixing times, see e.g. \cite{AF02,LPW09,K14}.

Recently, inspired by the study of random walks on the LRP model, there has been growing interest in analyzing the effective resistances associated with it.
More precisely, consider a sequence $\{p_x\}_{x\in \mathds{Z}^d}$, where $p_x\in [0,1]$ and $p_x=p_{x'}$ for all $x,x'\in \mathds{Z}^d$ such that $|x|=|x'|$. Assume that
\begin{equation*}
0<\beta:=\lim_{|x|\rightarrow \infty}p_x|x|^s<\infty
\end{equation*}
for some parameter $s>0$. The LRP on $\mathds{Z}^d$ is then defined by edges $\langle x,y\rangle$ occurring independently with probability $p_{x-y}$. In particular, in one dimension $p_1$ is the nearest-neighbor connecting probability, consistent with our aforementioned convention.

As previously mentioned, significant progress has been made in understanding the properties of random walks in the non-critical (i.e.\ $(d,s)\neq (d,2d)$) LRP model. We will now provide a detailed overview of studies in this area, particularly in relation to effective resistances.
In \cite{KM08}, the authors developed methods to estimate volumes and effective resistances from points to boxes, leading to  corresponding heat kernel estimates for the case when $d=1,\ s>2$ and $p_1=1$.
In fact, they established their results in a more general setting involving random media.
Additionally, \cite{M08} provided some up-to-constant estimates for the box-to-box resistances for $d=1,\ s\in (1, 2) \cup (2, \infty)$ as well as for $d \geq 2$ and $s\in (d, d+2) \cup (d +2, \infty)$.
Furthermore, the scaling limits of random walks on the LRP model were studied by \cite{CS13} and \cite{BT24}.
Specifically, \cite{CS13} showed that for $s\in (d,d+1)$, the random walk on the infinite cluster of the LRP model converges to an $\alpha$-stable process with $\alpha=s-d$.
Notably, in the case where $d=1$ and $s>2$, \cite{CS13} further established that the random walk converges to a Brownian motion. These results were extended in \cite{BT24} to the case of  $d\geq 2$ and $s\in [d+1,d+2)$, which showed that the associated random walk also converges to an $\alpha$-stable process with $\alpha=s-d$.

However, the study of  effective resistance and random walk in the critical LRP model, especially for $d\in \{1,2\}$,  presents significant challenges. For the critical LRP with $d\geq 3$, the growth exponent of resistance has been obtained in \cite{M08}.
In \cite{B02}, it was shown that for $s=2d$ with $d\in \{1,2\}$, the random walk is recurrent. This was established by proving that the effective resistance from the origin to $[-n,n]^c$ diverges as $n$ increases.
This result was later extended in \cite{Baum23b} to encompass more general LRP models characterized by  weight distributions that satisfy certain moment conditions.
It is worth noting that \cite{M08} also presented  bounds on box-to-box resistances for dimensions $d\in\{1,2\}$, including a specific constant lower bound for the case when $d=1$.

More recently, in \cite{DFH24+}, we established polynomial lower bounds for both point-to-box and box-to-box resistances in the case of $d=1$ and $s=2$.
A key tool to prove such lower bounds for effective resistances in the critical model was multi-scale analysis. 
This approach enabled us to effectively combine energies generated by flows passing through subintervals of various scales, allowing us to derive effective lower bounds. Based on this, the goal of this paper is to further establish the polynomial growth of the effective resistances, namely Theorem \ref{mainr1}.

\subsection{Outline of the proof}
In this section, we offer an overview for the proof of Theorem \ref{mainr1}.
At a high level, our strategy for proving the polynomial growth of the effective resistance $R(0,[-n,n]^c)$ shares similarities with the approach used to establish the polynomial growth of graph distance in \cite{Baumler23a}.
However, significant challenges arise in the context of the effective resistance.
Intuitively, when considering graph distance,
traversing an edge counts as taking one step; whereas in the case of effective resistance, the contribution of each edge depends on the amount of flow passing through it.
Therefore, when a flow passes through multiple edges from a ``bad interval'' (i.e. where the energy generated within the interval is too low), it becomes dispersed among these edges. A natural attempt to address the insufficiency of energy in these ``bad intervals'' is to make use of the energy generated by flows through ``good intervals'' nearby.
However, this leads to considerable complexity in the renormalization.


In order to describe our approach in more details, we begin by introducing some notations.
Fix $\beta>0$. For $n\in \mathds{N}$ and $i,j\in [0,n)$, we let $R_{[0,n)}(i,j)$ be the effective resistance between $i$ and $j$ restricted to the interval $[0,n)$, that is,
\begin{equation}\label{def-Rinsd}
\begin{aligned}
R_{[0,n)}(i,j)&=\inf\left\{\frac{1}{2}\sum_{k\sim l}f^2_{kl}: f\ \text{is a unit flow from $i$ to $j$ and }f_{kl}=0\right.\\
&\quad \quad \quad \quad\quad \quad \quad \quad\quad \quad \quad \quad \quad \quad  \text{for all }\langle k,l\rangle\in \mathcal{E}\setminus\mathcal{E}_{[0,n)^2} \Bigg\}.
\end{aligned}
\end{equation}
Denote
\begin{equation}\label{def-Lambda}
\Lambda (n,\beta)=\max_{i,j\in [0,n)}\mathds{E}\left[R_{[0,n)}(i,j)\right].
\end{equation}
For simplicity, we will write $\Lambda(n,\beta)$ as $\Lambda(n)$ when there is no risk of confusion.

As we will show in Section \ref{sect-severaltype}, several types of resistances, including the point-to-box resistance $R(0,[-n,n]^c)$ and the box-to-box resistance $R([-n,n],[-2n,2n]^c)$ (conditioned on the absence of edges connecting $[-n,n]$ and $[-2n,2n]^c$), are comparable to $\Lambda(n)$.
Therefore, to prove the polynomial growth of $R(0,[-n,n]^c)$ and $R([-n,n],[-2n,2n]^c)$, it suffices to establish that $\Lambda(n)$ exhibits both submultiplicativity and supermultiplicativity.

The proof of the submultiplicativity of $\Lambda(n)$ (see Section~\ref{sect-submult}) primarily relies on a renormalization argument, where the main additional novelty in light of \cite{Baumler23a} is a comparison result involving flows in two similar electric networks (see Lemma \ref{network}):
when an electric flow enters a certain point through multiple edges simultaneously, removing one of the edges will increase flows through the remaining edges.
By this comparison, we can overcome the difficulty that multiple edges sharing a common endpoint may disperse the flow too thinly. 

The key challenge of this work is to establish the supermultiplicativity of $\Lambda(n)$.  The primary tools we employ are coarse-graining argument and renormalization.

First, by using a variational  characterization of effective resistance from \cite[Proposition 2.3]{BDG20} and an enhanced lower bound for resistance from \cite[Theorem 1.1]{DFH24+},
we can show that the effective resistance between two points that are relatively close to each other can be controlled by the effective resistance between two points that are far away (see Section \ref{sect-far}).
This implies that $\Lambda(n)$ is within a constant factor of the resistance $\mathds{E}\left[R_ {[0,n)}(0,n-1)\right]$. Therefore, we only need to estimate $\mathds{E}\left[R_ {[0,n)}(0,n-1)\right]$.

Next, the most crucial step is to establish the weak supermultiplicativity of $\mathds{E}\left[R_ {[0,n)}(0,n-1)\right]$, as incorporated in Proposition \ref{weak-supermult}.
To achieve this, fix two sufficiently large $m,n\in \mathds{N}$. We divide the interval $[0,mn)$ into $n$ smaller intervals, each with length $m$.
We refer to such an interval as very good if its internal resistance is sufficiently large (see Definition \ref{def-alphagood-1}).
By selecting appropriate parameters, we can show that with high probability, such an interval is indeed very good (see Proposition \ref{prop-alphagood-1}).
This implies that when a unit flow travels from $0$ to $mn-1$, it will generate sufficiently large energy in most of the small intervals.
However, there are some small intervals where the flow does not generate enough energy. Our main effort is to address these intervals.
To do this, we employ a coarse-graining argument to show that, with high probability, we can find intervals at different scales to cover those problematic intervals (see Section \ref{sect-cg}).
Consequently, we can obtain the weak supermultiplicativity of $\mathds{E}\left[R_ {[0,n)}(0,n-1)\right]$ by taking a union bound over all unit flows (see Section \ref{sect-proofwm}). It is worth emphasizing that the weak supermultiplicativity of $\mathds{E}\left[R_ {[0,n)}(0,n-1)\right]$ plays an important role in the proof of the comparability between serval types of resistances discussed in Section \ref{sect-severaltype}.

Finally, by using the resistance estimates from Section \ref{sect-severaltype} and employing a similar coarse-graining argument as mentioned above, we can establish the supermultiplicativity of $\Lambda(n)$ in Section \ref{sect-supermult}, and then prove Theorem \ref{mainr1} in Section \ref{sect-proofmr}.

\subsection{Notational conventions}
We denote $\mathds{N}=\{1,2,3,\cdots\}$. For any $i,j\in \mathds{N}\cup \{0\}$ with $i<j$, we will define $[i,j]=\{i,i+1,\cdots, j\}$ and $[i,j)=\{i,i+1,\cdots, j-1\}$. When we refer to an interval $I$, it always mean $I\cap \mathds{Z}$.
For any interval $I,J\subset \mathds{Z}$, we denote $\mathcal{E}_{I\times J}$ for the edge set consisting of edges with one endpoint in $I$ and the other in $J$. Additionally, we denote $R_I(x,y)$ as the effective resistance within the interval $I$, i.e., we only consider edges with both endpoints inside $I$ only. For any $u\in \mathds{Z}$ and $r\in \mathds{N}$, let $B_r(u):=(u-r,u+r)\cap \mathds{Z}$.

For any graph $G=(V,E)$ and any $I_1,I_2\subset V$, we write $E_{I_1\times I_2}$ for the edge set consisting of edges that have  one endpoint in $I_1$ and  the other in $I_2$. For any vertex set $I\subset V$ and any edge set $H\subset E$, we denote $\#I$ and $\#H$ as the numbers of vertices and edges in $I$ and $H$, respectively.

Throughout the paper, we use $c_1, C_1, c_2, \ldots$ to denote positive constants that are numbered by sections. Note that these constants may vary from section to section.
Additionally, we use $c_{\cdot,*}$ and $C_{\cdot,*}$ to denote consistent positive constants that remain unchanged throughout the paper.

\section{Submultiplicativity}\label{sect-submult}

Recalling \eqref{def-Rinsd} and \eqref{def-Lambda}, the objective of this section is to establish the submultiplicativity of $\Lambda (n)$ as follows.

\begin{proposition}\label{submult}
For all $\beta>0$, the sequence $\{\Lambda(n)\}_{n\geq 1}$ is submultiplicative. That is, there exists a constant $C_{1,*}=C_{1,*}(\beta)<\infty$ {\rm(}depending only on $\beta${\rm)} such that for all $n,m\in \mathds{N}$,
$$
\Lambda(mn)\leq C_{1,*}\Lambda(m)\Lambda(n).
$$
\end{proposition}

The main approach for proving Proposition \ref{submult} involves comparing the electric flows through a single edge before and after increasing the conductance in a nearby edge. To streamline its use in the proof, we will perform preliminary work for general electric networks in Section \ref{cmp-2graph}. Then, in Section \ref{proof-submult}, we will present the proof of Proposition \ref{submult} by using the renormalization and the comparison result established in Section \ref{cmp-2graph}.

\subsection{A comparison result for electric networks}\label{cmp-2graph}
In order to state and prove the aforementioned comparison result, we will introduce some notations that will be used consistently throughout this subsection.

Let $G=(V,E)$ be a finite, connected graph with undirected edges.
For each edge $e=\langle u,v\rangle\in E$, we assign a deterministic conductance $c_e=c_{uv}\in [0,\infty)$, thereby forming an electric network denoted as $(G, \{c_e\}_{e\in E})$.
Note that an edge with a conductance of 0 is considered  removed from the graph here.
For any two distinct vertices $x,y\in V$, let $R(x,y;\{c_e\}_{e\in E})$ represent the effective resistance between $x$ and $y$ in the network $(G, \{c_e\}_{e\in E})$, which is defined as
\begin{equation}\label{def-resist}
R\left(x,y;\{c_e\}_{e\in E}\right)=\inf\left\{\frac{1}{2}\sum_{u\sim v}\frac{f_{uv}^2}{c_{uv}}:\ f \text{ is a unit flow from $x$ to $y$ in }G\right\}.
\end{equation}

In the following, we will fix conductances $\{c_e\}_{e\in E}$ and select two distinct vertices $x,y\in V$.
Let $g$ denote the unit electric flow from $x$ to $y$, i.e., the flow attaining the infimum in \eqref{def-resist}. Then we have
\begin{equation*}\label{def-g}
R(x,y;\{c_e\}_{e\in E})=\frac{1}{2}\sum_{u\sim v} \frac{g^2_{uv}}{c_{uv}}.
\end{equation*}
Now assume that $w,w_1,w_2\in V$ are three distinct vertices such that $g_{w_1w},\ g_{w_2w}>0$.
We define a new sequence of conductances $\{c'_e\}_{e\in E}$ on the edges of the graph $G$ as follows:
\begin{equation}\label{def-ce'}
c'_e=
\begin{cases}
c_{w_2w}+c_\triangle,\quad & e=\langle w_2,w\rangle,\\
c_e,\quad & e\in E\setminus \{\langle w_2,w\rangle\},
\end{cases}
\end{equation}
where $c_\triangle>0$.
Similarly, let $g'$ denote the unit electric flow from $x$ to $y$ with respect to $(G,\{c'_e\}_{e\in E})$.

We then present the following comparison result between $g_{w_1 w}$ and $g'_{w_1w}$ (see Figure \ref{flow-submult} below for an illustration).

\begin{lemma}\label{network}
For fixed $x,y\in V$, let $g$ be a unit electric flow from $x$ to $y$ in the network $(G,\{c_e\}_{e\in V})$.
Suppose that $w,w_1,w_2$ are three distinct vertices such that $g_{w_1w},g_{w_2w}>0$.
Then for the network $(G,\{c'_e\}_{e\in E})$ defined in \eqref{def-ce'}, we have
\begin{equation}\label{gw1}
g'_{w_1w}\leq g_{w_1w},
\end{equation}
where $g'$ is the  unit electric flow from $x$ to $y$ in the network $(G,\{c'_e\}_{e\in E})$.
\end{lemma}

\begin{proof}
Without loss of generality, we assume that $V=\{1,\cdots,n\}$ and $y=n$. Let $\{U_v\}_{v\in V}$ (resp. $\{U'_v\}_{v\in V}$) be the electric potential field with $U_n=0$ (resp.\ $U'_n=0$) under the flow $g$ (resp.\ $g'$) on the electric network $(G,\{c_e\}_{e\in E})$ (resp.\ $(G,\{c'_e\}_{e\in E})$). Thus, by Ohm's law we have
\begin{equation}\label{def-U}
U_u-U_v=g_{uv}/c_{uv}\quad\text{for all } \langle u,v\rangle\in E.
\end{equation}
From the above definition, we have
  \begin{equation}\label{Ohm}
    \sum_{v\neq u}c_{uv}(U_u-U_v)=\sum_{v\neq u}g_{uv}=\delta_{x}(u)\quad \text{for all }u\in V\setminus\{n\},
  \end{equation}
  where $\delta_{\cdot}$ denotes the Dirac measure.

In addition, for any $i\in \{1,2,\cdots,n-1\}$, let $\mathbf{D}_i$ be the $(n-1)$-dimensional column vector whose $i$-th component is 1 and other components are 0. Denote  $\mathbf{U}=(U_1,\cdots,U_{n-1})^T$. Then \eqref{Ohm} is equivalent to
  \begin{equation}\label{KCL-G1}
    \mathbf{C}\mathbf{U}=\mathbf{D}_x,
  \end{equation}
  where $\mathbf{C}=(C_{uv})_{u,v\in[1,n)}$ is the $(n-1)\times (n-1)$ matrix defined by
  \begin{equation}\label{def-Cij}
    C_{uv}=
    \begin{cases}
      -c_{uv}, &u\neq v,\\
      \sum_{w\neq u} c_{uw}, &u=v.
    \end{cases}
  \end{equation}
Moreover, based on the uniqueness of the electric flow on a finite connected graph, it is evident that the matrix $\mathbf{C}$ is invertible. We define $\mathbf{Z}$ as its inverse, that is,   $\mathbf{Z}=(Z_{uv})_{u,v\in[1,n)}=\mathbf{C}^{-1}$.
Therefore, \eqref{KCL-G1} implies that
  \begin{equation}\label{KCL-G2}
    \mathbf{U}=\mathbf{Z}\mathbf{D}_x.
  \end{equation}

Similarly, for the electric network $(G,\{c'_e\}_{e\in E})$, let $\mathbf{U}'=(U_1',\cdots,U_{n-1}')^T$. We also have
  \begin{equation}\label{KCL-G-1}
    \mathbf{C'}\mathbf{U'}=\mathbf{D}_x\quad \text{and}\quad  \mathbf{U}'=\mathbf{Z}'\mathbf{D}_x,
  \end{equation}
where $\mathbf{C}'$ is defined in \eqref{def-Cij} by replacing $\{c_e\}_{e\in E}$ with $\{c'_e\}_{e\in E}$, and $\mathbf{Z}'=(\mathbf{C}')^{-1}$.

Recall that $w,w_1,w_2\in V$ are three distinct vertices such that $g_{w_1w},\ g_{w_2w}>0$. Combining this with \eqref{def-U} and the assumption $U_n=0$, we note that $w_1\neq n$.
In the following, we will divide our discussion into two cases based on whether or not $w=n$.

We start with the case where $w=n$.
In this scenario, we have
$$
g_{w_1w}=c_{w_1w}U_{w_1} \quad \text{and}\quad  g'_{w_1w}=c_{w_1w}U'_{w_1}.
$$
Thus, to prove \eqref{gw1}, it suffices to show that $U_{w_1}\geq U_{w_1}'$.
Moreover, by \eqref{KCL-G2} and \eqref{KCL-G-1}, this is equivalent to proving
  \begin{equation}\label{goal-A}
    Z_{w_1 x}\geq Z_{w_1x}'.
  \end{equation}
Indeed, from $w=n$ we see that $\mathbf{C}'=\mathbf{C}+c_\triangle\mathbf{D}_{w_2}\mathbf{D}_{w_2}^T$ with $c_\triangle>0$. From this, along with the facts $\mathbf{Z}=\mathbf{C}^{-1}$ and $\mathbf{Z}'=(\mathbf{C}')^{-1}$, we claim that
  \begin{equation}\label{relation-Z-2}
    \mathbf{Z}'=\mathbf{Z}-\widehat{c_\triangle}\mathbf{Z}\mathbf{D}_{w_2}\mathbf{D}_{w_2}^T\mathbf{Z}
  \end{equation}
  with $\widehat{c_\triangle}:=\frac{c_\triangle}{1+c_\triangle\mathbf{D}_{w_2}^T\mathbf{Z}\mathbf{D}_{w_2}}$.
Indeed, since $\mathbf{D}_\cdot$ is a column vector, it is clear that $\mathbf{D}_{w_2}^T\mathbf{Z} \mathbf{D}_{w_2}\in \mathds{R}$. Combining this with the definitions of $\widehat{c_\triangle}$ and $\mathbf{C}'$, we get that
\begin{equation*}
\begin{aligned}
\left(\mathbf{Z}-\widehat{c_\triangle}\mathbf{Z}\mathbf{D}_{w_2}\mathbf{D}_{w_2}^T\mathbf{Z} \right)\mathbf{C}'
&=\left(\mathbf{Z}-\widehat{c_\triangle}\mathbf{Z}\mathbf{D}_{w_2}\mathbf{D}_{w_2}^T\mathbf{Z} \right)\left(\mathbf{C}+c_\triangle\mathbf{D}_{w_2}\mathbf{D}_{w_2}^T\right)\\
&=\mathbf{I}+c_\triangle \mathbf{Z}\mathbf{D}_{w_2}\mathbf{D}_{w_2}^T
-\widehat{c_\triangle}\mathbf{Z}\mathbf{D}_{w_2}\mathbf{D}_{w_2}^T
-\widehat{c_\triangle}c_\triangle\left(\mathbf{D}_{w_2}^T\mathbf{Z}\mathbf{D}_{w_2}\right)\mathbf{Z}\mathbf{D}_{w_2}\mathbf{D}_{w_2}^T\\
&=\mathbf{I}+\left(c_\triangle-\widehat{c_\triangle}-\widehat{c_\triangle}c_\triangle
\left(\mathbf{D}_{w_2}^T\mathbf{Z}\mathbf{D}_{w_2}\right)\right)\mathbf{Z}\mathbf{D}_{w_2}\mathbf{D}_{w_2}^T=\mathbf{I}.
\end{aligned}
\end{equation*}
This implies \eqref{relation-Z-2}.

We now claim that $\widehat{c_\triangle}>0$.
In fact, by the definition of the matrix $\mathbf{C}$ in \eqref{def-Cij}, it can be easily verified that $\mathbf{C}$ is a diagonally dominated symmetric matrix with all diagonal entries positive.
This implies that $\mathbf{C}$ is positive definite, leading to the conclusion that $\mathbf{Z}=\mathbf{C}^{-1}$ is also positive definite.
Consequently, given $c_\triangle>0$, we find that $\widehat{c_\triangle}>0$.
Combining this with \eqref{relation-Z-1}, we can conclude that to establish \eqref{goal-A}, it suffices to prove
  \begin{equation}\label{ZZ}
    \left(\mathbf{Z}\mathbf{D}_{w_2}\mathbf{D}_{w_2}^T\mathbf{Z}\right)_{w_1x}=Z_{w_1w_2}Z_{w_2x}\geq 0.
  \end{equation}
To do this, from \eqref{KCL-G2}, the assumptions $w=n$, $g_{w_2w}>0$ and $U_n=0$, we have
\begin{equation}\label{Zw2x}
Z_{w_2x}=U_{w_2}=U_{w_2}-U_w=c_{w_2w}g_{w_2w}>0.
\end{equation}
Additionally, consider a unit electric flow from $w_2$ to $w$ with respect to $\{c_e\}_{e\in E}$.
Recall $w = n$ and denote the associated electric potential field as $\mathbf{U}^*=(U_1^*,\cdots,U_{n-1}^*)^T$ with $U_w^*=U_n^*=0$.
Then, similar to \eqref{KCL-G1} and \eqref{KCL-G2}, we get that
  \begin{equation*}
    \mathbf{U}^*=\mathbf{Z}\mathbf{D}_{w_2},
  \end{equation*}
which implies $Z_{w_1w_2}=U^*_{w_1}$.
Since $w$ is the only sink, we have $0=U_w^*=\min_{v\in V} U_v^*\leq U^*_{w_1}=Z_{w_1w_2}$.
Combining this with \eqref{Zw2x} yields \eqref{ZZ}. That is, we obtain \eqref{gw1} when $w=n$.

We next consider the case where $w\neq n$.  The proof follows a similar approach, though the details differ somewhat.
Since $g_{w_1w}=c_{w_1w}(U_{w_1}-U_w)$, it suffices to show that $U_{w_1}-U_w\geq U_{w_1}'-U_w'$.
By \eqref{KCL-G2} and \eqref{KCL-G-1}, this is equivalent to showing
  \begin{equation*}\label{goal-B}
    Z_{w_1x}-Z_{wx}\geq Z_{w_1x}'-Z_{wx}'.
  \end{equation*}
To achieve this,  let $\mathbf{D}_{uv}=\mathbf{D}_u-\mathbf{D}_v$ for all $u\neq v$.
By the definition of $\{c'_e\}_{e\in E}$, it can be checked that $\mathbf{C}'=\mathbf{C}+c_\triangle \mathbf{D}_{w_2w}\mathbf{D}_{w_2w}^T$.
Using the similar argument for the case where $w=n$, we obtain
\begin{equation*}\label{relation-Z-1}
\mathbf{Z}'=\mathbf{Z}-\widehat{c_\triangle}'\mathbf{Z}\mathbf{D}_{w_2w}\mathbf{D}_{w_2w}^T\mathbf{Z}
\end{equation*}
with $\widehat{c_{\triangle}}':=\frac{c_\triangle}{1+c_\triangle\mathbf{D}_{w_2w}^T\mathbf{Z}\mathbf{D}_{w_2w}}$.
Thus, we only need to show that
  \begin{equation}\label{ZDDZ}
    (\mathbf{Z}\mathbf{D}_{w_2w}\mathbf{D}_{w_2w}^T\mathbf{Z})_{w_1x}-(\mathbf{Z}\mathbf{D}_{w_2w}\mathbf{D}_{w_2w}^T\mathbf{Z})_{wx}
    =(Z_{w_1w_2}-Z_{w_1w}-Z_{ww_2}+Z_{ww})(Z_{w_2x}-Z_{wx})\geq 0.
  \end{equation}
Indeed, according to \eqref{KCL-G2} and the assumption $g_{w_2w}>0$, one has that
\begin{equation}\label{Z-Z}
Z_{w_2x}-Z_{wx}=U_{w_2}-U_{w}=c_{w_2w}g_{w_2w}> 0.
\end{equation}
In addition, consider a unit electric flow from $w_2$ to $w$ with respect to $\{c_e\}_{e\in E}$, and denote $\mathbf{U}^*=(U^*_1,\cdots,  U^*_{n-1})$  as the associated electric potential field with $U^*_n=0$ (note that here we used a slightly different convention as opposed to the usual choice of setting $U^*_w = 0$). Similar to \eqref{KCL-G1} and \eqref{KCL-G2} we also get that
  \begin{equation*}
    \mathbf{U}^*=\mathbf{Z}(\mathbf{D}_{w_2}-\mathbf{D}_{w}).
  \end{equation*}
Combining this with the fact that $w$ is the only sink, we derive
\begin{equation*}
Z_{w_1w_2}-Z_{w_1w}-Z_{ww_2}+Z_{ww}=U_{w_1}^*-U_{w}^*=U_{w_1}^*-\min_{v\in V}U_v^*\geq0.
\end{equation*}
This, along with \eqref{Z-Z}, yields \eqref{ZDDZ}. Hence, the proof is complete.
\end{proof}

\begin{figure}[htbp]
\centering
\includegraphics[scale=1]{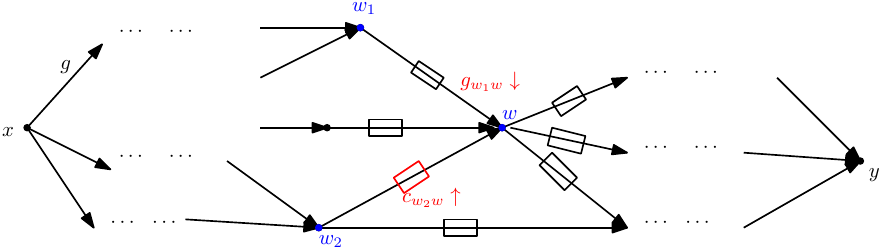}
\caption{The illustration for Lemma \ref{network}. The flow $g$ represents the unit electric flow from $x$ to $y$, satisfying $g_{w_1w},g_{w_2w}>0$. When we increase the conductance $c_{w_2w}$ (the red rectangle) on the edge $\langle w_2,w\rangle$, the amount of flow $g$ through the edge $\langle w_1,w\rangle$, i.e. $g_{w_1w}$, will decrease.}
\label{flow-submult}
\end{figure}

\subsection{Proof of Proposition~\ref{submult}}\label{proof-submult}
In this section, we aim to employ the renormalization technique and Lemma \ref{network} to complete the proof of Proposition~\ref{submult}.
To this end, we will fix sufficiently large $m,n\in \mathds{N}$.

Let $G=(V,E)$ be the renormalization from the $\beta$-LRP model by identifying the intervals $I_i:=[im,(i+1)m)$ to vertices $\varpi(i)$ for all $i\in \mathds{Z}$. That is, $V=\{\varpi(0),\varpi(-1),\varpi(1),\cdots\}$ and the edge set is defined as
\begin{equation*}
E=\left\{\langle \varpi(i),\varpi(j)\rangle:\ i\neq j\text{ and there exists an edge $\langle x,y\rangle \in \mathcal{E}$ such that $x\in I_i$ and $y\in I_j$}\right\}.
\end{equation*}
Note that if there are multiple edges between two intervals $I_i$ and $I_j$ in the edge set $\mathcal{E}$, we only include a single edge $\langle \varpi(i),\varpi(j)\rangle$ in the edge set $E$.
By the self-similarity of the model, it is obvious that $G$ is also the critical $\beta$-LRP on $\mathds{Z}$.
By placing a unit resistance on each edge, $G$ can be viewed as a random electric network. For simplicity, we will also refer to this network as $G$.
Denote $V_n=\{\varpi(0),\cdots, \varpi(n-1)\}$.
For any $i,j\in [0,n)$, let $R^G_{V_n}(\varpi(i),\varpi(j))$ denote the effective resistance in $G$ restricted to $V_n$, that is,
\begin{equation}\label{def-RGV}
\begin{aligned}
  R^G_{V_n}(\varpi(i),\varpi(j))&=\inf\left\{\frac{1}{2}\sum_{\varpi(k)\sim \varpi(l)}g_{\varpi(k)\varpi(l)}^2:\ g \text{ is a unit flow from $\varpi(i)$ to $\varpi(j)$}\right. \\
 & \quad \quad\quad \quad\quad  \text{such that $g_{\varpi(k)\varpi(l)}=0$ for all }\langle \varpi(k),\varpi(l)\rangle \in E\backslash  E_{V_n\times V_n}\Bigg\}.
 \end{aligned}
\end{equation}
It is worth emphasizing that, since $G$ is also a critical $\beta$-LRP model on $\mathds{Z}$, we have
\begin{equation*}
  \mathds{E}\left[R^G_{V_n}(\varpi(i),\varpi(j))\right]=\mathds{E}\left[R_{[0,n)}(i,j)\right]\quad \text{for all }i,j\in [0,n).
\end{equation*}
Therefore,
\begin{equation*}
  \Lambda(n)=\max_{i,j\in[0,n)}\mathbb{E}\left[R^G_{V_n}(\varpi(i),\varpi(j))\right].
\end{equation*}

Next, we fix $a,b\in [0,mn)$. Let $i_a,j_b\in [0,n)$ such that $a\in I_{i_a}$ and $b\in I_{j_b}$. Assume that $g$ is a unit electric flow from $\varpi(i_a)$ to $\varpi(j_b)$ in the graph $G$ restricted to $V_n$.
We want to construct a unit flow $f$ from $a$ to $b$ in the original $\beta$-LRP model according to the flow $g$.
To achieve this,  let
\begin{equation}\label{IG_in}
I^G_{i,\rm in}=\left\{k\in [0,n):\ g_{\varpi(k)\varpi(i)}>0\right\}\quad \text{for } i\in [0,n)\setminus \{i_a\}
\end{equation}
and
\begin{equation}\label{IG_out}
I^G_{i,\rm out}=\left\{l\in[0,n):\ g_{\varpi(i)\varpi(l)}>0\right\}\quad \text{for }i\in [0,n)\setminus \{j_b\}.
\end{equation}
In particular, denote $I^G_{i_a,\rm in}=\{i_a\}$ and $I^G_{j_b,\rm out}=\{j_b\}$.
Additionally, for any $\langle \varpi(i),\varpi(j)\rangle \in E$ with $g_{\varpi(i)\varpi(j)}>0$, denote the directed edge set $$\mathcal{E}_{i\rightarrow j}=\{\langle u,v\rangle\in \mathcal{E}:\ u\in I_i\ \text{and }v\in I_j\}.$$
By the definition of renormalization, it is clear that $\mathcal{E}_{i\rightarrow j}\neq \varnothing$.
For convenience, for each $i\in [0,n)$, denote $\mathcal{E}_i=\mathcal{E}_{I_i\times I_i}$.
We will now construct a unit flow $f$ from $a$ to $b$ in the $\beta$-LRP model as follows (see Figure \ref{g-f} for an illustration).

\begin{itemize}
  \item[(1)] For any $\langle \varpi(i),\varpi(j)\rangle \in E$ with $g_{\varpi(i)\varpi(j)}>0$, choose an edge $\langle x_{i\rightarrow j},y_{i\rightarrow j}\rangle\in \mathcal{E}_{i\rightarrow j}$ such that
  \begin{equation*}
  |x_{i\rightarrow j}-y_{i\rightarrow j}|=\text{arg}\min\left\{|x-y|:\langle x,y\rangle\in  \mathcal{E}_{i\rightarrow j}\right\}.
  \end{equation*}
  Then we define
  \begin{equation*}
    f_{xy}=
  \begin{cases}
  g_{\varpi(i)\varpi(j)},\quad & \langle x,y\rangle=\langle x_{i\rightarrow j},y_{i\rightarrow j}\rangle, \\
  0,\quad &\langle x,y\rangle\in \mathcal{E}_{i\rightarrow j}\setminus \{\langle x_{i\rightarrow j},y_{i\rightarrow j}\rangle\},
  \end{cases}
  \end{equation*}
  and let $f_{yx}=-f_{xy}$ for all $\langle x,y\rangle\in \mathcal{E}_{i\rightarrow j}$.

\medskip

  \item[(2)] In this step, we construct the flow $f$ inside intervals $I_i$ for all $i\in [0,n)$. For that, let $I_{i_a,\rm in}=\{a\}$ and
      \begin{equation*}
      I_{j,\rm in}=\left\{y_{i\rightarrow j}:\ i\in [0,n)\text{ such that }\langle x_{i\rightarrow j},y_{i\rightarrow j}\rangle\in \mathcal{E}_{i\rightarrow j} \right\}
      \end{equation*}
for all $j\in [0,n)\setminus \{i_a\}$.
Analogously, we define
      \begin{equation*}
      I_{j_b,\rm out}=\{b\}\quad \text{and}\quad I_{i,\rm out}=\left\{x_{i\rightarrow j}:\ j\in [0,n)\text{ such that }\langle x_{i\rightarrow j},y_{i\rightarrow j}\rangle\in \mathcal{E}_{i\rightarrow j} \right\}.
      \end{equation*}
      for all $i\in [0,n)\setminus \{j_b\}$. Recall that $I^G_{\cdot,\rm in}$ and $I^G_{\cdot,\rm out}$ are defined in \eqref{IG_in} and \eqref{IG_out}, respectively.

Now for any  $i\in [0,n)$, let $\{\theta_{kl}\}_{k\in I_{i,\rm in},\ l\in I_{i,\rm out}}$ be a sequence of nonnegative numbers such that
    \begin{equation*}
    \sum_{l\in I^G_{i,\rm out}}\theta_{kl}= g_{\varpi(k)\varpi(i)}\quad \text{for all }k\in I^G_{i,\rm in}
    \end{equation*}
and
\begin{equation*}
\sum_{k\in I^G_{i,\rm in}}\theta_{kl}=g_{\varpi(i)\varpi(l)}\quad \text{for all }l\in I^G_{i,\rm out}.
\end{equation*}
For any fixed $k\in I^G_{i,\rm in}$ and $l\in I^G_{i,\rm out}$, denote by $\phi^{(k,l)}$  the unit electric flow from $y_{k\rightarrow i}$ to $x_{i\rightarrow l}$ restricted to the interval $I_i$. Then we define
\begin{equation*}
f_{xy}=\sum_{k\in I^G_{i,\rm in}}\sum_{l\in I^G_{i,\rm out}}\theta_{kl}\phi^{(k,l)}_{xy}
\quad\text{ for all }\langle x,y\rangle\in \mathcal{E}_i.
\end{equation*}
\end{itemize}

\begin{figure}[htbp]
    \centering
    \subfigure{\includegraphics[scale=0.5]{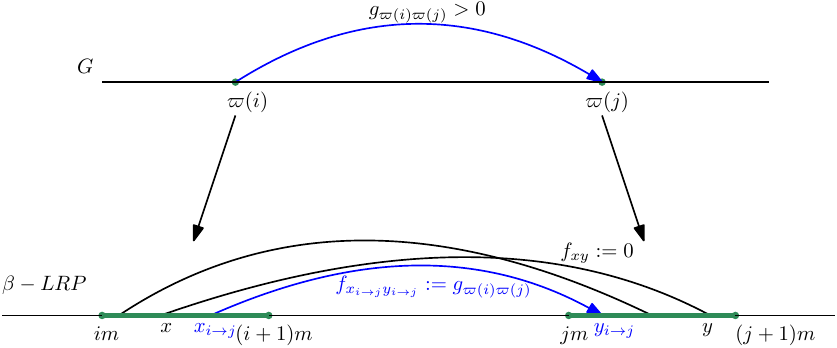}}
    \quad
    \subfigure{\includegraphics[scale=0.6]{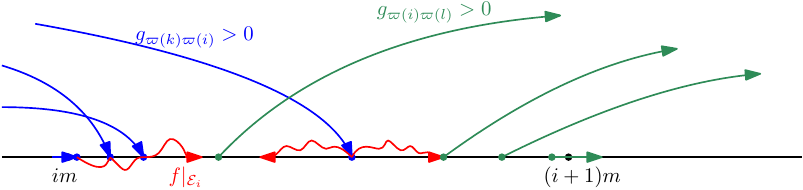}}
    \caption{The illustration for the construction of flow $f$. The flow $g$ represents the unit electric flow from $\varpi(i_a)$ to $\varpi(j_b)$ in the graph $G$. The left figure illustrates step (1), where the blue curve represents the edge $\langle x_{i\rightarrow j},y_{i\rightarrow j} \rangle$. The right figure  corresponds to step (2), where the blue point represents the set $I_{i,\rm in}$, while the dark green points correspond to the set $I_{i,\rm out}$. The red curves indicate the flow $f$ from $I_{i,\rm in}$ to $I_{i,\rm out}$ restricted to the interval $I_i$,  with the amount of $\sum_{k:g_{\varpi(k)\varpi(i)>0}}g_{\varpi(k)\varpi(i)}$.
    }
    \label{g-f}
    \end{figure}

From the above construction, it can be verified that $f$ is a unit flow from $a$ to $b$ restricted in the interval $[0,mn)$ with respect to the $\beta$-LRP model.
Furthermore, on the one hand, the flow $g$ (accordingly, $R^G_{V_n}(\varpi(i_a),\varpi(j_b))$), $I_{i,\rm in}$ and $I_{i,\rm out}$ for all $i\in [0,n)$  all depend only on the edge set
$$
\widehat{\mathcal{E}}:=\mathcal{E}\setminus (\cup_{i\in [0,n)}\mathcal{E}_i).
$$
On the other hand, given $\widehat{\mathcal{E}}$, the energies generated by $\phi^{(k,l)}$ for all $k\in I^G_{i,\rm in}$ and all $l\in I^G_{i,\rm out}$ depend only on the edge set $\mathcal{E}_i$ for all $i\in [0,n)$.
Therefore, by Cauchy-Schwartz inequality and the fact that
\begin{equation*}
\sum_{k\in I^G_{i,\rm in}}\sum_{l\in I^G_{i,\rm out}}\theta_{kl}= \sum_{\varpi(k):g_{\varpi(k)\varpi(i)}>0}g_{\varpi(k)\varpi(i)},
\end{equation*}
we have
  \begin{align}
    &2\mathbb{E}\left[R_{[0,mn)}(a,b)\ |\widehat{\mathcal{E}}\right]
    \leq \mathbb{E}\left[\sum_{x\sim y} f_{xy}^2\ |\widehat{\mathcal{E}}\right]\label{sub-multi-1}\\
    &=\sum_{\varpi(i),\varpi(j)\in V_n:\varpi(i)\sim \varpi(j)}g_{\varpi(i)\varpi(j)}^2
    +\sum_{\varpi(i)\in V_n}\mathds{E}\left[\sum_{\langle x,y\rangle\in\mathcal{E}_i }f_{xy}^2\ |\widehat{\mathcal{E}}\right]\nonumber\\
    &= \sum_{\varpi(i),\varpi(j)\in V_n:\varpi(i)\sim \varpi(j)}g_{\varpi(i)\varpi(j)}^2
    +\sum_{\varpi(i)\in V_n}\mathds{E}\left[\sum_{\langle x,y\rangle\in\mathcal{E}_i }\left(\sum_{k\in I^G_{i,\rm in}}\sum_{l\in I^G_{i,\rm out}}\theta_{kl}\phi^{(k,l)}_{xy}\right)^2\ |\widehat{\mathcal{E}}\right]\nonumber\\
    &\leq \sum_{\varpi(i),\varpi(j)\in V_n:\varpi(i)\sim \varpi(j)}g_{\varpi(i)\varpi(j)}^2
    + \sum_{\varpi(i)\in V_n}\left(\sum_{k\in I^G_{i,\rm in}}\sum_{l\in I^G_{i,\rm out}}\theta_{kl}\right)
    \mathds{E}\left[\sum_{\langle x,y\rangle\in\mathcal{E}_i }\sum_{k\in I^G_{i,\rm in}}\sum_{l\in I^G_{i,\rm out}}\theta_{kl}\left(\phi^{(k,l)}_{xy}\right)^2\ |\widehat{\mathcal{E}}\right]\nonumber\\
    &=\sum_{\varpi(i),\varpi(j)\in V_n:\varpi(i)\sim \varpi(j)}g_{\varpi(i)\varpi(j)}^2\nonumber\\
    &\quad +\sum_{\varpi(i)\in V_n}\left(\sum_{\varpi(k):g_{\varpi(k)\varpi(i)}>0}g_{\varpi(k)\varpi(i)}\right)\sum_{k\in I^G_{i,\rm in}}\sum_{l\in I^G_{i,\rm out}}\theta_{kl}
    \mathds{E}\left[\sum_{\langle x,y\rangle\in\mathcal{E}_i }\left(\phi^{(k,l)}_{xy}\right)^2\ |\widehat{\mathcal{E}}\right]\nonumber\\
        &=\sum_{\varpi(i),\varpi(j)\in V_n:\varpi(i)\sim \varpi(j)}g_{\varpi(i)\varpi(j)}^2\nonumber\\
    &\quad +2\sum_{\varpi(i)\in V_n}\left(\sum_{\varpi(k):g_{\varpi(k)\varpi(i)}>0}g_{\varpi(k)\varpi(i)}\right)\sum_{k\in I^G_{i,\rm in}}\sum_{l\in I^G_{i,\rm out}}\theta_{kl}
    \mathds{E}\left[R_{I_i}(y_{k\rightarrow i},x_{i\rightarrow l})\ |\widehat{\mathcal{E}}\right].\nonumber
  \end{align}

Additionally, according to the independence of edges and the definition of $\Lambda(\cdot)$ in \eqref{def-Lambda}, we obtain that for all $i\in [0,n)$, $y_{k\rightarrow i}\in I_{i,\rm in}$ and $x_{i\rightarrow l}\in I_{i,\rm out}$,
\begin{equation*}\label{sub-multi-2}
  \mathds{E}\left[R_{I_i}(y_{k\rightarrow i},x_{i\rightarrow l})\ |\widehat{\mathcal{E}}\right]
  \leq \max_{x,y\in I_i}\mathbb{E}\left[R_{I_i}(x,y)\right]= \Lambda(m).
\end{equation*}
Substituting this into \eqref{sub-multi-1} gives
\begin{equation}\label{sub-multi-3}
 2\mathbb{E}\left[R_{[0,mn)}(a,b)\ |\widehat{\mathcal{E}}\right]
 \leq \sum_{\varpi(i),\varpi(j)\in V_n:\varpi(i)\sim \varpi(j)}g_{\varpi(i)\varpi(j)}^2
    +2\Lambda(m)\sum_{\varpi(i)\in V_n}\left(\sum_{\varpi(k):g_{\varpi(k)\varpi(i)}>0}g_{\varpi(k)\varpi(i)}\right)^2.
\end{equation}

Next, we will present the following lemma to control the final term on the RHS of \eqref{sub-multi-3}.

\begin{lemma}\label{lem:compare}
  For all $\beta>0$, there exists a constant $C_2=C_2(\beta)<\infty$ {\rm(}depending only on $\beta${\rm)} such that the following holds.
  For any two distinct vertices $\varpi(i),\varpi(j)\in V_n$, let $g$ be a unit electric flow from $\varpi(i)$ to $\varpi(j)$ restricted to $V_n$.
  Then for all $\varpi(k)\in V_n$,
  \begin{equation}\label{eq:compare}
    \mathbb{E}\left[\left(\sum_{\varpi(l):g_{\varpi(l)\varpi(k)}>0}g_{\varpi(l)\varpi(k)}\right)^2\right]\leq C_2\mathbb{E}\left[\sum_{\varpi(l):g_{\varpi(l)\varpi(k)}>0}g_{\varpi(l)\varpi(k)}^2\right].
  \end{equation}
\end{lemma}

\begin{proof}
  It is evident that we only need to consider the graph $G$ restricted to $V_n$, which we denote as $G_n=(V_n,E_n)$.
  Since neighboring points are connected by edges in our model, $G_n$ is clearly a finite connected graph.
  Clearly, there are a finite number of configurations for $G_n$, each with a positive probability.
  For each configuration $\omega$, let $p(\omega)$ represent its probability, and $E_n(\omega)$ denote its edge set.

  For each $\varpi(k)\in V_n$, define
  $$
  J_{k,\rm in}=\left\{\varpi(l)\in V_n:\ g_{\varpi(l)\varpi(k)}>0\right\}.
   $$
   Additionally, for each $s\in\mathds{N}$, denote $\Omega_s(k)$ as the set of all configurations for which $\# J_{k,\rm in}(\omega)=s$.
 By applying Cauchy-Schwartz inequality, it is sufficient to find a constant $C(\beta)>0$ depending only on $\beta$ such that for  all $\varpi(k)\in V_n$,
  \begin{equation}\label{eq:compare2}
    \sum_{s\geq 1}s^2\mathbb{E}\left[\sum_{\varpi(l)\in J_{k,\rm in}}g_{\varpi(l)\varpi(k)}^2\I_{\{\omega\in\Omega_s(k)\}}\right]\leq C(\beta)\mathbb{E}\left[\sum_{\varpi(l):g_{\varpi(l)\varpi(k)}>0}g_{\varpi(l)\varpi(k)}^2\right].
  \end{equation}

To achieve this, for any $\varpi(l)\in J_{k,\rm in}$ and any $s\in \mathds{N}$, consider a configuration $\omega\in \Omega_s(k)$ such that $\varpi(l)\in J_{k,\rm in}$.
We then define $\omega_{k,l}^{-}$ as the configuration obtained from $\omega$ by removing all edges (or equivalently, setting the conductances of those edges to 0) that connect $\varpi(k)$ to $J_{k,\rm in}\setminus \{\varpi(l),\varpi(k-1),\varpi(k+1)\}$ directly.
By applying Lemma \ref{network} iteratively, with $w=\varpi(k)$, $w_1=\varpi(l)$ and  $w_2$ taken over vertices in $J_{k,\rm in}\setminus \{\varpi(l),\varpi(k-1),\varpi(k+1)\}$, we arrive at
  \begin{equation}\label{ww-}
    0<g_{\varpi(l)\varpi(k)}(\omega)\leq g_{\varpi(l)\varpi(k)}(\omega^-_{k,l}).
  \end{equation}
  As a result, we can deduce that
  \begin{equation}\label{eq:compare3}
    \begin{aligned}
      \text{LHS of \eqref{eq:compare2}}&=\sum_{s\geq 1}\sum_{\omega\in\Omega_s(k)} \sum_{\varpi(l)\neq \varpi(k)} s^2g_{\varpi(l)\varpi(k)}(\omega)^2\I_{\{\varpi(l)\in J_{k,\rm in}(\omega)\}}p(\omega)\\
      &\overset{\text{\eqref{ww-}}}{\leq}  \sum_{\varpi(l)\neq \varpi(k)}\sum_{s\geq 1}\sum_{\omega\in\Omega_s(k)} s^2g_{\varpi(l)\varpi(k)}(\omega_{k,l}^-)^2\I_{\{\varpi(l)\in J_{k,\rm in}(\omega_{k,l}^-)\}}p(\omega)\\
      &\leq \sum_{\varpi(l)\neq \varpi(k)}\sum_{\omega'}\sum_{s\geq 1}\sum_{\omega:\omega_{k,l}^-=\omega'}s^2g_{\varpi(l)\varpi(k)}(\omega')^2\I_{\{\varpi(l)\in J_{k,\rm in}(\omega')\}}p(\omega)\\
      &\leq \sum_{\varpi(l)\neq \varpi(k)}\sum_{\omega'}g_{\varpi(l)\varpi(k)}(\omega')^2\I_{\{\varpi(l)\in J_{k,\rm in}(\omega')\}}\left(\sum_{s\geq 1}\sum_{\omega:\omega_{k,l}^-=\omega'}s^2p(\omega)\right).
    \end{aligned}
  \end{equation}
  In the second and third inequality, we proceeded by fixing any possible $\omega'$ and summing over all $\omega\in\Omega_s(k)$ such that $\omega_{k,l}^-=\omega'$.

To establish an upper bound for the sum in the bracket on the RHS of \eqref{eq:compare3}, we will focus on the case where $\varpi(l)\notin \{\varpi(k-1),\varpi(k+1)\}$. Similar arguments can be applied when $\varpi(l)\in \{\varpi(k-1),\varpi(k+1)\}$.
For $s\geq 1$, consider any two configurations $\omega\in \Omega_s(k)$ and $\omega'$, such that $\omega^-_{k,l}=\omega'$.

It is clear that the edges in $\mathcal{E}(\omega')\subset \mathcal{E}(\omega)$ and the edges in $\mathcal{E}(\omega)\setminus \mathcal{E}(\omega')$ share the common endpoint $\varpi(k)$.
Denote $\xi$ as the number of vertices in $V_n\setminus J_{k,\rm in}(\omega')=V_n\setminus\{\varpi(l),\varpi(k-1),\varpi(k+1)\}$ that are  directly  connected to $\varpi(k)$. Then by the independence of edges in our model, we have
\begin{equation*}
  \begin{aligned}
    p(\omega')&=\mathds{P}\left[\varpi(u)\sim \varpi(v)\text{ for all }\langle \varpi(u),\varpi(v)\rangle \in \mathcal{E}(\omega')\right]\mathds{P}[\xi=0]\\
    &\qquad\cdot \mathbb{P}[\varpi(u)\nsim \varpi(v)\text{ for all }\langle \varpi(u),\varpi(v)\rangle \in \mathcal{E}\setminus\mathcal{E}(\omega')\text{ with $\varpi(u),\varpi(v)\neq \varpi(k)$}].
  \end{aligned}
\end{equation*}
Consequently, we find that
  \begin{equation}\label{eq:compare4}
    \sum_{s\geq 1}\sum_{\omega\in\Omega_s(k):\omega_{k,l}^-=\omega'}s^2p(\omega)\leq \frac{p(\omega')}{\mathds{P}[\xi=0]}\sum_{s\geq 1 }s^2\mathbb{P}[\xi=\mathfrak{h}(s)],
  \end{equation}
where
\begin{equation*}
\mathfrak{h}(s)=
\begin{cases}
s-1,\quad & \varpi(k-1),\varpi(k+1)\notin J_{k,\rm in}(\omega'),\\
s-2,\quad & \#\left(\{\varpi(k-1),\varpi(k+1)\}\cap J_{k,\rm in}(\omega')\right)=1,\\
s-3,\quad & \{\varpi(k-1),\varpi(k+1)\}\subset J_{k,\rm in}(\omega').
\end{cases}
\end{equation*}

In addition, it can be observed that $\xi$ can be viewed as the sum of a sequence of independent Bernoulli variables $\{X_i\}_{i=1}^{n-3}$ with parameters $\{q_i\}_{i=1}^{n-3}$ satisfying
$$
\max_i q_i\leq \max_{i\geq 2}p_{0,i}=p_{0,2}\quad  \text{and}\quad \sum_{i=1}^{n-3}q_i\leq 2\sum_{i=2}^{\infty}p_{0,i}=:\lambda_0(\beta)<\infty.
$$
Hence, there exists a constant $C(\beta)>0$ (depending only on $\beta$) such that for all sufficiently large $n\in \mathds{N}$,
  \begin{equation*}\label{eq:N1}
  \mathbb{P}[\xi=0]= \prod_{i=1}^{n-3} (1-q_i)\geq \prod_{i=2}^{n-3}(1-p_{0,i})\geq C(\beta)^{-1/2},
  \end{equation*}
and
    \begin{equation*}\label{eq:N2}
    \begin{aligned}
    \sum_{s\geq 1}s^2\mathbb{P}[\xi=\mathfrak{h}(s)]\leq \mathbb{E}\left[(\xi+3)^2\right]
    &= \mathrm{Var}(\xi)+\left(\mathds{E}[\xi]+3\right)^2\\
    &\leq \sum_{i=1}^{n-3}q_i+ \left(\sum_{i=1}^{n-3}q_i+3\right)^2 \leq C(\beta)^{1/2}.
  \end{aligned}
  \end{equation*}
  As a result, we have
  \begin{equation*}
    \sum_{s\geq 1}s^2\frac{\mathbb{P}[\xi=\mathfrak{h}(s)]}{\mathbb{P}[\xi=0]}\leq C(\beta).
  \end{equation*}
  Applying this to \eqref{eq:compare4} yields that for any configuration $\omega'$,
  \begin{equation*}\label{eq:compare5}
    \sum_{s\in\mathds{N}}\sum_{\omega\in\Omega_s(k):\omega_{k,l}^-=\omega'}s^2p(\omega)\leq C(\beta)p(\omega').
  \end{equation*}
  Combining this with \eqref{eq:compare3}, we obtain that
  \begin{equation*}
    \begin{aligned}
      \text{LHS of \eqref{eq:compare2}}&\leq C(\beta)\sum_{\varpi(l)\neq \varpi(k)}\sum_{\omega'}g_{\varpi(l),\varpi(k)}(\omega')^2\I_{\{\varpi(l)\in J_{k,\rm in}(\omega')\}}p(\omega')\leq \text{RHS of \eqref{eq:compare2}},
    \end{aligned}
  \end{equation*}
  which completes the proof.
\end{proof}

With Lemma~\ref{lem:compare} at hand,  we can present the

\begin{proof}[Proof of Proposition~\ref{submult}]
  Taking expectations on both sides of \eqref{sub-multi-3} and combining it with \eqref{eq:compare} for all $\varpi(k)\in V_n$, we obtain
  \begin{equation}\label{uper-R(a,b)}
    \begin{aligned}
      2\mathbb{E}\left[R_{[0,mn)}(a,b)\right]&\leq 2\mathbb{E}\left[R^G_{V_n}R(\varpi(i_a),\varpi(j_b))\right]
      +2C_2(\beta)\Lambda(m)\mathbb{E}\left[\sum_{\varpi(i)\in V_n}\sum_{\varpi(k):g_{\varpi(k)\varpi(i)}>0}g_{\varpi(k)\varpi(i)}^2\right]\\
      &\leq 2C_2(\beta)\Lambda(m)(2C_2(\beta)^{-1}+\Lambda(n)),
    \end{aligned}
  \end{equation}
  where $C_2<\infty$ (depending only on $\beta$) is the constant defined in Lemma~\ref{lem:compare}.
  In addition, from \cite[Theorem 1.1]{DFH24+} we find that $\lim_{n\to\infty}\Lambda(n)=\infty$. Combining this with \eqref{uper-R(a,b)} completes the proof.
 \end{proof}

\section{Resistances of far away points}\label{sect-far}
For $n\in \mathds{N}$ and $i,j\in [0,n)$, recall that the resistance between $i$ and $j$ within $[0,n)$ is defined by \eqref{def-Rinsd}.

The aim of this section is to establish the following result, which shows that the effective resistance between points that are relatively close to each other can be controlled by the effective resistance between two points that are far away.

\begin{proposition}\label{R(0,n)main}
For all $\beta>0$, there exists a constant $C'_{1,*}=C'_{1,*}(\beta)<\infty$ {\rm(}depending only on $\beta${\rm)} such that for all  $n\in \mathds{N}$ and all $i,j\in [0,n)$,
\begin{equation*}
\mathds{E}\left[R_{[0,n)}(i,j)\right]\leq C'_{1,*}\mathds{E}\left[R_{[0,n)}(0,n-1)\right].
\end{equation*}
Therefore,
\begin{equation*}
\mathds{E}\left[R_{[0,n)}(0,n-1)\right]\leq \Lambda(n)\leq C'_{1,*}\mathds{E}\left[R_{[0,n)}(0,n-1)\right].
\end{equation*}
\end{proposition}

The main input for proving Proposition \ref{R(0,n)main} is to establish that the effective resistances on a small-scale graph can be controlled by the corresponding effective resistances on a large-scale graph, as outlined below.
In fact, we just present a concise version here; however, in the proof of Proposition  \ref{R(0,n)main}, we will require a slightly stronger version (see Lemma \ref{R-renorm} below).

\begin{proposition}\label{propR-renorm}
For $\beta>0$, there exists a constant $C_2<\infty$ {\rm(}depending only on $\beta${\rm)} such that the following holds. For all $m,n\in \mathds{N}$ and all $i,j\in [0,mn)$, denote $i'=\lfloor i/m\rfloor$ and $j'=\lfloor j/m\rfloor$. Then
\begin{equation}\label{Rn-Rmn}
\mathds{E}\left[R_{[0,n)}(i',j')\right]\leq C_2\mathds{E}\left[R_{[0,mn)}(i,j)\right].
\end{equation}
\end{proposition}

Most of this section is dedicated to the proof of Proposition \ref{propR-renorm}. The primary method for this proof involves renormalization and a variational characterization of effective resistance in \cite{BDG20}.  The following lemma provides the necessary lower bound estimates for resistances, which represents a slightly stronger version of \cite[Theorem 1.1]{DFH24+}.

\begin{lemma}\label{Lem-RNtail-1}
For all $\beta>0$, there exist constants $0<c_*<C_*<\infty$ and $\delta^*=\delta^*(\beta)>0$ {\rm(}all depending only on $\beta${\rm)} such that the following holds. For any $\varepsilon\in (0,1/4)$, $\delta\in (0,\delta^*]$ and any $N\in \mathds{N}$,
\begin{equation}\label{est-RNhat}
\mathds{P}\left[\widehat{R}_N\geq c_*\varepsilon^{C_*\delta}N^{\delta}\right]\geq 1-\varepsilon,\quad  \mathds{P}\left[R(0,[-N,N]^c)\geq c_*\varepsilon^{C_*\delta}N^{\delta}\right]\geq 1-\varepsilon,
\end{equation}
and
\begin{equation*}\label{R-box}
\mathds{P}\left[R([-N,N],[-2N,2N]^c)\geq c_*\varepsilon^{C_*\delta}N^{\delta}\ | \mathcal{E}_{[-N,N]\times [-2N,2N]^c}=\emptyset\right]\geq 1-\varepsilon,
\end{equation*}
where
\begin{equation*}
\begin{aligned}
\widehat{R}_N&:=\inf\left\{\frac{1}{2}\sum_{u\sim v}f_{uv}^2:\ f\text{ is a unit flow from $(-\infty,0]$ to $(N,+\infty)$}\right.\\
&\quad \quad \quad \quad \quad \quad \quad \quad\quad \quad \text{and $f_{uv}=0$ for all }\langle u,v\rangle \in \mathcal{E}_{(-\infty,0]\times (N,+\infty)}\Bigg\}.
\end{aligned}
\end{equation*}
\end{lemma}

\begin{proof}
Fix $\varepsilon \in (0,1/4)$.
We begin by considering the case where $N=2^n$ for some $n\in \mathds{N}$.
From \cite[Proof of Theorem 3.1]{DFH24+}, we claim that there exist constants $\widetilde{c}_1=\widetilde{c}_1(\beta)>0$, $\delta^*=\delta^*(\beta)>0$ and $L=L(\beta)>0$ (all depending only on $\beta$) such that
\begin{equation}\label{DFH24-Thm3.1}
\mathds{P}\left[\widehat{R}_N\geq \widetilde{c}_1M^{-1}\e^{\delta^*(n-2(L+1)M)}\right]\geq 1-\varepsilon,
\end{equation}
where $M=\widetilde{C}\log(1/\varepsilon)$ for some constant $\widetilde{C}=\widetilde{C}(\beta)<\infty$ depending only on $\beta$.
In order to see this, from \cite[(3.25)]{DFH24+}, we find that the parameter $M_0$ in \cite[Proof of Theorem 3.1]{DFH24+} can be taken as $\widetilde{C}_1\log(1/\varepsilon)$ for some constant $\widetilde{C}_1=\widetilde{C}_1(\beta)<\infty$.
Additionally, according to \cite[(3.26)]{DFH24+} and \cite[Proof of Lemma 3.12]{DFH24+}, the parameter $M_1$ in \cite[Proof of Theorem 3.1]{DFH24+} can similarly be taken as $\widetilde{C}_2\log(1/\varepsilon)$ for some constant $\widetilde{C}_2=\widetilde{C}_2(\beta)<\infty$.
Thus, we can choose $M=\max\{M_0,M_1\}=\max\{\widetilde{C}_1,\widetilde{C}_2\}\log(1/\varepsilon):=\widetilde{C}\log(1/\varepsilon)$ in \cite[Proof of Theorem 3.1]{DFH24+}, allowing us to establish  \eqref{DFH24-Thm3.1}.

We now let $C_*=4(L+1)\widetilde{C}<\infty$ (depending only on $\beta$). Then from \eqref{DFH24-Thm3.1} we can see that there exists a constant $\widetilde{c}_2>0$ (depending only on $\beta$) such that for any $\delta\in (0,\delta^*]$,
\begin{equation*}
\mathds{P}\left[\widehat{R}_N\geq \widetilde{c}_2\varepsilon^{C_*\delta}N^{\delta}\right]\geq 1-\varepsilon,
\end{equation*}
which implies the desired statement for $\widehat{R}_N$ when $N=2^n$. Using the similar arguments as in the last paragraph in \cite[Proof of Theorem 3.1]{DFH24+}, we can extend the  desired result for $\widehat{R}_N$ to general $N\geq 1$.

Finally, the results concerning $R(0,[-N,N]^c)$ and $R([-N,N],[-2N,2N]^c)$ can be obtained by combining the above estimate for $\widehat{R}_N$ with \cite[(4.7)]{DFH24+} and \cite[Lemma 4.3]{DFH24+}, respectively.
\end{proof}

For the exponent $\delta^*(\beta)$ in Lemma \ref{Lem-RNtail-1}, we add the following remark.
\begin{remark}\label{rem-delta*}
(1) From \eqref{LRP}, we observe that the edge set $\mathcal{E}$ of the $\beta$-LRP model is nondecreasing in distribution with respect with $\beta$. Combining this with the monotonicity of the effective resistance in relation to the edges, we can see that the exponent $\delta^*(\beta)$ in Lemma \ref{Lem-RNtail-1} can be chosen as a decreasing function of $\beta$.

(2) It is worth emphasizing that Lemma \ref{Lem-RNtail-1} implies that there exists a constant $c>0$ (depending only on $\beta$) such that for all $N\in \mathds{N}$,
\begin{equation*}
\mathds{E}\left[\widehat{R}_N\right]\geq cN^{\delta^*},\quad \mathds{E}\left[R(0,[-N,N]^c)\right]\geq cN^{\delta^*},
\end{equation*}
and
\begin{equation*}
\mathds{E}\left[R([-N,N],[-2N,2N]^c)\ | \mathcal{E}_{[-N,N]\times [-2N,2N]^c}=\emptyset\right]\geq cN^{\delta^*}.
\end{equation*}
\end{remark}

In the remaining part of this section, we will perform preliminary work in Section \ref{est-good}.
Then, in Section \ref{sect-vf} we will review a variational characterization of effective resistance from \cite[Proposition 2.3]{BDG20}. Additionally, Sections \ref{proof-weaksup} and \ref{sect-R0n} are devoted to the proofs of Propositions \ref{propR-renorm} and \ref{R(0,n)main}, respectively.

\subsection{Good intervals and associated estimates} \label{est-good}

Throughout this subsection, we will fix $m,n\in \mathds{N}$ to be sufficiently large.
Recall that $\delta^*>0$ is the exponent from Lemma \ref{Lem-RNtail-1}.
For convenience, we denote the interval $[im,(i+1)m)$ in $\mathds{Z}$ as $I_i$, and then define
\begin{equation}\label{def-Ki-1}
\mathcal{K}_i=\left\{u\in I_i:\ \text{there exists $v\in [(i-1)m,(i+2)m)^c$ such that }\langle u,v\rangle\in \mathcal{E}\right\}.
\end{equation}
We will also let $R_{I_i}(\cdot,\cdot)$ represent the effective resistance within the interval $I_i$ as defined in \eqref{def-Rinsd}.

For any two intervals $J_1=[x_1,x_2],J_2=[x_3,x_4]\subset \mathds{R}$ with $-\infty\leq x_1<x_2<x_3<x_4\leq +\infty$,
we define $\widehat{R}(J_1,J_2)$ as the effective resistance generated by unit flows (confined to $[x_1,x_4]$) passing through the interval $[x_2,x_3]$, that is,
\begin{equation}\label{def-Rhat}
\begin{aligned}
\widehat{R}(J_1,J_2)&=\inf\left\{\frac{1}{2}\sum_{u\sim v}f^2_{uv}:\ f\ \text{is a unit flow from $J_1$ to }J_2\right.\\
&\quad\quad \quad \quad \quad \quad \quad\quad \text{and $f_{uv}=0$ for all }\langle u,v\rangle\in \mathcal{E}_{(-\infty,x_2]\times [x_3,+\infty)}\cup \mathcal{E}_{[x_1,x_4]^c\times(x_2,x_3)}\Bigg\}.
\end{aligned}
\end{equation}


We now define very good intervals, which states that the energy generated by any unit flow within the interval is not too low.

\begin{definition}\label{def-alphagood-1}
For $i\in \mathds{Z}$, $\delta\in (0,\delta^*]$ and $\alpha=(\alpha_1,\alpha_2)\in (0,1)^2$, we say the interval $I_i$ is \textit{$(\delta,\alpha)$-very good}  if it satisfies the following conditions.
\begin{itemize}
\item[(1)]For any two different edges $\langle u_1,v_1\rangle,\langle u_2,v_2\rangle\in \mathcal{E}$ with $v_1,u_2\in I_i$, $u_1\in I_i^c$ and $v_2\in [(i-1)m,(i+2)m)^c$, we have $|v_1-u_2|\geq \alpha_1 m$.
  \medskip

\item[(2)] The internal (optimal) energy
\begin{equation}\label{def-RIi}
R_{I_i}:=\inf_{\theta}\left(\sum_{u\in \mathcal{K}_i}\theta_u^2R_{I_i}\left(u,B_{\alpha_1m}(u)^c\right)+\left(1-\sum_{u\in \mathcal{K}_i}\theta_u\right)^2\widehat{R}(I_{i-1},I_{i+1})\right) \geq \alpha_2 m^{\delta},
\end{equation}
where the infimum is taken over $\theta=\{\theta_u\}_{u\in \mathcal{K}_i}$ with $\theta_u\geq 0$ and $\sum_{u\in \mathcal{K}_i}\theta_u\leq 1$ (see Figure \ref{figure-RIi} for an illustration).
\end{itemize}
\end{definition}

\begin{figure}[htbp]
\centering
\includegraphics[scale=0.7]{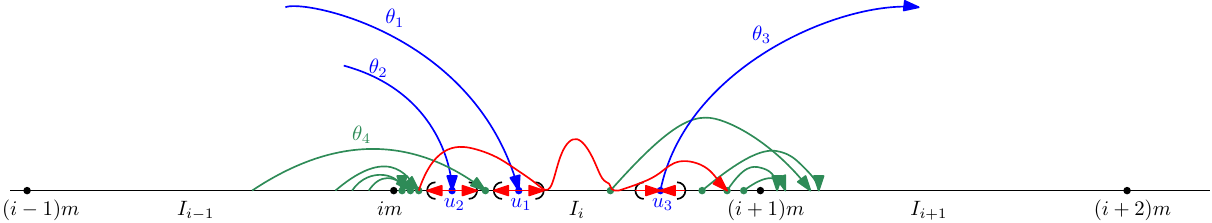}
\caption{The illustration for definition of $R_{I_i}$ in \eqref{def-RIi}. The blue curves represent the long edges in $\mathcal{E}_{I_i\times [(i-1)m,(i+2)m)^c}$, while the green curves represent the edges in $\mathcal{E}_{I_i\times (I_{i-1}\cup I_{i+1})}$. When a unit flow $\theta$ enters the interval $I_i$ and then flows out, there are two possible scenarios: either some subflows of $\theta$ enter through the blue long edges (as $\theta_1$ and $\theta_2$) or exit through them (as $\theta_3$), or some subflows of $\theta$ enter and exit through the green edges (as $\theta_4$). When the flow enters or exits through the blue long edges, it must generate energies $\theta^2_kR(u_k,B_{\alpha_1 m}(u_k)^c)$, $k=1,2,3$ in the vicinity of the endpoints of the long edges (red lines). Additionally, when the flow enters and exits through the green edges, it must generate energy in the interval $I_i$ given by $\theta^2_4\widehat{R}(I_{i-1},I_{i+1})$, as depicted by the red curves.
}
\label{figure-RIi}
\end{figure}

The aim of this section is to establish a lower bound on the probability for an interval to be very good.

\begin{proposition}\label{prop-alphagood-1}
For all $\beta>0$, $\delta\in (0,\delta^*]$ and $\varepsilon\in (0,1)$, there exists $\alpha=(\alpha_1,\alpha_2)\in (0,1)^2$ {\rm(} $\alpha_1$ depends only on $\beta$ and $\varepsilon$, while $\alpha_2$ depends only on $\beta$, $\delta$ and $\varepsilon${\rm)} such that for all $i\in \mathds{Z}$,
\begin{equation*}
\mathds{P}\left[I_i\text{ is $(\delta,\alpha)$-very good}\right]\geq 1-\varepsilon.
\end{equation*}
\end{proposition}

Before moving on to the proof, we would like to discuss the selection of the parameters $(\alpha_1,\alpha_2)$ in Proposition \ref{prop-alphagood-1}.

\begin{remark}\label{rem-alpha}
From the proof of Proposition \ref{prop-alphagood-1} below, it is evident that for $\delta\in (0,\delta^*]$ and sufficiently small $\varepsilon\in(0,1)$, the parameter $\alpha=(\alpha_1, \alpha_2)$ can be chosen as $\alpha_1=\varepsilon^{C_3}$ and $\alpha_2=\varepsilon^{C_4\delta}$ for some constants $C_3,\ C_4<\infty$ depending only on $\beta$, see \eqref{P-E1-1} and \eqref{def-alpha2} for more details.
\end{remark}

To prove Proposition \ref{prop-alphagood-1}, we will define certain events that are stronger versions of Definition \ref{def-alphagood-1} (see Definition \ref{def-verygood-1} below).
Next, we will establish some basic facts about these events before concluding the proof of Proposition \ref{prop-alphagood-1}.

The first definition is for the number of points in some interval $I_i$ which are incident to long edges. For that, recall that $\mathcal{K}_i$ is defined in \eqref{def-Ki-1}.

\begin{definition}\label{def-good-1}
For $i\in \mathds{Z}$ and $M\in \mathds{N}$, we say the interval $I_i$ is \textit{$M$-good} if
\begin{equation}\label{def-xi-1}
\xi_i:=\#\mathcal{K}_i\leq M.
\end{equation}
\end{definition}

We have the following estimate for the probability of an interval being $M$-good.

\begin{lemma}\label{lem-good-1}
For all $\beta>0$, there exists a constant $c_1=c_1(\beta)>0$ {\rm(}depending only on $\beta${\rm)} such that for all $i\in\mathds{Z}$ and sufficiently large $M\in \mathds{N}$,
$$
\mathds{P}[I_i\text{ is $M$-good}]\geq 1-{\rm e}^{-c_1M}.
$$
\end{lemma}
\begin{proof}
Due to the translation invariance of the LRP model, we only need to consider the case where $i=0$. By a simple calculation, we can see that there exists a constant $\widetilde{C}=\widetilde{C}(\beta)<\infty$ (depending only on $\beta$) such that
\begin{equation*}
\begin{aligned}
\mathds{E}[\xi_0]&=\sum_{u\in I_0}\sum_{v\in [-m,2m)^c}\mathds{P}[u\sim v]\\
&=\sum_{u\in I_0}\sum_{v\in [-m,2m)^c}\left[1-\exp\left\{-\int_{u}^{u+1}\int_v^{v+1}\frac{\beta}{|u-v|^2}\d u\d v\right\}\right] \leq \widetilde{C}.
\end{aligned}
\end{equation*}
Combining this with the Chernoff's bound, we obtain that for sufficiently large $M>\widetilde{C}$, there exists a constant $\widetilde{c}=\widetilde{c}(\beta)>0$ (depending only on $\beta$) such that
\begin{equation*}
\mathds{P}[\xi_0>M]\leq \e^{-\widetilde{c}M},
\end{equation*}
which implies the desired result.
\end{proof}

We next introduce a stronger version of very good interval defined in Definition \ref{def-alphagood-1} as follows.

\begin{definition}\label{def-verygood-1}
For $i\in \mathds{Z}$, $M\in\mathds{N}$, $\delta \in (0,\delta^*]$ and $\alpha=(\alpha_1,\alpha_2)\in (0,1)^2$, we say the interval $I_i$ is \textit{$(M,\delta,\alpha)$-very good} if it is $M$-good and also satisfies the following conditions.
\begin{itemize}

\item[(1)] For any two different edges $\langle u_1,v_1\rangle,\langle u_2,v_2\rangle\in \mathcal{E}$ with $v_1,u_2\in I_i$, $u_1\in I_i^c$ and $v_2\in [(i-1)m,(i+2)m)^c$, we have $|v_1-u_2|\geq \alpha_1 m$.
  \medskip

\item[(2)] For any $\langle u,v\rangle\in \mathcal{E}$ with $u\in I_i$ and $v\in [(i-1)m,(i+2)m)^c$, we have
$$
R_{I_i}\left(u,B_{\alpha_1m}(u)^c\right)\geq \alpha_2 m^{\delta}.
$$

\item[(3)] $\widehat{R}(I_{i-1},I_{i+1})\geq \alpha_2m^{\delta}$.
\end{itemize}
\end{definition}

 \begin{remark}\label{M-vg}
Note that if the interval $I_i$ is $(M,\delta,\alpha)$-very good, then it must be a $(\delta,\alpha')$-very good interval, where $\alpha'=(\alpha_1,\alpha_2(M+1)^{-1})$. To see this, by the definitions of $R_{I_i}$ in \eqref{def-RIi} and $M$-good interval in Definition \ref{def-good-1}, we have
\begin{equation*}
\begin{aligned}
R_{I_i}&=\inf_{\theta}\left(\sum_{u\in \mathcal{K}_i}\theta_u^2R_{I_i}\left(u,B_{\alpha_1m}(u)^c\right)+\left(1-\sum_{u\in \mathcal{K}_i}\theta_u\right)^2\widehat{R}(I_{i-1},I_{i+1})\right)\\
&\geq \alpha_2 m^{\delta}\inf_{\theta}\left(\sum_{u\in \mathcal{K}_i}\theta_u^2+\left(1-\sum_{u\in \mathcal{K}_i}\theta_u\right)^2\right)
\geq  \alpha_2(M+1)^{-1} m^{\delta},
\end{aligned}
\end{equation*}
where the infimum is taken over $\theta=\{\theta_u\}_{u\in \mathcal{K}_i}$ with $\theta_u\geq 0$ and $\sum_{u\in \mathcal{K}_i}\theta_u\leq 1$, and in the last inequality we used $\#\mathcal{K}_i\leq M$ (i.e., $I_i$ is $M$-good, see Definition \ref{def-good-1}). This implies Definition \ref{def-alphagood-1} (2) holds for $\alpha'$.
\end{remark}

In the following, we provide a lower bound for the probability of being $(M,\delta,\alpha)$-very good.

\begin{lemma}\label{P-goodint-1}
For all $\beta>0$, $\delta\in (0,\delta^*]$ and sufficiently large $M\in\mathds{N}$, there exists a constant $c_2=c_2(\beta)>0$ {\rm(}depending only on $\beta${\rm)} such that the following holds. There exist constants $\alpha=(\alpha_1,\alpha_2)\in (0,1)^2$ {\rm(}$\alpha_1$ depends only on $\beta$ and $M$, while $\alpha_2$ depends only on $\beta$, $\delta$ and $M${\rm)} such that for all $i\in \mathds{Z}$,
$$
\mathds{P}\left[I_i\ \text{is $(M,\delta,\alpha)$-very good}\ |I_i \text{ is $M$-good}\right]\geq 1-{\rm e}^{-c_2M}.
$$
\end{lemma}

Due to the translation invariance of the LRP model, we only need to prove Lemma \ref{P-goodint-1} with $i=0$. To do this, denote by $F_{k}(i)$ the events in Definition \ref{def-verygood-1} ($k$) for $k=1,2,3$, respectively. For simplicity, we will write $F_k(0)$ as $F_k$.

\begin{lemma}\label{P-E11}
For all $\beta>0$ and sufficiently large $M\in\mathds{N}$, there exist constants $c_3=c_3(\beta)>0$ {\rm(}depending only on $\beta${\rm)} and $\alpha_1\in (0,1)$ {\rm(}depending only on $\beta$ and $M${\rm)} such that
$$
\mathds{P}[F_1^c\ |I_0 \text{ is $M$-good}]\leq {\rm e}^{-c_3M}.
$$
\end{lemma}

\begin{proof}
Note that both the events $F_1$ and $\{I_0 \text{ is $M$-good}\}$ are decreasing with respect to edges. Hence, by the FKG inequality we obtain
\begin{equation*}\label{PE1|EM}
\mathds{P}[F_1^c\ |I_0 \text{ is $M$-good}]\leq \mathds{P}[F_1^c].
\end{equation*}
Thus, it suffices to bound $\mathds{P}[F_1^c]$ for the remainder of the proof.

For convenience, assume that $1/(2\alpha_1)\in \mathds{N}$ and is divisible with respect to $m$. Otherwise, we can replace $1/(2\alpha_1)$ with $\lfloor1/(2\alpha_1)\rfloor$ in the following proof.
Note that indivisibility will only introduce a small constant difference in the estimates in the proof, and thus will only make an immaterial difference. 
We now divide the interval $I_0=[0,m)$ into $1/\alpha_1$ small intervals of equal length, denoting them as $J_k$ for $k\in [0,1/\alpha_1-1]$.
We also define $\widetilde{J}_k$ as the union of $J_k$ and all intervals adjacent to $J_k$.
For $k\in [0,1/\alpha_1-1]$, we let $A_k$ be the event that there exist two edges $\langle u_1,v_1\rangle, \langle u_2,v_2\rangle\in \mathcal{E}$ such that $u_1\in [0,m)^c$, $v_1\in J_k$, $u_2\in \widetilde{J_k}$ and $v_2\in [-m,2m)^c$.
Then it is clear that
\begin{equation}\label{E01-A-1}
F_1^c\subset \bigcup_{k=0}^{1/\alpha_1-1} A_k.
\end{equation}
Hence we turn to bound $\mathds{P}[A_k]$ for each $k$.

Indeed, by the definition of $J_k$, on has $\text{dist}(J_k,I_0^c)\geq k\alpha_1 m$ and $\text{dist}(\widetilde{J}_k,[-m,2m)^c)\geq m$ for all $k\in [0,1/\alpha_1-1]$. Here $\text{dist}(I,J):=\min_{u\in I,v\in J}|u-v|$ for $I,J\subset \mathds{Z}$.
Combining this with the independence of edges in our model, we get that for each $k\in [0,1/(2\alpha_1)]$,
\begin{equation}\label{P-Ak}
\begin{aligned}
\mathds{P}[A_k]&=\left[1-\exp\left\{-\int_{[0,m)^c}\int_{J_k}\frac{\beta}{|x-y|^2}\d x\d y\right\}\right]\cdot\left[1-\exp\left\{-\int_{[-m,2m)^c}\int_{\widetilde{J}_k}\frac{\beta}{|x-y|^2}\d x\d y\right\}\right]\\
&\leq \left(1-\e^{-\widetilde{c}_1/k}\right)\cdot \left(1-\e^{-\widetilde{c}_2\alpha_1}\right) \\
&\leq\begin{cases}
\widetilde{c}_3\alpha_1,\quad &k=0, \\
\widetilde{c}_3\alpha_1/k,\quad &k\geq 1,
\end{cases}
\end{aligned}
\end{equation}
where $\widetilde{c}_r,\ r=1,2,3$ are positive constants depending only on $\beta$. Combining this with \eqref{E01-A-1}, we get that
\begin{equation}\label{P-E1-1}
\begin{aligned}
\mathds{P}[F_1^c]&\leq \sum_{k=0}^{1/\alpha_1-1}\mathds{P}[A_k]=2\sum_{k=0}^{1/(2\alpha_1)}\mathds{P}[A_k]\leq 2\widetilde{c}_3\alpha_1\left(1+\sum_{k=1}^{1/(2\alpha_1)}k^{-1}\right)\leq \widetilde{c}_4\alpha_1 \log(1/\alpha_1),
\end{aligned}
\end{equation}
where $\widetilde{c}_4$ is a positive constant depending only on $\beta$. Hence, the proof is complete by taking $\alpha_1=\e^{-M}$.
\end{proof}

We next turn to the events $F_2$ and $F_3$.

\begin{lemma}\label{P-E3}
For all $\beta>0$, $\delta\in (0,\delta^*]$, $\alpha_1\in(0,1)$ and sufficiently large $M\in\mathds{N}$, there exists a constant $\alpha_2\in (0,1)$ {\rm(}depending only on $\beta$, $M$, $\delta$ and $\alpha_1${\rm)} such that
$$
\mathds{P}\left[(F_2\cap F_3)^c\ | I_0 \text{ is $M$-good} \right]\leq {\rm e}^{-M}.
$$
\end{lemma}

\begin{proof}
Recall that $\delta^*>0$ (depending only on $\beta$) is the exponent in Lemma \ref{Lem-RNtail-1}.
For fixed $M\in \mathds{N}$ and $\alpha_1\in (0,1)$, it follows from Lemma \ref{Lem-RNtail-1} with $\varepsilon=\e^{-M}/(M+1)$ and the monotonicity of the resistance that
for any $\delta\in (0,\delta^*]$ and any $\langle u,v\rangle\in \mathcal{E}$ with $u\in I_0$ and $v\in [-m,2m)^c$, we have
\begin{equation*}
\begin{aligned}
&\mathds{P}\left[R_{I_0}\left(u,B_{\alpha_1 m}(u)^c\right)< c_*(\e^{-M}/(M+1))^{C_*\delta} (\alpha_1m)^{\delta}\right]\\
&\leq\mathds{P}\left[R\left(u,B_{\alpha_1 m}(u)^c\right)< c_*(\e^{-M}/(M+1))^{C_*\delta} (\alpha_1m)^{\delta}\right]\leq
 \e^{-M}/(M+1)
\end{aligned}
\end{equation*}
and
\begin{equation*}
\begin{aligned}
&\mathds{P}\left[\widehat{R}(I_{-1},I_{1})< c_*(\e^{-M}/(M+1))^{C_*\delta} m^{\delta}\right]\\
&\leq\mathds{P}\left[\widehat{R}_m< c_*(\e^{-M}/(M+1))^{C_*\delta} m^{\delta}\right] \leq \e^{-M}/(M+1),
\end{aligned}
\end{equation*}
where $0<c_*<C_*<\infty$ are the constants (both depending only on $\beta$) defined in Lemma \ref{Lem-RNtail-1}. We now let
\begin{equation}\label{def-alpha2}
\alpha_2=c_*(\e^{-M}/(M+1))^{C_*\delta}\alpha_1^\delta
\end{equation}
in Definition \ref{def-verygood-1} (2) and (3).
Then combining this with the definition of $M$-good in Definition \ref{def-good-1}, we conclude that
$$
\mathds{P}\left[(F_2\cap F_3)^c\ | I_0 \text{ is $M$-good}\right]\leq (M+1)\cdot \e^{-M}/(M+1)=\e^{-M},
$$
which completes the proof.
\end{proof}

We now give the

\begin{proof}[Proof of Lemma \ref{P-goodint-1}]
Due to the translation invariance of the LRP model, we only need to consider the case where $i=0$.
For $\delta\in (0,\delta^*]$ and sufficiently large $M\in \mathds{N}$, it follows from Lemmas \ref{P-E11} and \ref{P-E3} that there exist constants $\alpha_1\in (0,1)$ (depending only on $\beta$ and $M$) and $\alpha_2\in (0,1)$ (depending only on $\beta$, $\delta$ and $M$) such that
\begin{equation*}
\begin{aligned}
&\mathds{P}\left[F_1^c\cup (F_2\cap F_3)^c\ |I_0 \text{ is $M$-good}\right]\\
&\leq \mathds{P}\left[F_1^c\ |I_0 \text{ is $M$-good}\right]+\mathds{P}\left[(F_2\cap F_3)^c\ |I_0 \text{ is $M$-good}\right]\leq \e^{-c_3M} +\e^{-M},
\end{aligned}
\end{equation*}
where $c_3>0$ (depending only on $\beta$) is the constant defined in Lemma  \ref{P-E11}.
Combining this with  Definition \ref{def-verygood-1} yields the desired result.
\end{proof}

With the above lemmas at hand, we can present the

\begin{proof}[Proof of Proposition \ref{prop-alphagood-1}]
For $\delta\in (0,\delta^*]$, sufficiently large $M\in \mathds{N}$ and $\alpha=(\alpha_1,\alpha_2)\in (0,1)^2$, according to Remark \ref{M-vg}, we can see that if the interval $I_i$ is $(M,\delta,\alpha)$-very good, then the internal energy
\begin{equation*}
\begin{aligned}
R_{I_i}=&\inf_{\theta}\left(\sum_{u\in \mathcal{K}_i}\theta_u^2R_{I_i}\left(u,B_{\alpha_1m}(u)^c\right)+\left(1-\sum_{u\in \mathcal{K}_i}\theta_u\right)^2\widehat{R}(I_{i-1},I_{i+1})\right)
\geq  \alpha_2(M+1)^{-1} m^{\delta},
\end{aligned}
\end{equation*}
where the infimum is taken over $\theta=\{\theta_u\}_{u\in \mathcal{K}_i}$ with $\theta_u\geq 0$ and $\sum_{u\in \mathcal{K}_i}\theta_u\leq 1$.
Combining this with Definition \ref{def-verygood-1} (1) and Definition \ref{def-alphagood-1}, we can see that $I_i$ is $(\delta,\alpha)$-very good with $\alpha_2$ replaced by $\alpha_2(M+1)^{-1}$.
Now for any $\varepsilon\in (0,1)$, choose $M\in \mathds{N}$ such that $\e^{-c_1M}+\e^{-c_2M}\leq \varepsilon$, where $c_1,c_2$ are the constants in Lemmas \ref{lem-good-1} and \ref{P-goodint-1}, respectively. Then we have
\begin{equation}\label{good-Mgood}
\begin{aligned}
&\mathds{P}\left[I_i\text{ is not $(\delta,\alpha)$-very good}\right]\\
&\leq \mathds{P}[I_i\text{ is not $M$-good}]+ \mathds{P}\left[I_i\ \text{is not $(M,\delta,\alpha)$-very good}\ |I_i \text{ is $M$-good}\right]\\
&\leq \e^{-c_1M}+\e^{-c_2M}\leq \varepsilon,
\end{aligned}
\end{equation}
which implies the desired result.
\end{proof}


\subsection{A variational characterization of effective resistance}\label{sect-vf}
To prove Proposition \ref{propR-renorm}, a key tool we employ is a variational characterization of effective resistance, which is derived from \cite[Proposition 2.3]{BDG20}.
To introduce this characterization, we will first present some notations.
Let $\widetilde{G}=(\widetilde{V},\widetilde{E})$ be a finite, unoriented, connected graph where each edge $e\in \widetilde{E}$ is equipped with a resistance $r_e\geq 0$. For convenience, we will use $\widetilde{G}$ to denote both the corresponding network as well as the underlying graph, and let $R^{\widetilde{G}}(\cdot,\cdot)$ denote the effective resistance on $\widetilde{G}$.

For $u,v\in \widetilde{V}$, we say that a set of edges $\pi$ is a cutset between $u$ and $v$ if each path from $u$ and $v$ uses an edge in $\pi$. Let $\Pi_{u,v}$ denote the set of all finite collections of cutsets between $u$ and $v$.

\begin{lemma}[{\cite[Proposition 2.3]{BDG20}}]\label{vc-R}
For any $u,v\in \widetilde{V}$, we have
\begin{equation}\label{R-G}
\frac{1}{R^{\widetilde{G}}(u,v)}=\inf_{\Pi\in \Pi_{u,v}}\inf_{\{c_{e,\pi}: e\in \widetilde{E}, \pi\in \Pi\}\in \mathscr{C}_{\Pi}}\left(\sum_{\pi\in\Pi}\frac{1}{\sum_{e\in\pi}c_{e,\pi}}\right)^{-1},
\end{equation}
where $\mathscr{C}_\Pi$ is the collection of all sets $\{c_{e,\pi}:\ e\in \widetilde{E},\pi\in \Pi\}\in \mathds{R}_+^{\widetilde{E}\times \Pi}$ such that
\begin{equation}\label{cond-ce}
\sum_{\pi\in \Pi}\frac{1}{c_{e,\pi}}\leq r_e\quad \text{for all }e\in \widetilde{E}.
\end{equation}
The infima in \eqref{R-G} are {\rm(}jointly{\rm)} achieved.
\end{lemma}

\subsection{Proof of Proposition \ref{propR-renorm}}\label{proof-weaksup}
We continue to assume that $m,n\in \mathds{N}$ are sufficiently large, unless otherwise specified.
Recall that $G=(V,E)$ represents the renormalization of the $\beta$-LRP model, where we  identify the interval $I_i=[im,(i+1)m)$ with the vertex $\varpi(i)$ for all $i\in \mathds{Z}$. The subgraph $G_n=(V_n,E_n)$ is defined as the restriction of $G$ to $V_n=\{\varpi(0),\varpi(1),\cdots, \varpi(n-1)\}$.

For each $i\in \mathds{Z}$, let $\eta_i$ denote the degree of $\varpi(i)$ in the graph $G$, that is,
\begin{equation}\label{def-eta}
\eta_i=\left\{\varpi(j)\in V:\ \varpi(j)\sim \varpi(i)\right\}.
\end{equation}

\begin{lemma}\label{lem-eta}
For each $i\in \mathds{Z}$ and sufficiently large $C_5>0$ {\rm(}depending only on $\beta${\rm)}, we have
\begin{equation*}\label{prob-eta}
\mathds{P}\left[\eta_i\geq C_5\log n\right]\leq {\rm e}^{-C_5\log n/2}.
\end{equation*}
\end{lemma}
\begin{proof}
Due to the translation invariance of the LRP model, it suffices to consider the case where $i=0$. According to the definition of $\eta_0$ in \eqref{def-eta} and the fact that $\varpi(\pm 1)\sim \varpi(0)$, it is clear that $\eta_0=2+\sum_{j\neq \pm 1}\eta_{0j}$, where
\begin{equation*}
\eta_{0j}:=
\begin{cases}
1,\quad & \varpi(0)\sim \varpi(j),\\
0,\quad &\text{otherwise}.
\end{cases}
\end{equation*}
Note that $\{\eta_{0j}\}_{j\neq\pm 1}$ are independent.
In addition, by a simple calculation, we can see that there exist $\widetilde{C}_1,\widetilde{C}_2<\infty$ (depending only on  $\beta$) such that
\begin{equation*}
\mu:=\sum_{j\neq \pm 1}\mathds{E}\left[\eta_{0j}\right]=\sum_{j\neq \pm 1}\left(1-\exp\left\{-\beta\int_0^{1}\int_j^{j+1}\frac{1}{|x-y|^2}\d x\d y\right\}\right)\leq \sum_{j\neq \pm 1}\frac{\widetilde{C}_1\beta}{j^2}\leq \widetilde{C}_2.
\end{equation*}
Combining this with Chernoff's bound, we find that for $\widetilde{C}_3>2\widetilde{C}_2$,
\begin{equation*}
\mathds{P}\left[\eta_i\geq \widetilde{C}_3\log n\right]\leq \exp\left\{-\frac{(\widetilde{C}_3\log n/\mu-1)^2\mu}{1+\widetilde{C}_3\log n/\mu}\right\},
\end{equation*}
which implies the desired result for sufficiently large $\widetilde{C}_3$.
\end{proof}

For each $i\in \mathds{Z}$, recall that $\xi_i$ is the ``degree'' of the interval $I_i$ in the $\beta$-LRP model, as defined in \eqref{def-xi-1}. The following lemma provides an estimate of $\xi_i$ conditioned on the event $\{\eta_i< C_5\log n\}$.

\begin{lemma}\label{lem-xi}
For each $i\in \mathds{Z}$ and sufficiently large $C_5>0$,
\begin{equation*}
\mathds{P}\left[\xi_i\geq 2 C_5^2\log n \ | \eta_i< C_5\log n\right]\leq \exp\left\{-C_5^2\log n/3\right\}.
\end{equation*}
\end{lemma}
\begin{proof}
We only need to consider the case where $i=0$. For convenience, denote
$$
\xi_{0j}=\#\{\langle u,v\rangle\in \mathcal{E}:\ u\in I_0 \text{ and }v\in I_j \}.
$$
Then it is clear that $\xi_0=\sum_{j\neq \pm 1} \xi_{0j}$. Moreover, $\{\xi_{0j}\}$ are independent due to the independence of edges. Additionally, it follows from the definition of $\eta_0$  that there are only $\eta_0$ non-zero terms  among $\{\xi_{0j}\}$. Moreover, for each $j\neq \pm 1$, there exists a constant $\widetilde{C}_1<\infty$ (depending only on $\beta$) such that for all $\lambda\geq 0$,
\begin{equation*}
\begin{aligned}
\mathds{E}\left[\e^{\lambda \xi_{0j}}\ | \varpi(0)\sim \varpi(j)\right] =\frac{\mathds{E}[\e^{\lambda \xi_{0j}}]-\mathds{P}[\xi_{0j}=0]}{\mathds{P}[ \varpi(0)\sim \varpi(j)]}&=1+\frac{\mathds{E}[\e^{\lambda \xi_{0j}}]-1}{\mathds{P}[ \varpi(0)\sim \varpi(j)]}\\
&=1+\frac{\exp\left\{(e^\lambda -1)\mathds{E}[\xi_{0j}]\right\}-1}{\mathds{P}[ \varpi(0)\sim \varpi(j)]}\\
&\leq \widetilde{C}_1\e^\lambda \quad \quad\quad\text{by }\mathds{P}[ \varpi(0)\sim \varpi(j)]\asymp j^{-2}.
\end{aligned}
\end{equation*}
Consequently, we have
\begin{equation*}
\mathds{E}\left[\xi_0\ |\eta_0< C_5\log n\right]\leq \widetilde{C}_1C_5\log n \leq C^2_5\log n
\end{equation*}
for sufficiently large $C_5>0$. Combining this with the independence of $\{\xi_{0j}\}$  and Chernoff's bound, we arrive at
\begin{equation*}
\mathds{P}[\xi_0\geq 2C_5^2\log n \ | \eta_0< C_5\log n]\leq \exp\left\{-C_5^2\log n/3\right\}.
\end{equation*}
Hence, the proof is complete.
\end{proof}

We now introduce two electric networks on the subgraph $G_n$.
In the first network, we assign a unit resistance to each edge.
For simplicity, we will also refer to this network as  $G_n$, with the associated effective resistance denoted as $R^G_{V_n}(\cdot,\cdot)$, see \eqref{def-RGV}.

In the second network, we modify the resistances on the edges of the subgraph $G_n$, thereby creating a new electric network.
Specifically, recall that $R_{I_i}$ denotes the internal energy within $I_i$ as defined in \eqref{def-RIi} for all $i\in \mathds{Z}$.
For any edge $e=\langle \varpi(i),\varpi(j)\rangle \in E_n$, we assign a resistance $r'_e=\min\{R_{I_i}, R_{I_j}\}$ to $e$, resulting in a new electric network denoted as $G'_n$.
The effective resistance associated with this new network is denoted as $R^{G'}_{V_n}(\cdot,\cdot)$.
We now state  the following lemma.
\begin{lemma}\label{lemR-R^G'}
For $\delta\in (0,\delta^*]$ and $\alpha=(\alpha_1,\alpha_2)\in (0,1)^2$, assume that $I_i$ is $(\delta,\alpha)$-very good for all $i\in[0,n)$. Then for any $x,y\in [0,mn)$,
\begin{equation*}\label{R-R^G'}
R_{[0,mn)}(x,y)\geq R^{G'}_{V_n}(\varpi(i_x),\varpi(i_y)),
\end{equation*}
where $i_x,i_y\in [0,n)$ satisfying $x\in [i_xm,(i_x+1)m)$ and  $y\in [i_ym,(i_y+1)m)$.
\end{lemma}

\begin{proof}
For fixed $x,y\in [0,mn)$, let $f$ be the unit electric flow from $x$ to $y$ restricted to $[0,mn)$ in the LRP model.
We construct a unit flow $g$ from $i_x$ to $i_y$ based on $f$, which is defined as
\begin{equation}\label{def-Rn-g}
g_{\varpi(i)\varpi(j)}=\sum_{u\in I_i}\sum_{v\in I_j}f_{uv}\quad \text{for all distinct } i,j\in [0,n).
\end{equation}
Additionally, as we mentioned before, when the interval $I_i$ is $(\delta,\alpha)$-very good for $i\in [0,n)$, the internal energy $R_{I_i}$ represents the minimum energy generated by any unit flow passing through the interval $I_i$.
Combining this with \eqref{def-Rn-g}, we find that when $I_i$ is $(\delta,\alpha)$-very good for all $i\in[0,n)$,
\begin{equation*}
\begin{aligned}
R_{[0,mn)}(x,y)&=\frac{1}{2}\sum_{u\sim v}f^{2}_{uv}\\
&\geq \sum_{i\in[0,n)}\left(\sum_{u\in I_i}\sum_{v\in I^c_i}f_{vu}\I_{\{f_{vu}>0\}}\right)^2R_{I_i}\\
&= \sum_{i\in[0,n)}\left(\sum_{u\in I_i}\sum_{j\in [0,n)\setminus\{i\}}\sum_{v\in I_j}f_{vu}\I_{\{f_{vu}>0\}}\right)^2R_{I_i}\\
&\geq  \sum_{i\in[0,n)}\sum_{j\in [0,n)\setminus\{i\}}\left(\sum_{u\in I_i}\sum_{v\in I_j}f_{vu}\I_{\{f_{vu}>0\}}\right)^2R_{I_i}\\
&\geq \sum_{i\in[0,n)}\sum_{j\in [0,n)\setminus\{i\}} g^2_{\varpi(i)\varpi(j)}\I_{\{g_{\varpi(i)\varpi(j)}>0\}}\min\{R_{I_i},R_{I_j}\}\\
&= \sum_{i\in[0,n)}\sum_{j\in [0,n)\setminus\{i\}} g^2_{\varpi(i)\varpi(j)}r'_{\langle \varpi(i),\varpi(j) \rangle}\I_{\{g_{\varpi(i)\varpi(j)}>0\}}\quad \quad \text{by the definition of }r'_e\\
&= R^{G'}_{V_n}(\varpi(i_x),\varpi(i_y)),
\end{aligned}
\end{equation*}
which implies the desired result.
\end{proof}

With the above lemmas at hand, we can present the

\begin{proof}[Proof of Proposition \ref{propR-renorm}]

Fix $m,n\in \mathds{N}$. We will first clarify that it suffices to consider the case when $n\in \mathds{N}$ is sufficiently large in our proof. In fact, it follows from Remark \ref{rem-delta*} (2) that for any distinct $i,j\in [0,mn)$,
\begin{equation}\label{exp-R}
\mathds{E}\left[R_{[0,mn)}(i,j)\right]\geq \mathds{E}\left[R(0,B_{|i-j|}(0)^c)\right]\geq \widetilde{c}_1|i-j|^{\delta^*}\geq \widetilde{c}_1
\end{equation}
for some constant $\widetilde{c}_1>0$ (depending only on $\beta$). Additionally, it is obvious that $R_{[0,n)}(i,j)\leq n$. Therefore, we can derive \eqref{Rn-Rmn} for small $n$ by taking sufficiently large $C_2>0$.

In the following, we assume that $n\in \mathds{N}$ is sufficiently large.
Recall that $\xi_i$ and $\eta_i$ are defined in \eqref{def-xi-1} and \eqref{def-eta}, respectively.
Throughout the proof, we will assume that $C_5>0$ (depending only on $\beta$) is sufficiently large with
\begin{equation}\label{cond-C5}
C_5\geq \max\{6,\ 1/c_2\},
\end{equation}
where $c_2>0$ (depending only on $\beta$) is the constant defined in Lemma \ref{P-goodint-1}.
Let $A$ be the event that $\eta_i\leq C_5\log n$ for all $i\in [0,n)$ and let $B$ be the event that $\xi_i\leq 2C_5^2\log n$ for all $i\in [0,n)$. By applying Lemmas \ref{lem-eta} and \ref{lem-xi} with the union bound, we obtain
\begin{equation}\label{prob-AB}
\begin{aligned}
\mathds{P}\left[A^c\right]&\leq n\exp\{-C_5\log n/2\}\leq n^{-2}  \\
\text{and}\quad \mathds{P}\left[ B^c\ |A\right]&\leq n\exp\left\{-C_5^2\log n/3\right\}\leq n^{-2}.
\end{aligned}
\end{equation}
Additionally, recall that $R_{I_i}$ is the energy defined in \eqref{def-RIi} and $C_4<\infty$ (depending only on $\beta$) is the constant defined in Remark \ref{rem-alpha}.
We let $D$ denote the event that
$$
R_{I_i}>\e^{-3C_4}m^{1/\log n}\quad \text{for all }i\in [0,n).
$$
Then applying  \eqref{good-Mgood} with $\delta=1/\log n$, $M=2C_5^2\log n$ and using Remark \ref{rem-alpha}, we have

\begin{equation}\label{prob-ABC}
\begin{aligned}
\mathds{P}[D^c\ |A\cap B]&\leq \sum_{i=0}^{n-1}\mathds{P}\left[R_{I_i}\leq \e^{-3C_4}m^{1/\log n}\ |A\cap B\right]\\
&\leq \sum_{i=0}^{n-1}\mathds{P}\left[I_i\text{ is not $(\delta,\alpha)$-very good}\  | I_i\ \text{is $2C_5^2\log n$-good}\right]\\
&\leq n\e^{-2c_2C_5^2\log n}\leq n^{-2}\quad \quad \text{by \eqref{cond-C5}}.
\end{aligned}
\end{equation}

Now, we turn our attention to the electric networks $G_n$ and $G_n'$.
By applying Lemma \ref{vc-R} to the electric network $G_n$, we can find that for any $\varpi(i),\varpi(j)\in V_n$, there exists a collection of cutsets $\Pi_{\varpi(i),\varpi(j)}$ between $\varpi(i)$ and $\varpi(j)$ in $G_n$ and an assignment $\{c_{e,\pi}^{G}:e\in E_n,\pi\in\Pi\}\in \mathds{R}_+^{E_n\times \Pi_{\varpi(i),\varpi(j)}}$ satisfying \eqref{cond-ce} with $r_e=1$ such that
\begin{equation*}
R^{G}_{V_n}(\varpi(i),\varpi(j))=\sum_{\pi\in\Pi_{\varpi(i),\varpi(j)}}\frac{1}{\sum_{e\in\pi}c^{G}_{e,\pi}}.
\end{equation*}

Furthermore, since the  electric network $G'_n$ shares the same vertex and edge sets as $G_n$, $\Pi_{\varpi(i),\varpi(j)}$ is also a collection of cutsets between $\varpi(i)$ and $\varpi(j)$ in $G'_n$.
Moreover, for any edge $e=\langle \varpi(k),\varpi(l)\rangle$ and any cutset $\pi\in \Pi_{\varpi(i),\varpi(j)}$, we define
$$
c^{G'}_{e,\pi}=c^{G}_{e,\pi}/r'_e=c^{G}_{e,\pi}/\min\{R_{I_k},R_{I_l}\}.
$$
It is clear that the assignment $\{c^{G'}_{e,\pi}:\ e\in E_n, \pi\in \Pi_{\varpi(i),\varpi(j)}\}$ satisfies \eqref{cond-ce} with $r_e$ replaced by $r'_e$.
Therefore, applying  Lemma \ref{vc-R} to the electric network $G_n'$ again, we obtain that
\begin{equation*}
R^{G'}_{V_n}(\varpi(i),\varpi(j))\geq \sum_{\pi\in\Pi_{\varpi(i),\varpi(j)}}\frac{1}{\sum_{e\in\pi}c^{G'}_{e,\pi}}
=\sum_{\pi\in\Pi_{\varpi(i),\varpi(j)}}\frac{1}{\sum_{e\in\pi}c^{G}_{e,\pi}/r'_e}.
\end{equation*}
Combining this with the definitions of events $A,B,D$, along with \eqref{prob-AB} and \eqref{prob-ABC} we obtain that
\begin{equation}\label{cond-G}
\begin{aligned}
\mathds{E}\left[R^{G'}_{V_n}(\varpi(i),\varpi(j))\ |G\right]&\geq \mathds{E}\left[R^{G'}_{V_n}(\varpi(i),\varpi(j))\I_{A\cap B\cap D}\ |G\right]\\
&\geq \e^{-3C_4}m^{1/\log n}\sum_{\pi\in\Pi_{\varpi(i),\varpi(j)}}\frac{1}{\sum_{e\in\pi}c^{G}_{e,\pi}}\mathds{E}\left[\I_{A\cap B\cap D}\ |G\right]\\
&=\e^{-3C_4}m^{1/\log n}\I_A\mathds{P}[B|A]\mathds{P}[D|AB]R^{G}_{V_n}(\varpi(i),\varpi(j))\\
&\geq\frac{1}{4}\e^{-3C_4}m^{1/\log n}\I_AR^{G}_{V_n}(\varpi(i),\varpi(j)).
\end{aligned}
\end{equation}
In addition, note that by \eqref{prob-AB} and the fact $R^{G}_{V_n}(\varpi(i),\varpi(j))\leq n$, one has
$$
\mathds{E}\left[\I_{A^c}R^{G}_{V_n}(\varpi(i),\varpi(j))\right]\leq n\mathds{P}\left[A^c\right]\leq n^{-2}\cdot n\leq n^{-1}.
$$
Consequently, by taking expectations on both sides of \eqref{cond-G} and by the scaling invariance which implies $\mathds{E}[R^{G}_{V_n}(\varpi(i),\varpi(j))]=\mathds{E}[R_{[0,n)}(i,j)]$,  we derive
\begin{equation*}
\begin{aligned}
\mathds{E}\left[R^{G'}_{V_n}(\varpi(i),\varpi(j))\right]&\geq \frac{1}{4}\e^{-3C_4}m^{1/\log n}\left(\mathds{E}\left[R^{G}_{V_n}(\varpi(i),\varpi(j))\right]-n^{-1}\right)\\
&\geq \widetilde{c}_2\e^{-3C_4}m^{1/\log n}\mathds{E}\left[R_{[0,n)}(i,j)\right]
\end{aligned}
\end{equation*}
for some $\widetilde{c}_2>0$ depending only on $\beta$. Here in the last inequality we used \eqref{exp-R} and the assumption that $n$ is sufficiently large.
Combining this with Lemma \ref{lemR-R^G'} we have that for any $x,y\in [0,mn)$,
\begin{equation*}
\mathds{E}\left[R_{[0,mn)}(x,y)\right]\geq \widetilde{c}_2\e^{-3C_3}m^{1/\log n}\mathds{E}\left[R_{[0,n)}(\lfloor x/m\rfloor, \lfloor y/m\rfloor)\right].
\end{equation*}
Finally, letting $C_2:=\widetilde{c}_2^{-1}\e^{3C_3}$ (depending only on $\beta$), we conclude the desired result.
\end{proof}

\subsection{Proof of Proposition \ref{R(0,n)main}}\label{sect-R0n}

To prove Proposition \ref{R(0,n)main}, we start with a slightly stronger version of Proposition \ref{propR-renorm}.

\begin{lemma}\label{R-renorm}
Let $\beta>0$ and $n,n'\in \mathds{N}$ with $n\geq n'$. For any $i,j\in [0,n)$, denote $i'=\lfloor\frac{n'}{n}i\rfloor$ and $j'=\lfloor\frac{n'}{n}j\rfloor$. Then there exists a constant $C_{6,*}=C_{6,*}(\beta)<\infty$ {\rm(}depending only on $\beta${\rm)} such that for all $i,j\in [0,n)$,
\begin{equation}\label{ER-sl}
\mathds{E}\left[R_{[0,n')}(i',j')\right]\leq C_{6,*}\mathds{E}\left[R_{[0,n)}(i,j)\right].
\end{equation}
\end{lemma}

\begin{proof}
It is obvious that the statement holds for $i=j$. In the rest of the proof, we assume that $i\neq j$.

We start with the case $n=mn'$ for some $m\in \mathds{N}$.  From Proposition \ref{propR-renorm}, there exists a constant $\widetilde{c}_1(\beta)>0$ (depending only on $\beta$) such that
\begin{equation}\label{case-n=mn'}
\mathds{E}[R_{[0,n)}(i,j)]\geq \widetilde{c}_1(\beta)\mathds{E}[R_{[0,n')}(i',j')].
\end{equation}
This implies \eqref{ER-sl} holds for $n/n'\in \mathds{N}$.

For the general $n\geq n'$, write $n=mn'+r$ for some $m\in \mathds{N}$ and $0\leq r<n'$.
We will employ a coupling with the underlying continuous model to complete the proof.
To this end, let $\widetilde{\mathcal{E}}$ be a Poisson point process on $\mathds{R}^{2}$ with intensity $\beta/|x-y|^2$. We now define a random graph $G=(V,E)$ as follows: let $V=[0,n)$ and for any $i,j\in V$, $\langle i,j\rangle\in E$ if and only if $[i,i+1)\times [j,j+1)\cap n\widetilde{\mathcal{E}}\neq \emptyset$. Here $n\widetilde{\mathcal{E}}:=\{\langle nx,ny\rangle:\ \langle x,y\rangle\in \widetilde{\mathcal{E}}\}$.
Similarly, we define the random graph $G'=(V',E')$ by replacing $n$ with $mn'$ in the definition of $G$.
From the above definitions we can see that for any $\langle k,l \rangle\in E$, there exist $x_{kl},y_{kl}\in \mathds{R}$ such that
$$
\langle x_{kl},y_{kl}\rangle\in [k,k+1)\times [l,l+1)\cap n\widetilde{\mathcal{E}}.
$$
This implies
$\langle \frac{mn'}{n}x_{kl},\frac{mn'}{n}y_{kl}\rangle\in [0,mn')^2\cap (mn')\widetilde{\mathcal{E}}$.
Note that if there are multiple pairs of $\langle x_{kl},y_{kl}\rangle$, we select the pair that maximizes $|x_{kl}-y_{kl}|$.
Therefore, we have that
\begin{equation*}\label{edgeGn'-1}
\left\langle \left\lfloor\frac{mn'}{n}x_{kl}\right\rfloor, \left\lfloor\frac{mn'}{n}y_{kl}\right\rfloor\right\rangle\in E'.
\end{equation*}
In addition, for any $\langle k_1,l\rangle, \langle k_2,l\rangle\in E$, one can check that
\begin{equation}\label{edgeGn'-2}
\left|\left\lfloor\frac{mn'}{n}y_{k_1l}\right\rfloor-\left\lfloor\frac{mn'}{n}y_{k_2l}\right\rfloor\right|\leq 1.
\end{equation}

Now fix $i,j\in [0,n)$. Let $f$ be a unit electric flow from $i$ to $j$ restricted to the interval $[0, n)$.
Based on $f$, we construct a unit flow $\widetilde{f}$ from $i''=\lfloor\frac{mn'}{n}i\rfloor$ to $j''=\lfloor\frac{mn'}{n}j\rfloor$ in the graph $G'$ as follows. For any $\langle k, l\rangle \in E$ with $|k-l|\geq 2$, denote
$$
E_{kl}:=\left\{\langle k_1,l_1\rangle:\ \left\lfloor\frac{mn'}{n}x_{kl}\right\rfloor=\left\lfloor\frac{mn'}{n}x_{k_1l_1}\right\rfloor\ \text{and}\ \left\lfloor\frac{mn'}{n}y_{kl}\right\rfloor=\left\lfloor\frac{mn'}{n}y_{k_1l_1}\right\rfloor\right\}.$$
Since $mn'\leq n<(m+1)n'$, it is obvious that $|E_{kl}|\leq 4$. Now let

$$
\widetilde{f}_{\lfloor\frac{mn'}{n}x_{kl}\rfloor \lfloor\frac{mn'}{n}y_{kl}\rfloor}=\sum_{\langle k_1,l_1\rangle\in E_{kl}}f_{k_1l_1}.
$$
In addition, for any $\langle k, l\rangle \in E$ with $|k-l|=1$, we define $\widetilde{f}_{\lfloor\frac{mn'}{n}x_{kl}\rfloor\lfloor\frac{mn'}{n}y_{kl}\rfloor}$ such that $f$ becomes a flow between $i''$ to $j''$ in the graph $G'$. From the fact $|E_{kl}|\leq 4$ for each $\langle k,l\rangle\in E $ and Lemma \ref{lem:compare} applied for the flow $f$, we can see that
\begin{equation*}
\begin{aligned}
&\mathds{E}[R_{[0,mn')}(i'',j'')]\leq \mathds{E}\left[\frac{1}{2}\sum_{\langle k',l'\rangle\in E'}\widetilde{f}_{k'l'}^2\right]
\leq \mathds{E}\left[\frac{1}{2}\left[2^4\sum_{\langle k,l\rangle\in E}f_{kl}^2+\sum_{i\in V}\left(\sum_{j\in V:f_{ji}>0}f_{ji}\right)^2\right]\right]\\
&\leq \mathds{E}\left[\frac{1}{2}\left[2^4\sum_{\langle k,l\rangle\in E}f_{kl}^2+\widetilde{C}_1\sum_{i\in V}\sum_{j\in V:f_{ji}>0}f_{ji}^2\right]\right]\\
&\leq \widetilde{C}_2\mathds{E}[R_{[0,n)}(i,j)]
\end{aligned}
\end{equation*}
for some constants $\widetilde{C}_1, \widetilde{C}_2<\infty$ (depending only on $\beta$).
Combining this with \eqref{case-n=mn'}, we get that there exist constants $\widetilde{c}_2(\beta),\ \widetilde{c}_3(\beta),\ \widetilde{c}_4(\beta) >0$ (all depending only on $\beta$) such that
\begin{equation*}
\begin{aligned}
\mathds{E}[R_{[0,n)}(i,j)]&\geq \widetilde{c}_2(\beta)\mathds{E}[R_{[0,mn')}(i'',j'')]\\
&\geq \widetilde{c}_3(\beta)\mathds{E}\left[R_{[0,n')}\left(\left\lfloor\frac{i''}{m}\right\rfloor, \left\lfloor\frac{j''}{m}\right\rfloor\right)\right]
\geq \widetilde{c}_4(\beta)\mathds{E}[R_{[0,n')}(i',j')].
\end{aligned}
\end{equation*}
Hence, the proof is complete.
\end{proof}

\begin{proof}[Proof of Proposition \ref{R(0,n)main}]
For $i,j\in [0,n)$ with $j>i$, it follows from the monotonicity of effective resistance and Lemma \ref{R-renorm} that
\begin{equation*}
\mathds{E}[R_{[0,n)}(i,j)]\leq \mathds{E}[R_{[i,j]}(i,j)]=\mathds{E}[R_{[0,j-i]}(0,j-i)]\leq C_{6,*}\mathds{E}[R_{[0,n)}(0,n-1)],
\end{equation*}
where $C_{6,*}<\infty$ (depending only on $\beta$) is the constant defined in Lemma \ref{R-renorm}.
\end{proof}

\section{A weaker version of supermultiplicativity}\label{sect-weaksupermult}

In this section, we will prove a weaker version of the supermultiplicativity of effective resistance as outlined below.

\begin{proposition}\label{weak-supermult}
Recall $\delta^*$ from \ref{Lem-RNtail-1}.
For all $\beta>0$, there exists a constant $\delta=\delta(\beta)\in (0,\delta^*]$ {\rm(}depending only on $\beta${\rm)} such that for all $m,n\in \mathds{N}$ large enough,
\begin{equation*}
\mathds{E}\left[R_{[0,mn-1]}(0,mn-1)\right]\geq m^\delta \mathds{E}\left[R_{[0,n-1]}(0,n-1)\right].
\end{equation*}
\end{proposition}

 \begin{remark}\label{dcs-delta}
We note that the exponent $\delta $ in \eqref{exponent-delta} from the proof of Proposition \ref{weak-supermult} below  is derived from Lemma \ref{Lem-RNtail-1}. By combining this observation with Remark \ref{rem-delta*}, we can conclude that the exponent $\delta$ in Proposition \ref{weak-supermult} can also be chosen as a decreasing function of $\beta$.
\end{remark}

Hereon, we provide a general overview on our proof of Proposition \ref{weak-supermult}. For fixed sufficiently large $m,n\in \mathds{N}$, we divide the interval $[0,mn)$ uniformly into $n$ smaller intervals of length $m$, and we employ the notation of very good interval as defined in Definition \ref{def-alphagood-1}.
As shown in Proposition \ref{prop-alphagood-1}, with appropriately chosen parameters, each such interval is very good with high probability.
This implies that when a unit flow flows from $0$ to $mn-1$, it will generate sufficiently large energy in most of the small intervals.
However, there are some small intervals where the flow does not generate enough energy. Our main effort is to address these intervals.

To achieve this, we proceed in two steps. First,
we introduce the renormalization $G=(V,E)$ from the $\beta$-LRP model by identifying the intervals $I_i$ to $\varpi(i)$. Recall that $G_n=(V_n,E_n)$ is the restriction of $G$ to $V_n=\{\varpi(0),\varpi(1),\cdots, \varpi(n-1)\}$.
We say a vertex $\varpi(i)\in V$ is not very good if $I_i$ is not very good. We then consider the not very good component in $G_n$, which is a connected component of $G_n$  consisting of not very good vertices.
By strengthening the estimates from Section \ref{sect-far} along with the BK inequality, we prove that the growth of the size of such not very good component exhibits exponential decay in probability (see Proposition \ref{P-red-many-tree}).

In the second step, we will employ a coarse-graining argument, which will be accomplished in Section \ref{sect-cg}. Specifically, we will iteratively define ``good'' components at various scales, where the energy generated by the flow passing through these components is significantly smaller than the energy generated by the nearby intervals at the same scale.
We will then show that, with high probability, all not very good components can be considered ``good'', that is, we can find intervals at various scales to cover all these not very good components. Combining these statements, we can conclude Proposition \ref{weak-supermult}.

We now introduce some notation which will be used repeatedly. Throughout this section, we fix sufficiently large $m,n\in \mathds{N}$. Recall that $\delta^*$ (depending only on $\beta$) is the exponent in Lemma \ref{Lem-RNtail-1} and  we write the intervals $[im,(i+1)m)$ as $I_i$ for all $i\in \mathds{Z}$. For $\delta\in(0,\delta^*]$ and $\alpha=(\alpha_1,\alpha_2)\in (0,1)^2$, the $(\delta,\alpha)$-very good interval is defined in Definition \ref{def-alphagood-1}.
In the following, we will say the vertex $\varpi(i)$ is $(\delta,\alpha)$-very good if $I_i$ is $(\delta,\alpha)$-very good.

\subsection{Some estimates for not very good components}
The goal of this section is to estimate the probability of the occurrence of not $(\delta,\alpha)$-very good component, which is a connected component in $G$ consisting of not $(\delta,\alpha)$-very good vertices.
To do this, we let $\mathcal{CS}_k(0)$ denote the collection of all connected subsets of the graph $G$ that are of size $k$ and include the vertex $\varpi(0)$.

\begin{proposition}\label{P-red-many-tree}
For any $\beta>0$, $\delta\in (0,\delta^*]$ and sufficiently large $M>0$, there exist constants $c_1>0$ {\rm(}depending only on $\beta${\rm)} and $\alpha=(\alpha_1,\alpha_2)\in (0,1)^2$ {\rm(}depending only on $\beta$ and $M${\rm)} such that for all $k\in\mathds{N}$,
\begin{equation*}
\mathbb{P}\left[\exists Z\in\mathcal{CS}_k(0): \forall \varpi(j)\in Z\setminus \{\varpi(0)\}, \varpi(j)\text{ is not $(\delta,\alpha)$-very good}\right]\leq {\rm e}^{-c_1Mk}.
\end{equation*}
\end{proposition}

For $i\in \mathds{Z}$, recall that $\xi_i$ is the ``degree'' of the interval $I_i$ in the $\beta$-LRP model as defined in \eqref{def-xi-1}.
The following proposition serves as the main input for Proposition \ref{P-red-many-tree}. It shows that an animal (i.e., connected subset, drawing terminology from the lattice animal) composed of vertices in $G$ that are $(\delta,\alpha)$-not very good and possess relatively low degrees is not likely to be excessively large.

\begin{proposition}\label{P-red-many}
For any $\beta>0$, $\delta\in (0,\delta^*]$ and  sufficiently large $M>0$, there exist constants $c_2>0$ {\rm(}depending only on $\beta${\rm)} and $\alpha=(\alpha_1,\alpha_2)\in(0,1)^2$ {\rm(}depending only on $\beta$ and $M${\rm)} such that the following holds.
For any $k\in\mathds{N}$ and any $k$ distinct vertices $\varpi(i_1),\cdots,\varpi(i_k)\in V$, let  $\{W_{i_l}\}_{1\leq l\leq k}$ be a sequence of sets such that
\begin{equation*}
W_{i_l}\subset V\setminus \{\varpi(i_1),\cdots,\varpi(i_k)\}\quad  \text{and}\quad  0\leq \#W_{i_l}\leq 2\quad  \text{for each } l\in[1,k].
\end{equation*}
Then we have
  \begin{equation*}
    \mathbb{P}\left[ \varpi(i_l)\text{ is not $(\delta,\alpha)$-very good for all $l\in[1,k]$ and }\max_{1\leq l\leq k}\xi_{i_l}\leq M \ | \cap_{l=1}^k H_{i_l}\right]\leq {\rm e}^{-c_2 Mk},
  \end{equation*}
where $H_{i_l}$ is the event that $\varpi(i_l)\sim \varpi(j)$  for all $\varpi(j)\in W_{i_l}$. If $W_{i_l}=\emptyset$, then $H_{i_l}$ is regarded as the entire sample space.
\end{proposition}

To prove Proposition \ref{P-red-many}, we recall that for any $i\in  \mathds{Z}$,  $F_{k}(i)$ denotes the event in Definition \ref{def-verygood-1} ($k$) for $k=1,2,3$, respectively.
It is clear from \eqref{P-E1-1} with taking $\alpha_1=\e^{-M(M+1)}$ and the translation invariance of the LRP model, we have
\begin{equation*}
\mathds{P}[F_{1}(i)^c]\leq \e^{-cM(M+1)}
\end{equation*}
for some $c>0$ depending only on $\beta$.
We will now make a slight modification to the proof of Lemma \ref{P-E11} to derive the following estimate.

\begin{lemma}\label{P-Ek1-new}
  For all $\beta>0$ and sufficiently large $M>0$, there exist constants $c_3=c_3(\beta)>0$ {\rm(}depending only on $\beta${\rm)} and $\alpha_1\in (0,1)$ {\rm(}depending only on $\beta$ and $M${\rm)} such that for each $i\in \mathds{Z}$ and each $W\subset V\setminus\{\varpi(i)\}$ with $0\leq \#W\leq 2$,
  $$
  \mathds{P}\left[F_1(i)^c\ |\varpi(i)\sim \varpi(j)\text{ for all } \varpi(j)\in W\right]\leq {\rm e}^{-c_3M(M+1)}.
  $$
\end{lemma}
\begin{proof}
  Without loss of generality, we can assume that $i=0$, $W\neq \emptyset$ and $|j|>1$ for all $\varpi(j)\in W$.
  For convenience, we also assume that $1/(2\alpha_1)\in \mathds{N}$. Otherwise, we can replace $1/(2\alpha_1)$ with $\lfloor 1/(2\alpha_1)\rfloor$.

  Recall the intervals $J_k,\widetilde{J}_k$ and the events $A_k$ with $k\in \mathbb[0,1/\alpha_1-1)$ are defined in the proof of Lemma \ref{P-E11} (preceding to \eqref{E01-A-1}).
  Let $H$ denote the event that $\varpi(0)\sim \varpi(j)$ for all $\varpi(j)\in W$.
  For each $k\in \mathbb[0,1/\alpha_1-1)$, based on the number of edges connecting $\cup_{j:\varpi(j)\in W}I_j$ and $\widetilde{J}_k$, we can decompose the event $H$ into the following three events:
  \begin{equation*}
  H_{k,l}=H\cap\left\{\#\mathcal{E}_{\cup_{j:\varpi(j)\in W}I_j\times \widetilde{J}_k}=l\right\}\ \text{for }l=0,1\quad \text{and}\quad H_{k,2}=H\cap \left\{\#\mathcal{E}_{\cup_{j:\varpi(j)\in W}I_j\times \widetilde{J}_k}\geq 2\right\}.
  \end{equation*}
Then by the independence of edges,  we arrive at
\begin{equation}\label{eq:P-Ek1}
    \begin{aligned}
      \mathbb{P}[A_k\ |H]&= \frac{\mathbb{P}[A_k\cap H_{k,0}]+\mathbb{P}[A_k\cap H_{k,1}]+\mathbb{P}[A_k\cap H_{k,2}]}{\mathbb{P}[H]}\\
      &\leq \mathbb{P}[A_k]\prod_{j:\varpi(j)\in W}\frac{\mathbb{P}[I_j\sim I_0\setminus \widetilde{J}_k]}{\mathbb{P}[I_j\sim I_0]}
      +\mathbb{P}[I_0^c\setminus\cup_{j:\varpi(j)\in W}I_j\sim J_k]\sum_{j:\varpi(j)\in W}\frac{\mathbb{P}[I_j\sim \widetilde{J}_k]}{\mathbb{P}[I_j\sim I_0]}\\
      &\quad \quad +\frac{\mathbb{P}[A_k\cap H_{k,2}]}{\mathbb{P}[H]}.
    \end{aligned}
  \end{equation}
 Note that by some direct calculations, we can find $0<\widetilde{c}_1<\widetilde{C}_1$ (both depending only on $\beta$) such that the following holds for each $|j|>1$, $\alpha_1<1/10$ and $k\in [0,1/(2\alpha_1)]$:
  \begin{align*}
      \mathbb{P}[I_j\sim I_0]&=1-\exp\left\{-\beta\int_{I_0}\int_{I_j} \frac{dxdy}{|x-y|^2}\right\}\in [\widetilde{c}_1 j^{-2},\widetilde{C}_1 j^{-2}];\\
    \mathbb{P}[I_j\sim \widetilde{J}_k]&=1-\exp\left\{-\beta\int_{\widetilde{J}_k}\int_{I_j} \frac{dxdy}{|x-y|^2}\right\}\in [\widetilde{c}_1\alpha_1 j^{-2},\widetilde{C}_1\alpha_1 j^{-2}];\\
    \mathbb{P}[I_j\sim I_0\setminus\widetilde{J}_k]&=1-\exp\left\{-\beta\int_{I_0\setminus \widetilde{J}_k}\int_{I_j} \frac{dxdy}{|x-y|^2}\right\}\in [\widetilde{c}_1 j^{-2},\widetilde{C}_1 j^{-2}];\\
    \mathbb{P}[I_0^c\sim J_k]&=1-\exp\left\{-\beta\int_{J_k}\int_{I_0^c} \frac{dxdy}{|x-y|^2}\right\}\in [\widetilde{c}_1 k^{-1},\widetilde{C}_1 k^{-1}],\quad k\neq 0.
  \end{align*}
  Applying these estimates to \eqref{eq:P-Ek1} yields that there exists a constant $\widetilde{c}_2>0$ depending only on $\beta$ such that for each $k\in[0,1/(2\alpha_1)]$,
  \begin{equation}\label{eq:P-Ek1-2}
    \mathbb{P}[A_k\ |H]\leq \widetilde{c}_2\mathbb{P}[A_k]+\widetilde{c}_2\alpha_1\left(\frac{1}{k}\I_{\{k\neq 0\}}+\I_{\{k=0\}}\right)+\frac{\mathbb{P}[A_k\cap H_{k,2}]}{\mathbb{P}[H]}.
  \end{equation}

  In the following, we turn to estimate the last term on the RHS of \eqref{eq:P-Ek1-2}.
  We start with the case where $\# W=1$. For convenience, denote by $W=\{\varpi(j)\}$ for some $|j|>1$.
  In this case, $H_{k,2}$ implies that $\# \mathcal{E}_{I_j\times \widetilde{J}_k}\geq 2$.  Consequently (below we denote by $\mathrm{Poi}(\lambda)$ a Poisson variable with mean $\lambda$),
  \begin{equation}\label{eq:P-Ek1-3}
    \frac{\mathbb{P}[A_k\cap H_{k,2}]}{\mathbb{P}[H]}=\frac{\mathbb{P}[{\rm Poi}(\widetilde{\lambda}_{j,k})\geq 2]}{\mathbb{P}[{\rm Poi}(\lambda_j)\geq 1]}\leq \widetilde{c}_3 \alpha_1^2 |j|^{-2}\leq \widetilde{c}_3 \alpha_1^2,
  \end{equation}
  for some $\widetilde{c}_3>0$ (depending only on $\beta$), where $\widetilde{\lambda}_{j,k}:=\beta\int_{I_j}\int_{\widetilde{J}_k}\frac{dxdy}{|x-y|^2}$ and $\lambda_j:=\beta\int_{I_j}\int_{I_0}\frac{dxdy}{|x-y|^2}$.

For the case when $\# W=2$, we need to consider one more case when $I_j\sim \widetilde{J}_k$ for both $\varpi(j)\in W$.
Using a similar calculation we can also get \eqref{eq:P-Ek1-3} with another constant $\widetilde{c}_4$ (depending only on $\beta$) instead of $\widetilde{c}_3$.
Taking $\widetilde{c}=\max\{\widetilde{c}_2,\widetilde{c}_3,\widetilde{c}_4\}$ and combining this with \eqref{eq:P-Ek1-2} yields that for each $k\in[0,1/(2\alpha_1)]$,
  \begin{equation}\label{P-AkH}
    \mathbb{P}[A_k\ |H]\leq \widetilde{c}\left(\mathbb{P}[A_k]+\frac{\alpha_1}{k}\I_{\{k\neq 0\}}+\alpha_1\I_{\{k=0\}}+\alpha_1^2\right).
  \end{equation}
 Hence, we can complete the proof by applying \eqref{P-Ak} to \eqref{P-AkH} and then using a straightforward computation (similar to that in \eqref{P-E1-1}).
\end{proof}

In the following, we consider the near independence of $F_1(i)$ among different $i\in \mathds{Z}$.

\begin{lemma}\label{P-Ek1-many}
  For all $\beta>0$ and sufficiently large $M\in\mathds{N}$, let $\alpha_1$ and $c_3$ be the constants defined in Lemma~\ref{P-Ek1-new}. Then the following holds. For any $k\in\mathds{N}$ and $\varpi(i_1),\cdots \varpi(i_k)\in V$, let $\{W_{i_l}\}_{1\leq l\leq k}$ and $\{H_{i_l}\}_{1\leq l\leq k}$ be the sets and events defined in  Proposition~\ref{P-red-many}. Then we have
  $$
  \mathbb{P}\left[\cap_{l=1}^k F_1(i_l)^c,\ \max_{1\leq l\leq k}\xi_{i_l}\leq M \ |\cap_{l=1}^k H_{i_l}\right]\leq 2^k{\rm e}^{-c_3 Mk}.
  $$
\end{lemma}
\begin{proof}
  We first claim that if the events $\cap_{\ell=1}^k F_1(i_\ell)^c$ and $\max_{1\leq \ell\leq k}\xi_{i_\ell}\leq M $ both occur, then we can find $\lfloor k/(M+1)\rfloor+1$ events of the form $F_{1}(i_\ell)^c$ that are ``certified'' by disjoint edges sets (that is, there exists disjoint edge sets where each edge set ensures the occurrence of an event).
To illustrate this, we construct these $i_l$'s step by step.

   We begin by setting $j_1=1$.
   Since $F_1(i_{j_1})=F_1(i_1)$ is determined by the edges that have at least one endpoint in the interval  $[(i_1-1)m,(i_1+2)m)$, there are at most $M$ indices $i_l$  such that $|i_l-i_1|\geq 2$ and $\varpi(i_l)$ is directly connected to $\varpi(i_{1})$.
    We then remove these indices from the set $\{\varpi(i_1),\cdots, \varpi(i_k)\}$ and define $i_{j_2}$ as the smallest remaining index. This process is repeated for each selected $i_{j_l}$ until no new index can be found among  $\{\varpi(i_1),\cdots, \varpi(i_k)\}$.

  Since in each iteration we select one index $i_l$ and remove at most $M$ indices from the set $\{\varpi(i_1),\cdots, \varpi(i_k)\}$, we can conclude that there are at least $\lfloor k/(M+1)\rfloor$ events of the form $F_{1}(i_l)^c$ that are certified by disjoint edge sets. This, combined with the BK inequality \cite{BK85}, Lemma \ref{P-Ek1-new} and the independence of the different events $H_{i_l}$, implies that
  \begin{equation*}
  \begin{aligned}
    &\mathbb{P}\left[\bigcap_{l=1}^k F_1(i_l)^c,\ \max_{1\leq l\leq k}\xi_{i_l}\leq M\ |\cap_{l=1}^k H_{i_l}\right]\\
    &\leq \binom{k}{\lfloor k/(M+1)\rfloor}\left(\e^{-c_3M(M+1)}\right)^{k/(M+1)}\leq 2^k\e^{-c_3Mk},
    \end{aligned}
  \end{equation*}
  which completes the proof.
\end{proof}

We next turn to the events $F_2(i)$ and $F_3(i)$.

\begin{lemma}\label{P-Ek3}
For all $\beta>0$, $\delta\in (0,\delta^*]$, sufficiently large $M>0$ and $\alpha_1\in(0,1)$, there exist constants $\alpha_2=\alpha_2(\beta)\in (0,1)$ {\rm(}depending only on $\beta$, $M$ and $\alpha_1${\rm)} and $c_4>0$ {\rm(}depending only on $\beta${\rm)} such that the following holds. For any $k\in\mathds{N}$ and $\varpi(i_1),\cdots \varpi(i_k)\in V$, let $\{W_{i_l}\}_{1\leq l\leq k}$ and $\{H_{i_l}\}_{1\leq l\leq k}$ be the sets and events defined in Proposition~\ref{P-red-many}. Then we have
$$
\mathbb{P}\left[\bigcap_{l=1}^k \left(F_2(i_l)\cap F_3(i_l)\right)^c,\ \max_{1\leq l\leq k}\xi_{i_l}\leq M\ |\cap_{\ell=1}^k H_{i_\ell}\right]\leq {\rm e}^{-c_4Mk}.
$$
\end{lemma}

\begin{proof}
Recall that $\delta^*>0$ (depending only on $\beta$) is the exponent in Lemma \ref{Lem-RNtail-1}.
For fixed $M\in \mathds{N}$, $\delta\in (0,\delta^*]$ and $\alpha_1\in (0,1)$, it follows from Lemma \ref{Lem-RNtail-1} that there exists a constant $\alpha_2>0$ (depending on $\beta$, $M$ and $\alpha_1$) such that for any $l\in [1,k]$ and $\langle u,v\rangle\in \mathcal{E}$ with $u\in I_{i_l}$ and $v\in ((i_l-1)m,(i_l+2)m]^c$, we have
\begin{equation*}
\mathds{P}\left[R_{I_{i_l}}\left(u,B_{\alpha_1 m}(u)^c\right)< \alpha_2 m^{\delta}\right]
\leq \mathds{P}\left[R\left(u,B_{\alpha_1 m}(u)^c\right)< \alpha_2 m^{\delta}\right]
\leq \e^{-M(M+4)}/(M+1)
\end{equation*}
and
\begin{equation*}
  \mathds{P}\left[\widehat{R}(I_{i_l-1},I_{i_l+1})<\alpha_2m^{\delta}\right]
  \leq \mathds{P}\left[\widehat{R}_m<\alpha_2m^{\delta}\right]
  \leq \e^{-M(M+4)}/(M+1).
\end{equation*}
Consequently,  for any $l\in [1,k]$ one has
\begin{equation}\label{P-Iil}
\mathds{P}\left[R_{i_l}<\alpha_2 m^{\delta},\ \xi_{i_l}\leq M\right]\leq (M+1)\cdot\e^{-M(M+4)}/(M+1)=\e^{-M(M+4)},
\end{equation}
where
\begin{equation*}
R_{i_l}:=\min_{u\in \mathcal{K}_{i_l}}\left\{R_{I_{i_l}}\left(u,B_{\alpha_1 m}(u)^c\right),\ \widehat{R}(I_{i_l-1},I_{i_l+1})\right\}
\end{equation*}
and $\mathcal{K}_{i_l}$ is defined in \eqref{def-Ki-1}.

We now consider the independence of $R_{i_l}$ among different $i_l$. It is clear that on the event $\max_{1\leq l\leq k}\xi_{i_l}\leq M$, the resistance $R_{i_l}$ depends on at most $M$ edges which has one endpoint outside the interval $[(i_l-1)m,(i_l+2)m)$.
Therefore, using the similar construction in the proof of Lemma \ref{P-Ek1-many}, we can see that there are at least $\lfloor k/(M+4)\rfloor$ indices $i_l$ such that the corresponding $R_{i_l}$ are independent of one another as well as independent of the events $H_{i_l}$. Combining this with \eqref{P-Iil} we obtain that
  \begin{equation*}
  \begin{aligned}
    &\mathbb{P}\left[\bigcap_{l=1}^k (F_2(i_l)\cap F_3(i_l))^c,\ \max_{1\leq l\leq k}\xi_{i_l}\leq M\ |\cap_{l=1}^k H_{i_l}\right]\\
    &\leq \binom{k}{\lfloor k/(M+4)\rfloor}\left(\e^{-M(M+4)}\right)^{k/(M+4)}\leq \e^{-\widetilde{c}Mk}
    \end{aligned}
  \end{equation*}
for some constant $\widetilde{c}>0$ depending only on $\beta$.
\end{proof}

With the above lemmas at hand, we can present the

\begin{proof}[Proof of Proposition~\ref{P-red-many}]
For $\beta>0$, $\delta\in (0,\delta^*]$ and sufficiently large $M\in \mathds{N}$, let $\alpha=(\alpha_1,\alpha_2)\in (0,1)^2$ (depending only on $\beta$ and $M$) be the constants chosen in Lemmas \ref{P-Ek1-many} and \ref{P-Ek3}.

Note that on the event $\cap_{l=1}^k\{\varpi(i_l)\text{ is not $(\delta,\alpha)$-very good and }\xi_{i_l}\leq M\}$, there are  either at least $k/2$ indices $i_l$ for which $F_1(i_l)^c\cap \{\xi_{i_l}\leq M\}$ occurs or at least $k/2$ indices $i_l$ for which $(F_2(i_l)\cap F_3(i_l))^c\cap \{\xi_{i_l}\leq M\}$ occurs.
Therefore, from Lemmas~\ref{P-Ek1-many} and \ref{P-Ek3} we have
\begin{equation*}
  \begin{aligned}
    &\mathbb{P}\left[ \varpi(i_l)\text{ is not $(\delta,\alpha)$-very good for all $l\in[1,k]$ and }\max_{1\leq l\leq k}\xi_{i_l}\leq M \ | \cap_{l=1}^k H_{i_l}\right]\\
    &\leq \binom{k}{\lfloor k/2\rfloor+1}(2^k\e^{-c_3Mk/2} +\e^{-c_4Mk/2})\leq \e^{-\widetilde{c}Mk}
  \end{aligned}
\end{equation*}
for some $\widetilde{c}>0$ depending only on $\beta$, where $c_3,c_4>0$ (both depending only on $\beta$) are the constants in Lemmas \ref{P-Ek1-many} and \ref{P-Ek3}, respectively.
\end{proof}

We will now remove the restrictions on the degrees of the vertices from Proposition~\ref{P-red-many} to obtain the following estimate.

\begin{proposition}\label{P-red-many-strong}
For any $\beta>0$, $\delta\in (0,\delta^*]$ and  sufficiently large $M>0$, there exist constants $c_{5}>0$ {\rm(}depending only on $\beta${\rm)} and $\alpha=(\alpha_1,\alpha_2)\in(0,1)^2$ {\rm(}depending only on $\beta$ and $M${\rm)} such that the following holds.
For any $k\in\mathds{N}$ and any $k$ distinct vertices $\varpi(i_1),\cdots,\varpi(i_k)\in V$, let  $\{W_{i_l}\}_{1\leq l\leq k}$ and $\{H_{i_l}\}_{1\leq l\leq k}$ be the sets and events defined in Proposition~\ref{P-red-many}. Then we have
  \begin{equation}\label{prob-notgoodclst}
    \mathbb{P}\left[ \varpi(i_l)\text{ is not $(\delta,\alpha)$-very good for all $l\in[1,k]$} \ | \cap_{l=1}^k H_{i_l}\right]\leq {\rm e}^{-c_5 Mk}.
  \end{equation}
\end{proposition}

\begin{proof}
   It is obvious that if we have $\sum_{l=1}^k \xi_{i_l}\leq Mk/2 $, then there are at least $k/2$ indices $\varpi(i_l)$ such that $\xi_{i_l}\leq M$. As a result, applying Proposition~\ref{P-red-many} to this $k/2$ indices $i_l$, we get that
  \begin{equation}\label{up-LHS5}
    \text{LHS of \eqref{prob-notgoodclst}}\leq \e^{-c_2 Mk/2}+\mathbb{P}\left[\sum_{l=1}^k \xi_{i_l}\geq Mk/2\ \ |\cap_{l=1}^k H_{i_l}\right],
  \end{equation}
  where $c_2>0$ (depending only on $\beta$) is the constant in Proposition~\ref{P-red-many}.

 In the remaining part of the proof, we will estimate the last term on the RHS of \eqref{up-LHS5}.
 To this end, for $l\in [1,k]$, let $\widetilde{\xi}_{i_l}$ represent the number of edges in the LRP model that directly connect $[i_{l}m, (i_{l}+1)m)$ and $\mathds{Z}\setminus \left(\cup_{l'=1}^k [(i_{l'}-1)m, (i_{l'}+2)m)\right)$.
For $l,l'\in [1,k]$ with $|i_l-i_{l'}|>1$, let $\xi_{i_l,i_{l'}}$ denote the number of edges connecting the intervals $[i_{l}m, (i_{l}+1)m)$ and $[i_{l'}m, (i_{l'}+1)m)$ in the LRP model.
  It is evident, due to the independence of edges in the LRP model and the assumption $\varpi(i_l)\notin W_{i_l}$ for all $l\in [1,k]$, that $\{(\widetilde{\xi}_{i_l},H_{i_l})\}_{l=1}^k$ and $\{\xi_{i_l,i_{l'}}\}_{l,l'=1}^k$ are all independent.  In addition, note that if $\sum_{l=1}^k \xi_{i_l}\geq Mk/2$, then one has
  \begin{equation}\label{def-Y-0}
  \begin{aligned}
  2Y&:=2\left(\sum_{l,l'\in [1,k]:|i_l-i_{l'}|>1}\xi_{i_l,i_{l'}}+\sum_{l=1}^k \widetilde{\xi}_{i_l}\right)\geq 2\sum_{l,l'\in [1,k]:|i_{l}-i_{l'}|\geq 2}\xi_{i_l,i_{l'}}+\sum_{l=1}^k \widetilde{\xi}_{i_l}\\
  &\geq \sum_{l=1}^k\xi_{i_l}\geq Mk/2.
  \end{aligned}
  \end{equation}
Moreover, by a simple calculation, there exist constants $\widetilde{C}_1,\widetilde{C}_2<\infty$ (depending only on $\beta$) such that
\begin{equation*}
\begin{aligned}
\mathds{E}[Y\ |\cap_{l=1}^k H_{i_l}]
&\leq \sum_{l,l'\in [1,k]:|i_l-i_{l'}|>1}\mathbb{E}[\xi_{i_l,i_{l'}}]+\sum_{l=1}^k \mathbb{E}\left[\widetilde{\xi}_{i_l}\ | H_{i_l}\right]\\
&\leq \sum_{l=1}^k\sum_{i\in [lm,(l+1)m)} \sum_{j\in [(l-1)m,(l+2)m)^c}\frac{\widetilde{C}_1\beta}{|i-j|^2}+\sum_{l=1}^k \mathbb{E}\left[\widetilde{\xi}_{i_l}\ |H_{i_l}\right]\leq \widetilde{C}_2k.
\end{aligned}
\end{equation*}
Combining this with \eqref{def-Y-0}, the independence of  $\{(\widetilde{\xi}_{i_l},H_{i_l})\}_{l=1}^k$ and $\{\xi_{i_l,i_{l'}}\}_{l,l'=1}^k$ and Chernoff's bound, we can get that for all $M\geq 6\widetilde{C}_2$,
$$
\mathds{P}\left[\sum_{l=1}^k \xi_{i_l}\geq Mk/2\ |\cap_{l=1}^k H_{i_l}\right]\leq \mathds{P}\left[Y\geq Mk/4\ |\cap_{l=1}^k H_{i_l}\right]\leq \e^{-Mk/72}. 
$$
Hence, we can complete the proof by applying this to \eqref{up-LHS5}.
\end{proof}

To prove Proposition \ref{P-red-many-tree}, we also need some estimates for the connected subsets in the LRP model.
To this end, for any $k\in \mathds{N}$, recall that $\mathcal{CS}_k(0)$ denotes the collection of all connected subsets of the graph $G$ that have  size $k$ and contain $\varpi(0)$.
Let $\mathcal{TA}_k(0)$ be the set of admissible trees $T$ on $V=\{\varpi(0),\varpi(-1),\varpi(1),\cdots\}$ that have size $k$ and contain $\varpi(0)$, where an admissible tree is a tree that can possibly occur in $G$.
We also let $\mathcal{T}_k(0)$ denote the set of all trees in $G$ that have size $k$ and contain $\varpi(0)$.
Note that $\mathcal{TA}_k(0)$ is deterministic and $\mathcal{T}_k(0)$ is random.
It is worth emphasizing that each tree $T$ in $\mathcal{T}_k(0)$ can be viewed as a spanning tree of a connected subset $Z\in \mathcal{CS}_k(0)$. 
\begin{lemma}[{\cite[Lemma 3.1]{Baumler23a}}]\label{Baumlerlemma3.1}
For any $\beta>0$ and $k\in \mathds{N}$, $\mathds{E}[|\mathcal{T}_k(0)|]\leq (4\mu_\beta)^k$, where $\mu_\beta:=\mathds{E}[\rm{deg}(\varpi(0))]$.
\end{lemma}

We now present the

\begin{proof}[Proof of Proposition \ref{P-red-many-tree}]
Fix $\beta>0$, $\delta \in (0,\delta^*]$ and sufficiently large $M>0$. We let $\alpha=(\alpha_1,\alpha_2)\in (0,1)^2$ be the constants defined in Proposition \ref{P-red-many-strong}.
  Without loss of generality, let $k'=k/6$ be an integer. Otherwise, we can replace $k/6$ with $\lfloor k/6\rfloor$.

  Note that every connected subset $Z\in \mathcal{CS}_k(0)$ corresponds to one of its spanning trees $T\in \mathcal{T}_k(0)$.
Therefore, we obtain that
    \begin{align}
      &\mathds{P}\left[\exists Z\in\mathcal{CS}_k(0): \forall \varpi(j)\in Z,\ \varpi(j)\text{ is not $(\delta,\alpha)$-very good}\right]\label{eq:to-fix-tree}\\
      &\leq \mathds{P}\left[\exists T\in\mathcal{T}_k(0): \forall \varpi(j)\in T,\ \varpi(j)\text{ is not $(\delta,\alpha)$-very good}\right]\nonumber\\
      &\leq \sum_{T\in \mathcal{TA}_k(0)} \mathds{P}\left[T\in \mathcal{T}_k(0)\right]\mathds{P}\left[ \forall \varpi(j)\in T,\ \varpi(j)\text{ is not $(\delta,\alpha)$-very good}\ |\text{all edges in $T$ exist}\right]\nonumber.
  \end{align}

  In addition, for each fixed $T\in \mathcal{TA}_k(0)$, since $T$ only contains $k-1$ edges, we can identify at least $3k'=k/2$ vertices $\varpi(i_1),\cdots,\varpi(i_{3k'})$ such that each $\varpi(i_l)$ is directly connected to at most two vertices in $T$.
  Let $W_{i_l}$ be the neighbor(s) of $\varpi(i_l)$ in $T$ for $l\in[1,3k']$. Then, we can select at least $k'$ vertices (also denoted by $\varpi(i_1),\cdots, \varpi(i_{k'})$) such that $W_{i_l}\cap\{\varpi(i_1),\cdots,\varpi(i_{k'})\}=\emptyset$ and we define the event $H_{i_l}=\{\varpi(i_l)\sim \varpi(j)\ \text{for all } \varpi(j)\in W_{i_l}\}$ for $1\leq l\leq k'$.
  Then from Proposition~\ref{P-red-many-strong} we can see that for any $T\in \mathcal{TA}_k(0)$,
  \begin{equation*}
    \begin{aligned}
    &\mathbb{P}\left[ \forall \varpi(j)\in T,\ \varpi(j)\text{ is not $(\delta,\alpha)$-very good}\ |\text{all edges in $T$ exist}\right]\\
    &\leq 2^k\mathbb{P}\left[ \forall l\in[1,k'],\ \varpi(i_l)\text{ is not $(\delta,\alpha)$-very good}|\cap_{l=1}^{k'}H_{i_l}\right]\\
    &\leq \e^{-c_5Mk'}=\e^{-c_5Mk/6},
    \end{aligned}
  \end{equation*}
where $c_5>0$ (depending only on $\beta$) is the constant defined in Proposition~\ref{P-red-many-strong}.
Combining  this with \eqref{eq:to-fix-tree} and Lemma \ref{Baumlerlemma3.1}, we get that
  \begin{equation*}
    \begin{aligned}
      &\mathbb{P}[\exists Z\subset\mathcal{CS}_k(0): \forall \varpi(j)\in Z,\ \varpi(j)\text{ is not $(\delta,\alpha)$-very good}]\\
      &\leq \e^{-c_5Mk/6}\sum_{T\in \mathcal{TA}_k(0)}\mathbb{P}[T\in \mathcal{T}_k(0)]=(4\mu_\beta)^k\e^{-c_5Mk/6},
    \end{aligned}
  \end{equation*}
  which implies the result.
\end{proof}

\subsection{Coarse-graining argument}\label{sect-cg}
Throughout this subsection, we fix sufficiently large $m,n\in \mathds{N}$ and $\delta\in (0,\delta^*]$.
 Recall that we write the interval $[im,(i+1)m)$ in $\mathds{Z}$ as $I_i$, and $G=(V,E)$ is the renormalization from the $\beta$-LRP model by identifying the intervals $[im,(i+1)m)$ to vertices $\varpi(i)$. Denote $V_n=\{\varpi(0),\varpi(1),\cdots, \varpi(n-1)\}$ and $G_n=(V_n,E_n)$ with $E_n:=E_{V_n\times V_n}$.

For simplicity, we will refer to a vertex $\varpi(i)\in V$ as an $\alpha$-\textit{red vertex} if it is not $(\delta,\alpha)$-very good. Otherwise, we will say $\varpi(i)$ is an $\alpha$-\textit{black vertex}. We say an edge $\langle \varpi(i),\varpi(j)\rangle$ is an $\alpha$-\textit{red edge} if both $\varpi(i)$ and $\varpi(j)$ are $\alpha$-red vertices.
Additionally, we will refer to a vertex set $L\subset V_n$ as an $\alpha$-\textit{red animal} if it is a connected subset in $G_n$ consisting only of $\alpha$-red vertices.
Furthermore, if no other $\alpha$-red vertex is directly connected to $L$, we will call $L$ an $\alpha$-\textit{red component}.
Let $\mathcal{L}$ denote the collection of all $\alpha$-red components that contain at least two vertices.

In this subsection we will employ the coarse-graining argument to classify 
$\alpha$-red components at various scales and subsequently identify corresponding good regions at each scale to cover them.
To achieve this, we now introduce some parameters which will be used repeatedly. Let $M$ be a sufficiently large number and define
\begin{equation}\label{def-Alambda}
A=M^{0.2}>1\quad \text{and}\quad \lambda=M^{-0.1}\in(0,1).
\end{equation}
We introduce a sequence of numbers $\{a_k\}_{k\geq 0}$ defined as
\begin{equation}\label{def-ai}
a_0=0,\quad  a_1=A\quad  \text{and} \quad a_k=\lfloor (1+\lambda) a_{k-1}\rfloor+1\quad \text{for all }k\geq 2.
\end{equation}
For convenience, we will denote
\begin{equation}\label{def-bk}
b_k=\exp\left\{M^{0.2}a_k^{1/2}\right\}\quad \text{and}\quad b_{k-1,k}=\exp\left\{M^{0.2}\left(\frac{a_{k-1}+a_k}{2}\right)^{1/2}\right\}   \quad \text{for all } k\geq 1,
\end{equation}
and let
\begin{equation}\label{def-K}
K_{*}=\min\left\{k:\exp\left\{c_6M^{0.25}a_k^{1/2}\right\}>2n^3\right\},
\end{equation}
where $c_6\in (0,1)$ (depending only on $\beta$) is a constant to be specified in Proposition \ref{lem-prob-CF} below.

Throughout this section, we will assume that $M$ is sufficiently large such that
\begin{equation}\label{cond-M1}
 \frac{3}{2(1+M^{-0.1})}\geq 1.49\quad \text{and}\quad  \quad b_{K_{*}}<n^{0.1}.
\end{equation}
Additionally, we will refer to an interval $J\subset V_n$ as an $a_k$-interval if $\#J=b_k$.

\subsubsection{Notations for the coarse-graining argument}\label{notat-cg}
We now introduce some general definitions (not limited to just $\alpha$-red animals) that will be useful in the coarse-graining argument.
For any subset $L\subset V$, we denote $\#L$ as the number of vertices in $L$ and define the degree of $L$ as
\begin{equation*}
\text{deg}(L)=\sum_{\varpi(i)\in L}\#\left\{\varpi(j)\in V_n\setminus L:\ \varpi(j)\sim \varpi(i) \right\}.
\end{equation*}
Let $\mathcal{C}_0$ be the set consisting of all admissible animals in the graph $G$ and let $\mathcal{C}$ be a subset of $\mathcal{C}_0$ (we treat $\mathcal{C}$ here and in Definition \ref{def-ak-clt} as some generic subset, and we will specify $\mathcal{C}$ later).

\begin{definition}\label{def-ak-clt}
For any $k\geq 1$ and any subset $L\subset V$, we say $L$ is an \textit{$a_k$-subset} if
\begin{equation*}
  \begin{aligned}
  &\# L\leq a_k \ \text{and}\ \mathrm{deg}(L)\leq 20\mu_\beta Ma_k,\\
  \text{but} \ &\# L> a_{k-1} \ \text{or}\ \mathrm{deg}(L)> 20\mu_\beta Ma_{k-1},
  \end{aligned}
\end{equation*}
 where $\mu_\beta=\mathds{E}[\text{deg}(\varpi(0))]$ as defined in Lemma \ref{Baumlerlemma3.1}. 
 Then 
  let $\mathcal{C}_{k}$ be the set of all $a_k$-subsets in $\mathcal{C}$.
\end{definition}

A crucial step in our coarse-graining argument is the merging and expansion of the red components. To accomplish this, we introduce the following two merging operations.
The first operation involves merging an animal with its adjacent long edges (see Figure~\ref{fig-LECF}).

\begin{definition}\label{def-LECF}
For any $k\geq 1$ and any animal $L\in \mathcal{C}_k$,  we define the following operation as $a_k$-long-edge animal fusion ($a_k$-LEAF). For any vertex $\varpi(i)\in V_n \setminus L$, if

\begin{itemize}
\item[(1)] $\varpi(i)\sim L$;
\medskip

\item[(2)] there exists a long edge $\langle\varpi(j_1),\varpi(j_2)\rangle \in E_n\setminus E_{L\times L}$ connecting $B_{b_{k-1}}(\varpi(i))$ and $B_{b_{k-1,k}}(\varpi(i))^c$ directly or connecting $B_{b_{k-1,k}}(\varpi(i))$ and $B_{b_{k}}(\varpi(i))^c$ directly,
\end{itemize}
then we color the edge between $\varpi(i)$ and $L$ in dashed red, add a dashed red edge between $\varpi(i)$ and $\varpi(j_1)$, and include both this dashed edge and the edge $\langle \varpi(j_1),\varpi(j_2)\rangle $ in the animal $L$ to form a new animal.
We say that $a_k$-LEAF is applied iteratively to $L$ if, after applying $a_k$-LEAF once to $L$ resulting in a new animal $L'$, we then apply $a_k$-LEAF again to $L'$, and so on multiple iterations  (until no further changes can be made). It is worth emphasizing that we still continue applying $a_k$-LEAF even if $L'$ has become an $a_{l}$-subset for some $l>k$.
It is important to note that this newly formed structure may not constitute a true animal in $G_n$, since it is possible that $\varpi(j_1)\nsim \varpi(i)$.
\end{definition}

\begin{figure}[h]
  \includegraphics[scale=0.8]{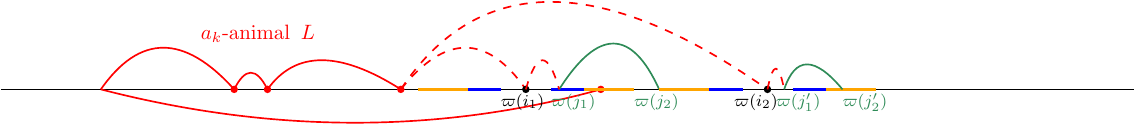}
  \caption{The illustration for Definition~\ref{def-LECF}. The red vertices and curves in the figure represent an $a_k$-animal. The black vertices $\varpi(i_1)$ and $\varpi(i_2)$ are both connected to $L$. The blue lines denote the annulus $B_{b_{k-1,k}}(\varpi(i_l))\setminus B_{b_{k-1}}(\varpi(i_l))$, while the orange lines represent the annulus $B_{b_{k}}(\varpi(i_l))\setminus B_{b_{k-1,k}}(\varpi(i_l))$  for $l=1,2$. The green curves represent the long edges $\langle \varpi(j_1),\varpi(j_2)\rangle$ and $\langle \varpi(j'_1),\varpi(j'_2)\rangle$, which connect $B_{b_{k-1,k}}(\varpi(i_1))$ to $B_{b_{k}}(\varpi(i_1))^c$ and  $B_{b_{k-1}}(\varpi(i_2))$ and $B_{b_{k-1,k}}(\varpi(i_2))^c$, respectively.}
  \label{fig-LECF}
\end{figure}

The second one involves the merging of two animals with the same scale that are close to each other (see Figure~\ref{fig-NNCF}).

\begin{definition}\label{def-NNCF}
  For any $k\geq 1$ and any animal $L\in \mathcal{C}_k$, we define the following operation as
$a_k$-near-neighbor animal fusion ($a_k$-NNAF). For any vertex $\varpi(i)\in V_n\setminus L$, if
\begin{enumerate}
  \item $\varpi(i)\sim L$;
  \medskip

  \item there exist another animal $L'\in\mathcal{C}_{k}$ and a vertex $\varpi(j)\in V_n\setminus L'$ such that $\varpi(j)\sim L' $ and
  $
     |i-j|\leq 2b_k,
  $
\end{enumerate}
then we will color the edge connecting $\varpi(i)$ and $L$ and the edge connecting $\varpi(j)$ and $L'$ in dashed red and then add a dashed red edge between $\varpi(i)$ and $\varpi(j)$. We apply $a_k$-NNAF to each pair of $L,L'\in\mathcal{C}_k$ satisfying the conditions above and add a couple of red dashed edges. Then we consider the subgraph consisting of all edges contained in some $L\in\mathcal{C}_k$ and all red dashed edges and refer to any component of this subgraph as a new animal. 
We also note that this newly formed structure may not constitute a true animal in $G_n$.
\end{definition}

\begin{figure}
  \includegraphics[scale=0.6]{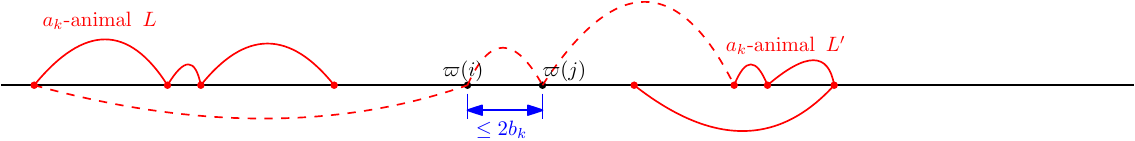}
  \caption{The illustration for Definition \ref{def-NNCF}. The red curves represent the two animals close to each other. The red dashed line represents the dashed red edge connecting them that we add.}
  \label{fig-NNCF}
\end{figure}

To control the energy generated by the flows passing through animals in $\mathcal{C}$, we also need to consider the resistances associated with the intervals adjacent to these animals. To address this, for any set $L\subset V_n$, denote
\begin{equation}\label{def-BL}
B_r(L)=\bigcup_{\varpi(i)\in V_n\setminus L: \varpi(i)\sim L}B_{r}(\varpi(i)).
\end{equation}
We then introduce the following condition for being bad.
\begin{definition}\label{def-a1bad}
  For any $k\geq 1$ and any animal $L\in \mathcal{C}_k$, we say that $L$ is \textit{$a_k$-bad} if
  \begin{equation}\label{eff-ik}
  R\left(B_{b_{k-1,k}}(L), B_{b_k}(L)^c;E_n\setminus E_{L\times L}\right)\leq c_*(20\mu_\beta Ma_k)^{-1}\exp\left\{\delta M^{0.2}a_k^{1/2}/2\right\},
  \end{equation}
  where $R(\cdot,\cdot;E_n\setminus E_{L\times L})$ denotes the effective resistance in $G_n$ after removing the edge set $E_{L\times L}$,
  $\delta\in (0,\delta^*]$ (being fixed) and $c_*,\ \delta^*$ (both depending only on $\beta$) are the constants defined in Lemma \ref{Lem-RNtail-1}.
\end{definition}

\subsubsection{Construction of a covering}\label{T-T}
We now return to considering the $\alpha$-red component set $\mathcal{L}$.
For simplicity, we will omit the notation $\alpha$ in this subsection and simply refer to the $\alpha$-red components and $\alpha$-black vertices as red components and black vertices, respectively.
For $k\geq 1$, recall that $\mathcal{L}_{k}$ is the collection of red components in $V_n$ that are also $a_k$-subsets, which is defined by replacing $\mathcal{C}$ with $\mathcal{L}$ in Definition \ref{def-ak-clt}.
Our goal is to expand the red components using the two merge operations and the resistance condition introduced in Section \ref{notat-cg}, and then establish the associated estimates to complete the coarse-graining argument.

We begin by defining the expansion of red components inductively as follows.
\begin{definition}\label{def-expansion}
Let $\widetilde{\mathcal{L}}_1=\mathcal{L}_{1}$. For $k\geq 1$, assuming that $\widetilde{\mathcal L}_1,\cdots,\widetilde{\mathcal{L}}_k$ have been defined, we next inductively define $\widetilde{\mathcal L}_{k+1}$ via the following steps.
\begin{enumerate}
  \item[(1)] Now for any pair of $L,L'\in\widetilde{\mathcal{L}}_k$ satisfying the conditions in Definition~\ref{def-NNCF}, we apply $a_k$-NNAF to them and add a few red dashed edges. 
      For each $l>k$, let $\widetilde{\mathcal{L}}_{k\to l}^{(1)}$ denote the set of animals $L$ such that $L$ is an $a_l$-subset and also a red component with respect to the subgraph containing all red and dashed red edges that have been added until now. 
      We also define $\widetilde{\mathcal{L}}_{k\to k}^{(1)}$ as the set of components with respect to the same subgraph that contain at least one animal in $\widetilde{\mathcal{L}}_{k}$ and are not included in $\widetilde{\mathcal{L}}_{k\to l}^{(1)}$ for all $l>k$.

  \item[(2)] Next, we expand animals in $\widetilde{\mathcal{L}}_{k\to k}^{(1)}$ using the $a_k$-LEAF, adding a couple of red dashed edges in each iteration. During this process, if two animals in  $\widetilde{\mathcal{L}}_{k\to k}^{(1)}$ add the same red dashed long edge from the $a_k$-LEAF, they will be merged together. (See Figure~\ref{LEAF-same-edge} for an illustration.)
  \begin{figure}[h]
    \includegraphics[scale=0.7]{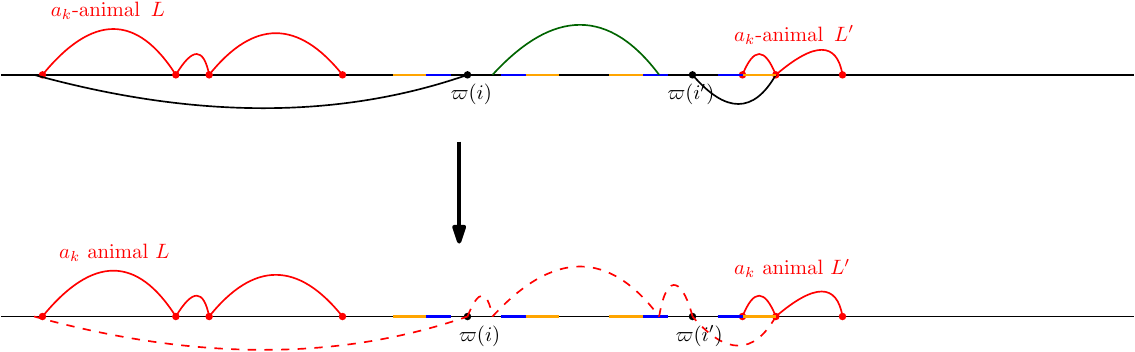}
    \caption{The red curves represent two $a_k$-animals, $L$ and $L'$. In the top figure, the green curve represents a long edge that we can add to both $L$ and $L'$ according to the $a_k$-LEAF operation. In the bottom figure, we illustrate the process of coloring the green long edge and the two black edges as red dashed curves, while also adding red dashed edges between them, effectively merging $L$ and $L'$ into a larger animal.}
    \label{LEAF-same-edge}
  \end{figure}
      This process is applied to all animals $L$ in $\widetilde{\mathcal{L}}_{k\to k}^{(1)}$.
      For each $l>k$, let $\widetilde{\mathcal{L}}_{k\to l}^{(2)}$ denote the set of animals $L$ such that $L$ is an $a_l$-subset and also a red component with respect to the subgraph containing all red and dashed red edges that have been added until now. 
      We also define $\widetilde{\mathcal{L}}_{k\to k}^{(2)}$ as the set of components with respect to the same subgraph that contain at least one animal in $\widetilde{\mathcal{L}}_{k\to k}^{(1)}$ and are not included in $\widetilde{\mathcal{L}}_{k\to l}^{(2)}$ for all $l>k$. 

  \item[(3)] Let $\widetilde{\mathcal{L}}^{(2)}_{k\to k,\mathrm{bad}}$ denote the set of all $a_k$-bad animals in $\widetilde{\mathcal{L}}^{(2)}_{k\to k}$ and define $\mathcal{G}_k=\widetilde{\mathcal{L}}^{(2)}_{k\to k}\setminus \widetilde{\mathcal{L}}^{(2)}_{k\to k,\mathrm{bad}}$.

  \item[(4)] Finally, we define
  \begin{equation}\label{L0ak}
  \widetilde{\mathcal{L}}_{k+1}=\mathcal{L}_{k+1}\bigcup\left(\bigcup_{j\leq k}\widetilde{\mathcal{L}}^{(1)}_{j\to (k+1)}\right)\bigcup\left(\bigcup_{j\leq k}\widetilde{\mathcal{L}}^{(2)}_{j\to (k+1)}\right)\bigcup \widetilde{\mathcal{L}}^{(2)}_{k\to k,\mathrm{bad}}.
  \end{equation}
  We will refer to the animals in $\widetilde{\mathcal{L}}_{k+1}$ as red $a_{k+1}$-animals (although a red animal may contain black vertices, from our procedure we can see that all the edges in these animals are red or dashed red).
  Note that an $a_{k+1}$-animal is not necessarily an $a_{k+1}$-subset as in Definition \ref{def-ak-clt}, since any $L \in \widetilde{\mathcal L}^{(2)}_{k\to k}$ is not an $a_{k+1}$-subset.
\end{enumerate}
\end{definition}
In addition, recall that the parameter $K_{*}$ is defined in \eqref{def-K}. Let $\mathcal{H}$ be the event that $\bigcup_{k>K_{*}}\widetilde{\mathcal{L}}_{k}=\emptyset$.

\subsubsection{Analysis of the covering}
The main focus of this subsection is to bound the probability of $\mathcal{H}$ as follows.

\begin{proposition}\label{prop-eventH}
For $\beta>0$ and sufficiently large $M\in \mathds{N}$ satisfying \eqref{cond-M1},
we have $\mathds{P}[\mathcal{H}]\geq1- n^{-2}$ for all $n\geq 1$.
\end{proposition}

The main input of the proof of Proposition \ref{prop-eventH} is an estimate of the probability that an $a_k$-interval is connected to some red $a_k$-animal, as defined below. For $k,s\in \mathds{N}$, let
\begin{equation}\label{p0k}
  \begin{aligned}
    p_{k}(s)&=\max_{a_k\text{-interval }J\subset V_n}\mathbb{P}\left[\text{there exists $L\in \widetilde{\mathcal{L}}_{k}$ such that }\# L=s\ \text{and }J\sim L\right],\\
    q_{k}(s)&=\max_{a_k\text{-interval }J\subset V_n}\mathbb{P}\left[\text{there exists $L\in \widetilde{\mathcal{G}}_{k}$ such that }\# L=s\ \text{and }J\sim L\right], \\
    \widetilde{q}_{k}(s)&=\max_{a_k\text{-interval }J\subset V_n}\mathbb{P}\left[\text{there exists $L\in \widetilde{L}_{k\to k}^{(2)}$ such that }\# L=s\ \text{and }J\sim L\right],
  \end{aligned}
\end{equation}
where $J\sim L$ means that there exists $\varpi(i)\in J\cap (V_n\setminus L)$ such that $\varpi(i)\sim L$.
Then we have the following estimates.

\begin{proposition}\label{lem-prob-CF}
  For $\beta>0$ and sufficiently large $M\in \mathds{N}$ satisfying \eqref{cond-M1}, 
  there exists a constant $c_6>0$ {\rm(}depending only on $\beta${\rm)} such that for all $k,s\in\mathds{N}$,
    \begin{equation}\label{est-pk}
      p_{k}(s)\leq \exp\left\{-c_{6}M^{0.25}\max\left\{a_{k-1}^{1/2},sa_{k-1}^{-1/2},A^{1/2}\right\}\right\}
      \end{equation}
      and
      \begin{equation}\label{est-qk}
      q_k(s)\leq \widetilde{q}_k(s) \leq 1600\mu_\beta Ma_k \exp\left\{-c_6 M^{0.25}\max\left\{a_{k-1}^{1/2},sa_{k-1}^{-1/2},A^{1/2}\right\}\right\}.
    \end{equation}
\end{proposition}

 The proof of Proposition \ref{lem-prob-CF} relies on induction. To this end, we make some preparations. Specifically, we begin by considering the probability for merging multiple small red animals into a large red animal via NNAF.

\begin{lemma}\label{est-j-k-explore}
  For $\beta>0$, sufficiently large $M\in \mathds{N}$ and $\varpi(v)\in V_n$, the following holds.
  For any $K,s,k,j\in \mathds{N}$ with $j\leq k$, the probability that $\varpi(v)$ is connected to a red $a_k$-animal $L\in\widetilde{\mathcal{L}}_{j\to k}^{(1)}$ which contains $s$ vertices and consists of $K$ red $a_j$-animals in $\widetilde{\mathcal{L}}_{j}$, is at most
  $$
  (400\mu_\beta Ma_j)^K\sum_{x_1+\cdots+x_K=s, x_i\leq a_j}\prod_{i=1}^K p_{j}(x_i).
  $$
\end{lemma}

\begin{proof}
  From the definition of $\widetilde{\mathcal{L}}^{(1)}_{j\to k}$ in Definition \ref{def-expansion}, it is clear that if $\varpi(v)\sim L$ for some $L\in\widetilde{\mathcal{L}}^{(1)}_{j\to k}$, which is composed of $K$ red $a_j$-animals in $\widetilde{\mathcal{L}}_{j}$ (note that these $a_j$-animals are connected by red dashed long edges, see the illustration in Figure \ref{LEAF-same-edge}), then we can find $K$ $a_j$-intervals $J_1,J_2,\cdots, J_K$ such that $\varpi(v)\in J_1$, and the following two conditions hold:
  \begin{enumerate}
    \item[(1)] For each $ l\in[1,K]$, $J_l$ is connected to some red $a_j$-animal in $\widetilde{\mathcal{L}}_{j}$.
    \medskip

    \item[(2)] For each $ l\in[1,K]$, there exists $l'\neq l$ such that the red $ a_j $-animals in $ \widetilde{\mathcal{L}}_j $ connected to $ J_l $ and $ J_{l'} $ can be merged together via $ a_j $-NNAF. In this case, we say that $ J_l $ and $ J_{l'} $ are \textit{animal neighbors}. In addition, any two intervals $J_l$ and $J_{l'}$ are connected by a sequence of intervals where the two neighboring intervals in the sequence are animal neighbors.
  \end{enumerate}

Based on these observations, we now employ a ``depth-first'' exploration process $(J'_i)_{i=1}^{2K-1}$ (see Figure \ref{explore} for an illustration) to encode a red $a_k$-animal  $L\in\widetilde{\mathcal{L}}^{(1)}_{j\to k}$, which consists of $K$ red $a_j$-animals in $\widetilde{\mathcal{L}}_{j}$.
Specifically, we define an auxiliary sequence $\textbf{w}=({\rm w}_i)_{i=1}^{2K-2}\in\{{\rm u},{\rm d}\}^{2K-2}$ as follows (here ``u'' means up and ``d'' means down):
  \begin{enumerate}
    \item[(i)] Start with $J'_1 = J_1$.
    \medskip

    \item[(ii)] For $i=2,\cdots,2K-1$, assume that $J_1',\cdots, J_{i-1}'$ have been defined.
      \begin{enumerate}
        \item[(a)] If there exists a $J_{i'}\notin\{J_1',\cdots J_{i-1}'\}$ such that $J_{i'}$ is an animal neighbor of $J'_{i-1}$,
        we select the leftmost $J_{i'}$ among them and set $J_{i}'=J_{i'}$.
        We then say that $J_{i}'$ is a child of $J_{i-1}'$ while $J_{i-1}'$ serves as the parent of $J_i'$. We set ${\rm w}_{i-1}=\rm d$.
        \item[(b)] Otherwise, we set $J_i'$ as the parent of $J_{i-1}'$ and let ${\rm w}_{i-1}=\rm u$.
      \end{enumerate}
  \end{enumerate}
  \begin{figure}[h]
    \includegraphics[scale=0.7]{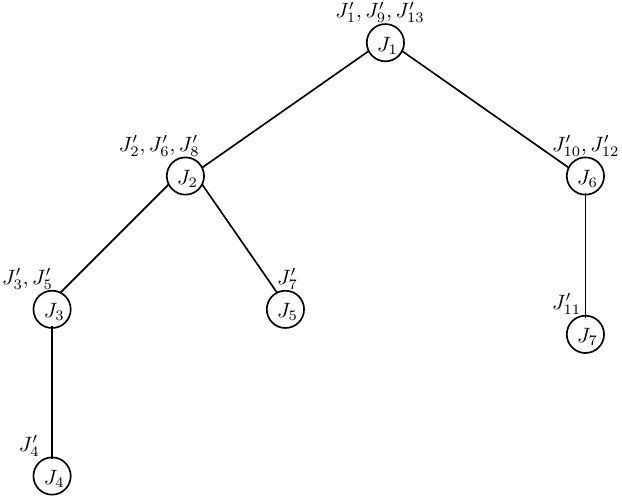}
    \caption{The illustration for the ``depth-first'' exploration process. In the figure, the process is generated by animal neighbor $\{J_i\}_{i=1}^7$, which are marked within the circles. the corresponding exploration sequence is represented by $(J'_i)_{i=1}^{13}$. For this exploration, we have $\textbf{w}=(\rm d,d,d,u,u,d,u,u,d,d,u,u)$.}
    \label{explore}
  \end{figure}

  We now consider the selection of $J_1,\cdots, J_K$ when $\textbf{w}$ is fixed.
  Our goal is to bound the number of ``trees'' formed by all $J_i$'s containing the fixed $J_1$, each having this specific up-and-down structure with respect to the animal neighboring relations. 

  Suppose the exploration process visits exactly $l$ children of some interval $J_i$.
  For convenience, we denote $L_i$ as the red $a_j$-animal connected to the interval $J_i$. According to the definition of $a_j$-animal in Definition \ref{def-ak-clt}, it follows that $\text{deg}(L_i)\leq 20\mu_\beta Ma_j$. Together with Definition \ref{def-NNCF} (1), this implies that $J_i$ has at most $5\times 20\mu_\beta Ma_j$ candidate children (a candidate child is an animal that may be chosen as a child in our exploration process). Hence, the number of ways to choose $l$ children from the candidate children of $J_i$ in an increasing order is dominated by $\left(100\mu_\beta Ma_j\right)^l$.
  This bound applies to all intervals $J_i$.
  In addition, note that the total number of children across all intervals is $K-1$, as each interval, except $J_1$, is the child of exactly one other interval.
  Therefore, the number of ``trees'' with a specified up-and-down structure can be bounded from above by $(100\mu_\beta Ma_j)^{K-1}$.

  Up to now, we have considered a fixed up-and-down structure. However, there are at most $4^K$ possible up-and-down structures $\textbf{w}$ (in fact, there are significantly fewer, since there are additional constraints, such as ${\rm u}_1 = \rm d$). Additionally, since all red animals and the exploration process among all $J_i$'s do not share the same edge, there is independence and we can multiply probabilities arising from $a_j$-animals in different intervals. Thus for fixed $x_1,\cdots,x_K$, the probability that $\varpi(v)$ is connected to a red $a_k$-animal $L\in\widetilde{\mathcal{L}}_{j\to k}^{(1)}$ which  consists of $K$ red $a_j$-animals $L_1,\cdots,L_K\in\widetilde{\mathcal{L}}_{j}$ with $\# L_i=x_i$ for each $i\in[1,K]$
  is dominated by
  $$
  4^K\cdot(100\mu_\beta Ma_j)^{K-1}\prod_{i=1}^K p_j(x_i)\leq  (400\mu_\beta Ma_j)^{K}\prod_{i=1}^K p_j(x_i).
  $$
  We now take a union bound over the possible number of vertices contained in each animal connected to $J_i$. This implies that the probability that  $\varpi(v)$ is connected to a red $a_k$-animal $L\in\widetilde{\mathcal{L}}_{j\to k}^{(1)}$ which contains $s$ vertices and consists of $K$ red $a_j$-animals in $\widetilde{\mathcal{L}}_{j}$,
  is dominated by
  $$
  (400\mu_\beta Ma_j)^{K}\sum_{x_1+\cdots+x_K=s}\prod_{i=1}^K p_j(x_i) .
  $$
  Thus, we complete the proof.
 \end{proof}

Using Lemma~\ref{est-j-k-explore}, we having the following estimate.

\begin{lemma}\label{prob-total-explore}
For any $\beta>0$, $k\geq 1$ and sufficiently large $M\in \mathds{N}$ satisfying \eqref{cond-M1}, 
assume that there exists a constant $c_6>0$ {\rm(}depending only on $\beta${\rm)} such that \eqref{est-pk} holds for all $j<k$ and all $s\leq a_j$. Then, for each $a_{k-1}<s\leq a_k$ and each $a_k$-interval $J$, we have
  \begin{equation}\label{eq-prob-total-explore}
    \mathbb{P}\left[\text{there exists $L\in \bigcup_{j<k}\widetilde{\mathcal{L}}_{j\to k}^{(1)}$ such that }\# L=s\ \text{and }J\sim L\right]\leq \frac{1}{4}\exp\left\{-c_{6}M^{0.25}sa_{k-1}^{-1/2}\right\}.
  \end{equation}
\end{lemma}
\begin{proof}
For fixed $k\geq 1$, recall that we assume \eqref{est-pk} holds for all $j<k$ and all $s\leq a_j$ . By combining this with Lemma~\ref{est-j-k-explore} and taking a union bound, we get that
    \begin{align}
      &\text{LHS of }\eqref{eq-prob-total-explore}\nonumber\\
      &\leq \sum_{j<k}b_k\sum_{K}\sum_{x_1+\cdots+x_K=s,x_i\leq a_j}(400\mu_\beta Ma_j)^K\prod_{i=1}^K p_j(x_i)\nonumber\\
      &\leq \sum_{j<k} b_k \sum_{K\geq s/a_j} \sum_{x_1+\cdots+x_K=s, x_i\leq a_j}(400\mu_\beta Ma_j)^K \exp\left\{-c_{6}KM^{0.25}\max\left\{a_{j-1},A\right\}^{1/2}\right\}\label{L2union}\\
      &\leq \sum_{j<k} b_k \sum_{K\geq s/a_j} \left(400\mu_\beta Ma_j^2\right)^K \exp\left\{-c_{6}KM^{0.25}\max\left\{a_{j-1},A\right\}^{1/2}\right\}\nonumber\\
      &\leq 2\sum_{j<k} b_k \left(400\mu_\beta Ma_j^2\exp\left\{-c_{6}M^{0.25}\max\left\{a_{j-1},A\right\}^{1/2}\right\}\right)^{\lfloor s/a_j\rfloor+1}.\nonumber
    \end{align}
Here, the term $b_k$ in the first line comes from taking a union bound over all $\varpi(v)\in J$ and the third inequality arises from the bound on the number of choices of $(x_1,\cdots,x_K)$.

In the following, we will estimate each term in the summation on the RHS of \eqref{L2union}. For $j<k$, we first consider the case where $a_{k-1}/a_j>3/2$.
It follows from the definition of $a_k$ in \eqref{def-ai} and the fact $a_{k-1}<s\leq a_k$ that
\begin{equation}\label{akak-1s}
\sqrt{a_{k-1}a_k}/s \leq 2a_{k-1}/s\leq 2.
\end{equation}
Combining this with the definition of $b_k$ in \eqref{def-bk} and the fact that $M$ is sufficiently large, we have
  \begin{equation}\label{case1}
    \begin{aligned}
      &b_k \left(400\mu_\beta Ma_j^2\exp\left\{-c_6M^{0.25}\max\left\{a_{j-1},A\right\}^{1/2}\right\}\right)^{\lfloor s/a_j\rfloor+1}\\
      &\leq \exp\left\{M^{0.2}a_k^{1/2}+sa_j^{-1}\left(-c_6 M^{0.25}\max\{a_{j-1},A\}^{1/2}+\ln\left(400\mu_\beta Ma_j^2\right)\right)\right\}\\
      &\leq \exp\left\{M^{0.2}a_k^{1/2}-0.99c_6 sa_j^{-1}M^{0.25}\max\{a_{j-1},A\}^{1/2}\right\}\\
      &= \exp\left\{-c_6 M^{0.25}sa_{k-1}^{-1/2}\left(0.99\sqrt{a_{k-1}a_{j-1}}/a_{j}-2c_6 ^{-1}M^{-0.05}\right)\right\}\quad \quad \text{by \eqref{akak-1s}}\\
      &\leq \exp\left\{-1.1c_6 M^{0.25}sa_{k-1}^{-1/2}\right\}\quad \quad \text{by }a_{k-1}/a_j>3/2.
    \end{aligned}
  \end{equation}

Additionally, for the case $a_{k-1}/a_j\leq 3/2$, since $\lambda=M^{-0.1}$ (see \eqref{def-Alambda}) and $a_{k-1}<s\leq a_k$, we can find that for sufficiently large $M$,
$$
1<s/a_j\leq a_k/a_j=\frac{\lfloor (1+\lambda)a_{k-1}\rfloor+1 }{a_j}<2.
$$
Thus, we have $\lfloor s/a_j\rfloor +1=2$. From this we obtain
  \begin{equation}\label{case2}
    \begin{aligned}
      &b_k \left(400\mu_\beta Ma_j^2\exp\left\{-c_6 M^{0.25}\max\{a_{j-1},A\}^{1/2}\right\}\right)^{\lfloor s/a_j\rfloor+1}\\
      &\leq \exp\left\{M^{0.2}a_k^{1/2}+2\left(-c_6 M^{0.25}\max\left\{a_{j-1},A\right\}^{1/2}+\ln\left(400\mu_\beta Ma_j^2\right)\right)\right\}\\
      &\leq \exp\left\{-c_6 M^{0.25}sa_{k-1}^{-1/2}\left(1.98\sqrt{a_{k-1}a_{j-1}/s}-2c_6 ^{-1}M^{-0.05}\right)\right\}\\
      &\leq \exp\left\{-1.1c_6 M^{0.25}sa_{k-1}^{-1/2}\right\}.
    \end{aligned}
  \end{equation}
  Applying \eqref{case1}, \eqref{case2} to \eqref{L2union} yields that for sufficiently large $M$,
  \begin{equation*}
    \text{LHS of }\eqref{eq-prob-total-explore}\leq 2k\exp\left\{-1.1 c_6 M^{0.25}sa_{k-1}^{-1/2}\right\}\leq \frac{1}{4}\exp\left\{-c_6 M^{0.25}sa_{k-1}^{-1/2}\right\},
  \end{equation*}
  which completes the proof.
\end{proof}

The following is a modified version of Proposition \ref{P-red-many-tree}.

\begin{lemma}\label{prob-1-bad}
For $\beta>0$ and sufficiently large $M\in \mathds{N}$, there exists a constant $c_7>0$ {\rm(}depending only on $\beta${\rm)} such that for all $k,s\in\mathbb{N}$ and all $\varpi(v)\in V_n$,
  \begin{equation}\label{P-Lk-s}
  \mathbb{P}\left[\text{there exists $L\in \mathcal{L}_{k}$ such that }\# L=s\ \text{and }\varpi(v)\sim L\right]\leq {\rm e}^{-c_7Ma_k}.
  \end{equation}
\end{lemma}
\begin{proof}
For fixed $k\in\mathds{N}$, recall that $\mathcal{L}_k$ is defined in Definition \ref{def-ak-clt} with replacing $\mathcal{C}$ by $\mathcal{L}$.
   We first consider the case where $k>1$ and $s\leq a_{k-1}$. According to Definition \ref{def-ak-clt}, we see that  $\mathrm{deg}(L)\geq 20\mu_\beta Ma_{k-1}$.
   Then using the similar arguments as in the proof of \cite[Lemma 3.2]{Baumler23a}, we can obtain that for any $\varpi(v)\in V_n$,
  \begin{equation}\label{prob-L2}
    \begin{aligned}
      &\mathds{P}\left[\text{there exists $L\in \mathcal{L}_{k}$ such that }\#L=s,\ \varpi(v)\sim L\ \text{and } \mathrm{deg}(L)\geq 20\mu_\beta Ma_{k-1}\right]\\
      &\leq \mathds{P}\left[\exists Z\in \mathcal{CS}_{s}(\varpi(v)):\ \text{deg}(Z)\geq 20\mu_\beta M a_{k-1}\right]\\
      &\leq (4\mu_\beta)^{s}\e^{-6\mu_\beta Ma_{k-1}}\leq \e^{-\widetilde{c}_1Ma_{k}}
    \end{aligned}
  \end{equation}
  for some $\widetilde{c}_1>0$ that depends only on $\beta$. Here the last inequality follows from the fact that $a_k=\lfloor (1+\lambda)a_{k-1}\rfloor+1$ for $\lambda\in (0,1)$  (see \eqref{def-ai}).

   Moreover, for the case  $s>a_{k-1}$ with $k\geq 1$, it follows from the definition of $\mathcal{L}_k$ that when $s>a_k$, the probability on the LHS of \eqref{P-Lk-s} is zero.
   Therefore, it suffices to consider the case where $s\in (a_{k-1},a_k]$. From Proposition~\ref{P-red-many-tree}, there exists a constant $\widetilde{c}_2>0$ (depending only on $\beta$) such that
  \begin{equation*}\label{prob-L1}
  \begin{aligned}
    &\mathds{P}\left[\text{there exists $L\in \mathcal{L}_{k}$ such that }\# L=s\ \text{and }\varpi(v)\sim L \right]\\
    &\leq \e^{-c_1Ms}\leq \e^{-\widetilde{c}_2Ma_k},
    \end{aligned}
  \end{equation*}
   where $c_1>0$ (depending only on $\beta$) is the constant defined  in Proposition~\ref{P-red-many-tree}.
\end{proof}

We also need the following estimate for the probability of $a_k$-LEAF.

\begin{lemma}\label{p-j-k-2}
For $\beta>0$, $k\geq 1$, $s'\in (a_{k-1}, a_k]$ and sufficiently large $M\in \mathds{N}$ satisfying \eqref{cond-M1}, 
assume that there exists a constant $c_6>0$ {\rm(}depending only on $\beta${\rm)} such that \eqref{est-pk} holds for all $j<k$ and all $s\leq a_j$, as well as holds for $j=k$ and $s<s'$. Then for each $j\leq k$, the probability that an $a_k$-interval $J$ is connected to an animal $L\in \widetilde{\mathcal{L}}_{j\to k}^{(2)}$ with $\# L=s'$ is at most
  \begin{equation*}\label{eq-L2}
    \begin{aligned}
    1600\mu_\beta Ma_k &\exp\left\{-c_6 M^{0.25} \max\left\{a_{k-1}^{1/2},s'a_{k-1}^{-1/2},A^{1/2}\right\}\right\}\I_{\{j=k\}}\\
    &+2b_k/b_j \exp\left\{-\frac{1}{32} M^{0.1}\sqrt{a_{j-1}}-c_6 M^{0.25}a_{j-1}^{-1/2}s'\right\}\I_{\{j<k\}}.
    \end{aligned}
  \end{equation*}
\end{lemma}
Before moving on to the proof, we would like to add the following remark.
\begin{remark}\label{rem-j=k}
It is worth emphasizing that the conclusion of Lemma \ref{p-j-k-2} also holds for the case where $j=k$ and $s'\leq a_{k-1}$. This is because, during the proof of this lemma (particularly in the proof of Lemma \ref{p-j-k-2-step1}), when $j=k$, we did not rely on the restriction $s'\in (a_{k-1},a_k]$.
\end{remark}

To prove Lemma \ref{p-j-k-2}, we first need to estimate the probability arising from $\widetilde{\mathcal{L}}^{(1)}_{k\to k}$.
\begin{lemma}\label{p-L1-k-k}
  For $\beta>0$, $k\geq 1$, $s'\in (0, a_k]$ and sufficiently large $M\in \mathds{N}$ satisfying \eqref{cond-M1}, 
  assume that there exists a constant $c_6>0$ {\rm(}depending only on $\beta${\rm)} such that \eqref{est-pk} holds for all $j<k$ and all $s\leq a_j$, as well as holds for $j=k$ and $s<s'$. Then  the probability that an $a_k$-interval $J$ is connected to an animal $L\in \widetilde{\mathcal{L}}_{k\to k}^{(1)}$ with $\# L=s'$ is at most
  \begin{equation}\label{eq-L1-k-k}
    p_k^1(s'):=800\mu_\beta Ma_k \exp\left\{-c_6 M^{0.25} \max\left\{a_{k-1}^{1/2},s'a_{k-1}^{-1/2},A^{1/2}\right\}\right\}.
  \end{equation}
\end{lemma}
\begin{proof}
  From Lemma~\ref{est-j-k-explore} with $j=k$ and the assumption for \eqref{est-pk} we get that the desired probability can be bounded above by
  \begin{equation*}
    \begin{aligned}
      & 400\mu_\beta Ma_k p_k(s')+\sum_{K\geq 2}\sum_{x_1+\cdots+x_K=s'}\prod_{i=1}^K \left(400\mu_\beta Ma_k\right)^Kp_k(x_i) \\
      &\leq 400\mu_\beta Ma_k p_k(s')+\sum_{K\geq 2}\left(400\mu_\beta Ma_k^2\exp\left\{-c_6 M^{0.25}\max\left\{a_{k-1}^{1/2},A^{1/2}\right\}\right\}\right)^K\\
      &\leq 800 \mu_\beta Ma_k \exp\left\{-c_6 M^{0.25} \max\left\{a_{k-1}^{1/2},s'a_{k-1}^{-1/2},A^{1/2}\right\}\right\},
    \end{aligned}
  \end{equation*}
  where the last inequality follows from the fact that $2\max\{a_{k-1}^{1/2},A^{1/2}\}> \max\{a_{k-1}^{1/2}, s'a_{k-1}^{-1/2},A^{1/2}\}$ by $s'\leq a_k$ and the definition of $a_k$ in \eqref{def-ai}.
\end{proof}

 We also need the following lemma regarding to animals in $\widetilde{\mathcal{L}}_{j\to k}^{(2)}$.
\begin{lemma}\label{p-j-k-2-step1}
  Under the same assumption in Lemma~\ref{p-j-k-2}, for each $j\leq k$, the probability that an $a_k$-interval $J$ is connected to some animal $L\in \widetilde{\mathcal{L}}_{j\to k}^{(2)}$ with $\# L=s'$ is at most
  \begin{equation}\label{eq-L2-start}
    \begin{aligned}
      & (b_k/b_j)\cdot
      \sum_{K\geq 1}\sum_{s'-s\leq 3(K-1)}\sum_{x_1+\cdots+x_K=s,0\leq x_i\leq a_j} \left(80\mu_\beta^2 Ma_j(b_{j-1}/b_{j-1,j}+b_{j-1,j}/b_j)\right)^{K-1}\prod_{i=1}^K p_j^1(x_i),
    \end{aligned}
  \end{equation}
   where $p_k^1(0):=1$, $p_k^1(x_i)$ for $x_i>0$ is defined in \eqref{eq-L1-k-k}.
\end{lemma}
\begin{proof}
  For $j\leq k$, recall that $\widetilde{\mathcal{L}}_{j\to k}^{(2)}$ is defined in Definition \ref{def-expansion}. For any two $a_j$-intervals $J_1,J_2$, we say $J_1,J_2$ are animal-LE neighbors if
  $$
  J_i\bigcup\left(\bigcup_{L\in\widetilde{\mathcal{L}}^{(1)}_{j\to j}:L\sim J_i}L\right),\quad i=1,2
  $$
  can be connected to each other by a long edge satisfying Definition~\ref{def-LECF}~(2).

  Now let $J$ be an $a_k$-interval. According to the definition of $\widetilde{\mathcal{L}}_{j\to k}^{(2)}$ in Definition \ref{def-expansion}, it can be observed that if an $a_k$-interval $J$ is connected to an animal $L\in \widetilde{\mathcal{L}}_{j\to k}^{(2)}$ with $\# L=s'>a_{k-1}$, then there exists $K\in \mathds{N}$ and $K$ $a_j$-intervals $J_1,\cdots, J_K$ with $J_1\subset J$ such that $\{J_1, \ldots, J_K\}$ forms a connected graph via the animal-LE neighboring relation. 

   For each $l\in[1,K]$, if there is at least one red animal $L_l\in \widetilde{\mathcal{L}}^{(1)}_{j\rightarrow j}$ such that $L_l\sim J_l$, then there must be only one such red animal $L_l$ from the definition of $a_j$-NNAF and we denote $x_l=\#L_l$. If no such red animal exists, we set $x_l=0$.
Note that by the definition of $\widetilde{\mathcal{L}}_{j\to k}^{(2)}$, one has $L_l\subset L$ for each $l\in [1,K]$. Therefore, we have $s:=\sum_{i=1}^K x_i\leq s'$. Moreover, we have 
\begin{equation}\label{K-and-xi}
  3(K-1)\geq s'-\sum_{i=1}^K x_i, 
\end{equation}
since we have performed $a_j$-LEAF operation at most $K-1$ times and in each operation we will add at most 3 vertices to the animal. Moreover we have $0\leq x_i\leq a_j$ for each $i\in[1,l]$ from the definition of $\widetilde{\mathcal{L}}_{j\to j}^{(1)}$.

Building on the observation above, similar to the proof of Lemma~\ref{est-j-k-explore},  we can  employ a ``depth-first'' exploration process $(J'_i)_{i=1}^{2K-1}$ to encode an $a_k$-animal $L\in\widetilde{\mathcal{L}}_{j\to k}^{(2)}$ with the aid of an auxiliary sequence $\textbf{w}=({\rm w}_i)^{2K-2}_{i=1}\in\{{\rm u,d}\}^{2K-2}$ as follows:
\begin{enumerate}
  \item[(i)] Start with $J'_1 = J_1$.
\medskip

  \item[(ii)] For $i=2,\cdots,2K-1$, assume that $J_1',\cdots, J_{i-1}'$ have been defined.
    \begin{enumerate}
      \item[(a)] If there exists a $J_{i'}\notin\{J_1',\cdots J_{i-1}'\}$ such that $J_{i'}$ is an animal-LE neighbor of $J'_{i-1}$,
      we select the leftmost $J_{i'}$ among them and set $J_{i}'=J_{i'}$.
      We then say that $J_{i}'$ is a child of $J_{i-1}'$ while $J_{i-1}'$ serves as the parent of $J_i'$. We set ${\rm w}_{i-1}=\rm d$.
      \item[(b)] Otherwise, we set $J_i'$ as the parent of $J_{i-1}'$ and let ${\rm w}_{i-1}=\rm u$.
    \end{enumerate}
\end{enumerate}

We now consider the selection of $J_1,\cdots,J_K$ when the sequence $\textbf{w}$ is fixed. Our goal is to provide an upper bound on the number of ``trees'' on $\{J_1, \ldots, J_K\}$ with $J_1\subset J$ fixed and the specific up-and-down structure with respect to the connections generated by animal-LE neighbors.

Suppose the exploration process visits exactly $l_i$ children, denoted by $J'_1,\cdots,J'_{l_i}$, of some interval $J_i$. Let $L_i$ be the red animal in $\widetilde{\mathcal{L}}_{j\rightarrow j}^{(1)}$ such that $L_i\sim J_i$. Define
\begin{equation*}
\Omega_i=\bigcup_{\text{black }\varpi(j)\sim L_i} B_{b_j}(\varpi(j)).
\end{equation*}
Note that if no such $L_i$ exists, let $\Omega_i=J_i$.
Then the expected number of ways to choose these $l_i$ children is at most
\begin{equation*}
  \begin{aligned}
    &\sum_{J'_1,\cdots,J'_{l_i}}\prod_{i'=1}^{l_i}\mathbb{P}\left[J'_{i'}\sim \Omega_i\text{ by a long edge in Definition~\ref{def-LECF}}\right]\\
    &\leq \left(\mathbb{E}\left[\#\{J':J'\sim \Omega_i\text{ by a long edge in Definition~\ref{def-LECF}}\}\right]\right)^{l_i}\\
    &\leq \left(20\mu_\beta Ma_j\mathbb{E}\left[\#\{J':J'\sim J_i\text{ by a long edge in Definition~\ref{def-LECF}}\}\right]\right)^{l_i}\quad \text{by }\text{deg}(L_i)\leq 20\mu_\beta Ma_j\\
    &\leq \left(20\mu_\beta^2 Ma_j(b_{j-1}/b_{j-1,j}+b_{j-1,j}/b_j)\right)^{l_i},
  \end{aligned}
\end{equation*}
where the summation is taken over all $a_j$-intervals and the last inequality comes from the following estimate:
\begin{equation*}
  \begin{aligned}
    &\mathbb{E}\left[\#\{J':J'\sim J_i\text{ by a long edge in Definition~\ref{def-LECF}}\}\right]\\
    &\leq \mathbb{E}\left[\#\{\langle \varpi(i_1),\varpi(i_2)\rangle\in E_n:\ \varpi(i_1)\in B_{b_{j-1}}(\varpi(i)),\ \varpi(i_2)\in B_{b_{j-1,j}}(\varpi(i))^c\}\right]\\
    &\quad +\mathbb{E}\left[\#\{\langle \varpi(i_1),\varpi(i_2)\rangle\in E_n:\ \varpi(i_1)\in B_{b_{j-1,j}}(\varpi(i)),\ \varpi(i_2)\in B_{b_{j}}(\varpi(i))^c\}\right]\\
    &\leq \mu_\beta (b_{j-1}/b_{j-1,j}+b_{j-1,j}/b_j).
  \end{aligned}
\end{equation*}

Note that the above bound can be applied to all intervals $J_i$. Since the total number of children across all intervals is $K-1$, as each interval, except $J_1$, is the child of exactly one other interval. Therefore, the expected number of   ``trees'' with a specified up-and-down structure can
be bounded from above by
\begin{equation*}
  \left(20\mu_\beta^2 Ma_j(b_{j-1}/b_{j-1,j}+b_{j-1,j}/b_j)\right)^{K-1}.
\end{equation*}

Now we consider taking a union bound over all possible up-and-down structures $\textbf{w}$ (as we have shown in the proof of Lemma~\ref{est-j-k-explore}, the number of choices for $\textbf{w}$ is at most $4^{K}$). Also note that since all red animals and the structure among all $J_i$'s do not share the same edge, there is independence and we can multiply probabilities arising from $a_j$-animals in different intervals in the exploration process. Thus for fixed $K$ and $x_1,\cdots,x_K$ and $\textbf{w}$, the probability that an $a_j$-interval $J_1$ is connected to an $a_k$-animal $L\in\widetilde{\mathcal{L}}_{j\to k}^{(2)}$ whose exploration process can be encoded by such parameters is at most
\begin{equation*}
  \left(20\mu_\beta^2 Ma_j(b_{j-1}/b_{j-1,j}+b_{j-1,j}/b_j)\right)^{K-1}\prod_{i=1}^K p_j^1(x_i),
\end{equation*}
where $p_j^1(x_i)$ for $x_i>0$ is defined in \eqref{eq-L1-k-k}. If $x_i=0$ we set $p_j^1(x_i)=1$.
In addition, we take a union bound over the parameters $K$ and $x_1,\cdots, x_K$ and $\textbf{w}$. 
Thus combining this with the restriction for $K$ and $\{x_i\}_{i=1}^K$ in \eqref{K-and-xi}, we get that the probability that an $a_j$-interval $J_1$ is connected to an $a_k$-animal $L\in\widetilde{\mathcal{L}}_{j\to k}^{(2)}$ which contains $s'$ vertices is at most
\begin{equation*}
  \sum_{K\geq 1}\sum_{s\leq s'}\sum_{x_1+\cdots+x_K=s,0\leq x_i\leq a_j}\I_{\{3(K-1)\geq s'-s\}} \left(80\mu_\beta^2 Ma_j(b_{j-1}/b_{j-1,j}+b_{j-1,j}/b_j)\right)^{K-1}\prod_{i=1}^K p_j^1(x_i),
\end{equation*}
where the range of $x_i$ comes from the analysis around \eqref{K-and-xi}.
Thus we take a union bound over all $a_j$-intervals contained in $J$ (which is an $a_k$-interval) and get that
\begin{equation}\label{Ljk2-start}
  \begin{aligned}
    &\mathbb{P}\left[\text{there exists $L\in\widetilde{\mathcal{L}}_{j\to k}^{(2)}$ such that }L\sim J\ \text{and }\#L=s'\right]\\
    &\leq (b_k/b_j)\sum_{K\geq 1}\sum_{s\leq s'}\sum_{x_1+\cdots+x_K=s,0\leq x_i\leq a_j}\I_{\{3(K-1)\geq s'-s\}} \left(80\mu_\beta^2 Ma_j(b_{j-1}/b_{j-1,j}+b_{j-1,j}/b_j)\right)^{K-1}\\
    &\quad\quad\quad \cdot \prod_{i=1}^K p_j^1(x_i),
  \end{aligned}
\end{equation}
which implies the lemma.
\end{proof}

We now turn to the

\begin{proof}[Proof of Lemma~\ref{p-j-k-2}]
With Lemma~\ref{p-j-k-2-step1} in hand, the remaining part of the proof is to estimate the RHS on \eqref{eq-L2-start}. To simplify this formula, we first note that when $K=1$ the only term in the summation is $p_k^1(s')\I_{\{j=k\}}$. Also note that by the definition of $b_j$ and $b_{j-1,j}$ in \eqref{def-bk}, one has
\begin{equation*}\label{bk-1k}
\begin{aligned}
  &b_{j-1}/b_{j-1,j}+b_{j-1,j}/b_j\\
  &=\exp\left\{M^{-0.2}\left(\sqrt{a_{j-1}}-\sqrt{(a_{j-1}+a_j)/2}\right)\right\}
  +\exp\left\{M^{-0.2}\left(\sqrt{(a_{j-1}+a_j)/2}-\sqrt{a_{j}}\right)\right\}\\
  &\leq 2\exp\left\{-M^{0.2} \left(\sqrt{1+\lambda/2}-1\right)\sqrt{a_{j-1}}\right\}\\
  &\leq 2\exp\left\{-\frac{1}{8}\lambda M^{0.2}\sqrt{a_{j-1}}\right\}=2\exp\left\{-\frac{1}{8} M^{0.1}\sqrt{a_{j-1}}\right\}.
  \end{aligned}
\end{equation*}
Here the last equality uses $\lambda=M^{-0.1}$. Combining this with \eqref{Ljk2-start} and Lemma~\ref{p-L1-k-k} yields
\begin{equation}\label{Ljk2-step2}
  \begin{aligned}
    &\mathbb{P}\left[\text{there exists $L\in\widetilde{\mathcal{L}}_{j\to k}^{(2)}$ such that }L\sim J_j\ \text{and }\#L=s'\right]\\
    &\leq p_k^1(s')\I_{\{j=k\}}\\
    &\quad +(b_k/b_j)\sum_{K\geq 2}\sum_{0\leq s'-s\leq 3(K-1)} \left(160\mu_\beta^2 Ma_j\exp\left\{-\frac{1}{8} M^{0.1}\sqrt{a_{j-1}}\right\}\right)^{K-1}\sum_{x_1+\cdots+x_K=s}\prod_{i=1}^K p_j^1(x_i)\\
    &\leq  p_k^1(s')\I_{\{j=k\}}+(b_k/b_j)\sum_{K\geq 2}\sum_{0\leq s'-s\leq 3(K-1)} \left(160\mu_\beta^2 Ma_j\exp\left\{-\frac{1}{8}M^{0.1}\sqrt{a_{j-1}}\right\}\right)^{K-1}\\
    &\qquad\qquad\qquad\qquad\cdot(2a_j)^K\left(800\mu_\beta Ma_j\right)^K \exp\left\{-c_6 M^{0.25}\max\left\{sa_{j-1}^{-1/2},a_{j-1}^{1/2}, A^{1/2}\right\}\right\}\\
    &=:p_k^1(s')\I_{\{j=k\}}+(b_k/b_j)\sum_{K\geq 2}\sum_{0\leq s'-s\leq 3(K-1)} T_{K,s},\\
    &=p_k^1(s')\I_{\{j=k\}}+(b_k/b_j)\sum_{s\leq s'}\sum_{K\geq 2: K-1\geq (s'-s)/3} T_{K,s}
  \end{aligned}
\end{equation}

In the following, for convenience, we denote
$$p_s=\exp\left\{-c_6 M^{0.25}\max\left\{sa_{j-1}^{-1/2},a_{j-1}^{1/2}, A^{1/2}\right\}\right\}.$$
Note that since $M$ is sufficiently large, by the definition of $a_k$ in \eqref{def-ai} we have
\begin{equation*}
  T_{K+1,s}/T_{K,s}\leq \widetilde{C}_1 \mu_\beta^4 M^2a_j^5\exp\left\{-\frac{1}{8} M^{0.1}\sqrt{a_{j-1}}\right\}<1/2
\end{equation*}
for some deterministic constant $\widetilde{C}_1<\infty$. As a result, applying this to \eqref{Ljk2-step2}, we obtain that for sufficiently large $M\in \mathds{N}$,
\begin{equation}\label{Ljk2-step3}
  \begin{aligned}
    &\mathbb{P}\left[\text{there exists $L\in\widetilde{\mathcal{L}}_{j\to k}^{(2)}$ such that }L\sim J_j\ \text{and }\#L=s' \right]\\
    &\leq p_k^1(s')\I_{\{j=k\}}+(b_k/b_j)\sum_{s\leq s'} \left(\widetilde{C}_1\mu_\beta^4 M^2a_j^5\exp\left\{-\frac{1}{8} M^{0.1}\sqrt{a_{j-1}}\right\}\right)^{\max\{(s'-s)/3,1\}} p_s\\
    &\leq p_k^1(s')\I_{\{j=k\}}+(b_k/b_j)\sum_{s\leq s'} \left(\exp\left\{-\frac{1}{16} M^{0.1}\sqrt{a_{j-1}}\right\}\right)^{\max\{(s'-s)/3,1\}} p_s\\
    &\leq p_k^1(s')\I_{\{j=k\}}+(b_k/b_j)\sum_{s\leq s'} \left(\exp\left\{-\frac{1}{32} M^{0.1}\sqrt{a_{j-1}}\right\}\right)\left(\exp\left\{-\frac{1}{32} M^{0.1}\sqrt{a_{j-1}}\right\}\right)^{(s'-s)/3} p_s\\
    &=p_k^1(s')\I_{\{j=k\}}+(b_k/b_j)\exp\left\{-\frac{1}{32} M^{0.1}\sqrt{a_{j-1}}\right\}\sum_{s\leq s'}\exp\left\{-\frac{1}{96}(s'-s)M^{0.1}\sqrt{a_{j-1}}\right\}p_s.
  \end{aligned}
\end{equation}

In addition, we let
\begin{equation*}
\begin{aligned}
  Y_s&=\exp\left\{-\frac{1}{96}(s'-s)M^{0.1}\sqrt{a_{j-1}}\right\}p_s\\
  &=\exp\left\{-\frac{1}{96}(s'-s)M^{0.1}\sqrt{a_{j-1}}-c_6M^{0.25}\max\left\{sa_{j-1}^{-1/2},a_{j-1}^{1/2}, A^{1/2}\right\}\right\}.
\end{aligned}
\end{equation*}
Then using 
the fact that $M$ is sufficiently large we have
  \begin{align}
    Y_{s+1}/Y_s&\geq \exp\left\{\frac{1}{96}M^{0.1}\sqrt{a_{j-1}}-c_6 M^{0.25}a_{j-1}^{-1/2}\right\}\nonumber\\
    &=\exp\left\{\frac{1}{96}M^{0.1}\sqrt{a_{j-1}}\left(1-96c_6 M^{0.15}a_{j-1}^{-1}\right)\right\}\nonumber\\
    &\geq \exp\left\{\frac{1}{96}M^{0.1}\sqrt{a_{j-1}}\left(1-96c_6M^{0.15}A^{-1}\right)\right\}>2,\nonumber
  \end{align}
where the last inequality uses $A=M^{0.2}$. Applying this to \eqref{Ljk2-step3}, we have
\begin{equation*}
  \begin{aligned}
    &\mathbb{P}\left[\text{there exists $L\in\widetilde{\mathcal{L}}_{j\to k}^{(2)}$ such that }L\sim J_j\right]\\
    &\leq p_k^1(s')\I_{\{j=k\}}+2(b_k/b_j)\exp\left\{-\frac{1}{32} M^{0.1}\sqrt{a_{j-1}}\right\}\exp\left\{-c_6 M^{0.25}\max\left\{s'a_{j-1}^{-1/2},a_{j-1}^{1/2}, A^{1/2}\right\}\right\}\\
    &\leq \begin{cases}
      1600\mu_\beta Ma_k \exp\left\{-c_6 M^{0.25} \max\left\{a_{k-1}^{1/2},s'a_{k-1}^{-1/2},A^{1/2}\right\}\right\},&\quad j=k;\\
     2(b_k/b_j)\exp\left\{-\frac{1}{32} M^{0.1}\sqrt{a_{j-1}}\right\}\exp\left\{-c_6 M^{0.25}s'a_{j-1}^{-1/2}\right\}, & \quad j<k,\ s'\in(a_{k-1},a_k],
    \end{cases}
  \end{aligned}
\end{equation*}
which implies the result.
\end{proof}

From Lemma~\ref{p-j-k-2} with $j=k$ and Remark \ref{rem-j=k}, we directly have the following relation between the probabilities $p_k$ and $q_k$, defined in \eqref{p0k}.

\begin{lemma}\label{p-to-q}
  For each $k\geq 1$, assume that \eqref{est-pk} holds for each $s\leq a_k$. Then we have \eqref{est-qk} for each $s\leq a_k$.
\end{lemma}

Additionally, by applying the union bound on the probabilities in Lemma~\ref{p-j-k-2} with respect to $j$, we have the following estimate.

\begin{lemma}\label{p-j-k-2-2}
  For $\beta>0$, $k\geq 1,\ s'\in (a_{k-1},a_k]$ and sufficiently large $M\in \mathds{N}$ satisfying \eqref{cond-M1}, assume that there exists a constant $c_6>0$ {\rm(}depending only on $\beta${\rm)} such that \eqref{est-pk} holds for all $j<k$. Then the probability that an $a_k$-interval $J$ is connected to some animal $L\in\cup_{j<k}\widetilde{\mathcal{L}}^{(2)}_{j\to k}$ with $\# L=s'$ is at most $ \exp\{-c_6 M^{0.25}s'a_{k-1}^{-1/2}\}/2$.
\end{lemma}
\begin{proof}
From Lemma~\ref{p-j-k-2}, it is sufficient to show that
  \begin{equation}\label{goal-jk2}
    2\sum_{j< k} (b_k/b_j)\exp\left\{-\frac{1}{32}M^{0.1}a_{j-1}^{1/2}-c_6 M^{0.25}s'a_{j-1}^{-1/2}\right\}\leq \frac{1}{2} \exp\left\{-c_6 M^{0.25}s'a_{k-1}^{-1/2}\right\}.
  \end{equation}
In the following, we will estimate the LHS of \eqref{goal-jk2} term by term.

For a fixed $j< k$, note that by the definition of $a_k$ in \eqref{def-ai} and $s'\in (a_{k-1},a_k]$, we have
\begin{equation*}
  \begin{aligned}
    s'\left(a_{j-1}^{-1/2}-a_{k-1}^{-1/2}\right)=s'\frac{a_{k-1}^{1/2}-a_{j-1}^{1/2}}{a_{j-1}^{1/2}a_{k-1}^{1/2}}
    &\geq a_{k-1}^{1/2}-a_{j-1}^{1/2}\\
    &\geq (1+\lambda)^{-1/2}\left(a_{k}^{1/2}-a_j^{1/2}\right)\geq \left(a_{k}^{1/2}-a_j^{1/2}\right)/2.
  \end{aligned}
\end{equation*}
As a result we get
\begin{equation*}
  \begin{aligned}
    -c_{6}M^{0.25}s'a_{j-1}^{-1/2}
    &= -c_{6}M^{0.25}s'a_{k-1}^{-1/2}-c_{6}M^{0.25}s'\left(a_{j-1}^{-1/2}-a_{k-1}^{-1/2}\right)\\
    &\leq -c_{6}M^{0.25}s'a_{k-1}^{-1/2}-\frac{1}{2}c_{6}M^{0.25}\left(a_{k}^{1/2}-a_j^{1/2}\right),
  \end{aligned}
\end{equation*}
which implies that (recalling the definition of $b_j$ in \eqref{def-bk}) for sufficiently large $M$,
\begin{equation*}
  \begin{aligned}
  \exp\left\{-c_{6}M^{0.25}s'a_{j-1}^{-1/2}\right\}&\leq \exp\left\{-c_{6}M^{0.25}s'a_{k-1}^{-1/2}\right\} \exp\left\{-\frac{1}{2}c_{6}M^{0.25}(a_{k}^{1/2}-a_j^{1/2})\right\}\\
  &\leq b_j/b_k \exp\left\{-c_{6}M^{0.25}s'a_{k-1}^{-1/2}\right\}.
  \end{aligned}
\end{equation*}
Applying this to \eqref{goal-jk2} yields that
\begin{equation*}
  \text{LHS of \eqref{goal-jk2}}\leq 2\exp\left\{-c_{6}M^{0.25}s'a_{k-1}^{-1/2}\right\}\sum_{j<k}\exp\left\{-\frac{1}{32} M^{0.1}a_{j-1}^{1/2}\right\}\leq \text{RHS of \eqref{goal-jk2}},
\end{equation*}
where the last inequality uses the fact that $M$ is sufficiently large.
\end{proof}

With the above lemmas at hand, we now can present the

\begin{proof}[Proof of Proposition \ref{lem-prob-CF}]
We start with the case where $k=1$. Recall that $a_1=A=M^{0.2}$. By Proposition~\ref{P-red-many-tree} and a union bound over all vertices in an $a_1$-interval, we have
   \begin{equation*}\label{prob-L2-1}
    p_1(s)\leq b_1\exp\{-c_1Ms\}\leq \exp\left\{M^{0.2}\sqrt{a_1}-c_1M\right\}\leq \exp\left\{-c_1M^{0.25}A^{1/2}\right\}
   \end{equation*}
   for all $s\geq 1$, where $c_1>0$ (depending only on $\beta$) is the constant defined in Proposition \ref{P-red-many-tree}.
Thus, \eqref{est-pk} holds for $k=1$ and $s\leq a_1$ with $c_6\leq c_1$.
In the rest of the proof, let $c_6=\min\{c_7,c_1\}$, where $c_7>0$ (depending only on $\beta$) is the constant defined in Lemma \ref{prob-1-bad}.

We now consider the case where $k>1$, assuming that \eqref{est-pk} and \eqref{est-qk} hold for all $l<k$ and all $s\leq a_{l}$. Based on induction and Lemma~\ref{p-to-q}, it suffices to show that \eqref{est-pk} holds for $k$ and $s\leq a_k$.

  We first consider the case when $s\leq a_{k-1}$. Let $J$ be an $a_k$-interval. In this case, the event in the definition of $p_k(s)$ in \eqref{p0k} is equivalent to the event that there exists $L\in\widetilde{\mathcal{L}}_{k}$ such that $J\sim L$ with either $\mathrm{deg}(L)\geq 20\mu_\beta Ma_{k-1}$ or $L\in \widetilde{\mathcal{L}}^{(2)}_{(k-1)\to (k-1),\mathrm{bad}}$. Hence,
  \begin{equation}\label{est-pk2}
  \begin{aligned}
  p_k(s)&\leq \mathds{P}\left[\text{there exists $L\in \widetilde{\mathcal{L}}_k$ such that }J\sim L\ \text{and }\mathrm{deg}(L)\geq 20\mu_\beta Ma_{k-1}\right]\\
  &\quad +\mathds{P}\left[\text{there exists $L\in \widetilde{\mathcal{L}}_k$ such that }J\sim L,\ \#L=s\ \text{and }L\in \widetilde{\mathcal{L}}^{(2)}_{(k-1)\to (k-1),\mathrm{bad}}\right]\\
  &=:(I)+(II).
  \end{aligned}
  \end{equation}

  For the first term $(I)$ on the RHS of \eqref{est-pk2}, using the similar arguments for deriving \eqref{prob-L2}, we get that
  \begin{equation}\label{est-I}
  \begin{aligned}
  &(I)\leq (4\mu_\beta)^s\exp\left\{-6\mu_\beta Ma_{k-1}\right\}<\frac{1}{2}\exp\left\{-c_6M^{0.25}a_{k-1}^{1/2}\right\}.
  \end{aligned}
  \end{equation}

We now turn to the second term $(II)$ on the RHS of \eqref{est-pk2}.
Recall for a red animal $L$, the event that $L\in \widetilde{\mathcal{L}}^{(2)}_{(k-1)\to (k-1),\mathrm{bad}}$ is equivalent to the event that  $L\in \widetilde{\mathcal{L}}^{(2)}_{(k-1)\rightarrow (k-1)}$ and it is $a_{k-1}$-bad, see Definition \ref{def-expansion} (3).
For simplicity, for any $l\in\mathds{N}$ and any red $a_l$-animal $L$, we write $R_{L,l}$ for the effective resistance in \eqref{eff-ik} with $k$ replacing by $l$, that is,
\begin{equation}\label{eff-L}
R_{L,l}=  R\left(B_{b_{l-1,l}}(L), B_{b_{l}}(L)^c;E_n\setminus E_{L\times L}\right),
\end{equation}
where $B_\cdot(L)$ is defined in \eqref{def-BL}. Let $\mathcal{AA}_{k-1}$ be the set of all admissible $a_{k-1}$-animals. Here we say $L$ is an admissible animal if
$$\mathbb{P}\left[L\in \widetilde{\mathcal{L}}^{(2)}_{(k-1)\to {(k-1)}}\right]>0.$$
Then from \eqref{p0k} and Definition \ref{def-a1bad}, it is clear that
\begin{equation}\label{term-II}
\begin{aligned}
(II)&= \sum_{L\in \mathcal{AA}_{k-1}: \# L=s}\mathds{P}\left[L\in\widetilde{\mathcal{L}}^{(2)}_{(k-1)\to {(k-1)}}, J\sim L\right] \cdot \mathds{P}\left[L\ \text{is $a_{k-1}$-bad}\ |L\in\widetilde{\mathcal{L}}^{(2)}_{(k-1)\to {(k-1)}}, J\sim L\right]\\
&\leq \left(\sum_{L\in \mathcal{AA}_{k-1}: \# L=s}\mathds{P}\left[L\in\widetilde{\mathcal{L}}^{(2)}_{(k-1)\to {(k-1)}}, J\sim L\right]\right) \\
& \quad\quad  \cdot \max_{L\in \mathcal{AA}_{k-1}: \# L=s} \mathds{P}\left[L\ \text{is $a_{k-1}$-bad}\ |L\in\widetilde{\mathcal{L}}^{(2)}_{(k-1)\to {(k-1)}}, J\sim L\right]\\
&\leq \frac{b_k}{b_{k-1}} \widetilde{q}_{k-1}(s)\\
&\quad \cdot  \max_{L\in \mathcal{AA}_{k-1}: \# L=s}\mathds{P}\left[R_{L,k-1}\leq c_*(20\mu_\beta Ma_{k-1})^{-1}\e^{\delta M^{0.2}\sqrt{a_{k-1}}/2}\ |\text{$ L\in \widetilde{\mathcal{L}}^{(2)}_{(k-1)\to {(k-1)}}$ , $J\sim L$}\right].\\
\end{aligned}
\end{equation}
Here the first term in the last inequality follows from a union bound over all $a_{k-1}$-intervals and then uses the fact that each $ a_{k-1}$-interval is connected to at most one red animal $L\in \widetilde{\mathcal{L}}^{(2)}_{(k-1)\to {(k-1)}}$ from Definitions \ref{def-LECF} and \ref{def-NNCF}.

In the following, we aim to estimate the final conditional probability in \eqref{term-II} for any fixed $L\in\mathcal{AA}_{k-1}$. To achieve this, for any $\varpi(i)\in V_n\setminus L$ with $\varpi(i)\sim L$, we write
\begin{equation}\label{eff-ik2}
R_{\varpi(i),l}=  R\left(B_{b_{l-1,l}}(\varpi(i)), B_{b_l}(\varpi(i))^c;E_{\varpi(i),l}\right),
\end{equation}
where
$$
E_{\varpi(i),l}:=E_n\setminus \left(E_{L\times L}\cup\left(\cup_{\varpi(j)\in V_n\setminus (L\cup \{\varpi(i)\})}B_{b_{l-1,l}}(\varpi(j))\right) \right).
$$
Since $L\in \widetilde{\mathcal{L}}^{(2)}_{(k-1)\to {(k-1)}}$, one has $\text{deg}(L)\leq 20\mu_\beta Ma_{k-1}$. Combining this with \eqref{eff-L} and \eqref{eff-ik2} gives that
\begin{equation*}
\left\{R_{L,k-1}\leq c_*(20\mu_\beta Ma_{k-1})^{-1}\e^{\delta M^{0.2}\sqrt{a_{k-1}}/2}\right\}\subset \bigcup_{\varpi(i)\in V_n\setminus L:\varpi(i)\sim L}\left\{R_{\varpi(i),k-1}\leq c_*\e^{\delta M^{0.2}\sqrt{a_{k-1}}/2}\right\}.
\end{equation*}
This implies that
\begin{equation}\label{term-II-2}
\begin{aligned}
&\mathds{P}\left[R_{L,k-1}\leq c_*(20\mu_\beta Ma_{k-1})^{-1}\e^{\delta M^{0.2}\sqrt{a_{k-1}}/2}\ |\text{$ L\in \widetilde{\mathcal{L}}^{(2)}_{(k-1)\to {(k-1)}}$ , $J\sim L$}\right]\\
&\leq \sum_{\varpi(i)\in V_n\setminus\{L\}: \varpi(i)\sim L}\mathds{P}\left[R_{\varpi(i),k-1}\leq c_*\e^{\delta M^{0.2}\sqrt{a_{k-1}}/2}\ |\text{$L\in \widetilde{\mathcal{L}}^{(2)}_{(k-1)\to {(k-1)}}$, $J\sim L$}\right].
\end{aligned}
\end{equation}

We now claim that the two events in the conditional probability on the RHS of \eqref{term-II-2} are independent.
Indeed, on the one hand, from Definition \ref{def-a1bad}, we can see that the effective resistance $R_{\varpi(i),k-1}$ defined in \eqref{eff-ik2} is independent of the number of vertices contained in $L$ and its degree.
On the other hand, we claim that if $L$ is upgraded from some $L'\in \widetilde{\mathcal{L}}_{a_l}$ for some $l<k-1$ due to the $a_l$-bad condition, then for any $\varpi(i)\sim L'\subset L$, the collection of edges responsible for the resistance $R_{\varpi(i),l}$ does not overlap with that for $R_{\varpi(i),k-1}$.
Otherwise, suppose that $R_{\varpi(i),l}$ and $R_{\varpi(i),k-1}$ do share a common responsible edge $\langle \varpi(u),\varpi(v)\rangle$. Then from \eqref{eff-ik2}, we can deduce that
$$
\langle \varpi(u),\varpi(v)\rangle \in E_{(B_{b_l}(\varpi(i))\setminus B_{b_{l-1,l}}(\varpi(i)))\times (B_{b_{k-1}}(\varpi(i))\setminus B_{b_{k-2,k-1}}(\varpi(i)))}.
$$
However, according to Definition \ref{def-LECF}, the edge $\langle \varpi(u),\varpi(v)\rangle$ must be incorporated into $L$ through the $a_l$-LEAF process. Combining this with the definition of $R_{\varpi(i),k-1}$ in \eqref{eff-ik2}, we can find that the resistance  $R_{\varpi(i),k-1}$ does not use the edge $\langle \varpi(u),\varpi(v)\rangle$, leading to a contradiction (See Figure~\ref{fig-resist} for an illustration).

\begin{figure}[h]
  \includegraphics[scale=0.7]{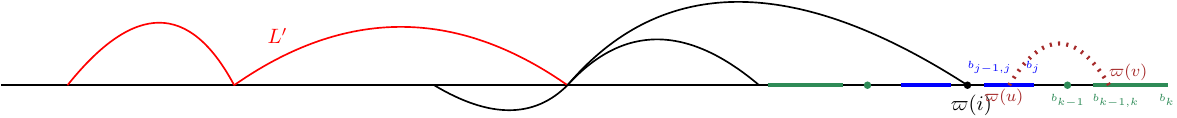}
  \caption{The illustration for the proof of the claim about the independence. The red curves represent some red animal $L'\in\widetilde{\mathcal{L}}_{a_l}\ (l<k-1)$ that can be upgraded to some red $a_{k-1}$-animal $L$. The black points are the vertices connected to $L'$ directly. The blue and green lines represent the annulus $B_{b_l}(\varpi(i))\setminus B_{b_{l-1,l}}(\varpi(i))$ and $B_{b_{k-1}}(\varpi(i))\setminus B_{b_{k-2,k-1}}(\varpi(i))$, respectively. The dotted brown curve represents the absence of the long edge directly connecting $B_{b_l}(\varpi(i))\setminus B_{b_{l-1,l}}(\varpi(i))$ and $B_{b_{k-1}}(\varpi(i))\setminus B_{b_{k-2,k-1}}(\varpi(i))$. Since the resistance $R_{\varpi(i),k-1}$ depends only on edges with one endpoint in $B_{b_{k-1}}(\varpi(i))\setminus B_{b_{k-2,k-1}}(\varpi(i))$, it does not share any edge with the resistance $R_{\varpi(i),l}$. This implies these two resistances are independent.
  }
  \label{fig-resist}
\end{figure}

Additionally, recall that $\delta\in (0,\delta^*]$ is fixed. Here $\delta^*$ is the constant defined in Lemma \ref{Lem-RNtail-1}. According to  \eqref{eff-ik2}, Lemma \ref{Lem-RNtail-1}, the monotonicity of the effective resistance and the translation invariance of the model, we get that for $\varpi(i)\in V_n\setminus L$ with $\varpi(i)\sim L$,
\begin{equation*}
\begin{aligned}
&\mathds{P}\left[R_{\varpi(i),k-1}\leq c_*\e^{\delta M^{0.2}a_{k-1}^{1/2}/2}\right]\\
&\leq \mathds{P}\left[R\left([-b_{k-1,k},b_{k-1,k}], [-b_k,b_k]^c\right)\leq c_*\e^{\delta M^{0.2}\sqrt{a_{k-1}}/2}\ | [-b_{k-1,k},b_{k-1,k}]\nsim [-b_k,b_k]^c\right]\\
&\leq \exp\left\{-\widetilde{C}_*M^{0.2}a_{k-1}^{1/2}\right\},
\end{aligned}
\end{equation*}
where $\widetilde{C}_*:=(2C_*)^{-1}$, $c_*$ and $C_*$ are the constants defined in Lemma \ref{Lem-RNtail-1}.
 Applying this and $\text{deg}(L)\leq 20\mu_\beta Ma_{k-1}$ for $L\in \widetilde{\mathcal{L}}^{(2)}_{(k-1)\rightarrow (k-1)}$ to \eqref{term-II-2} and then to \eqref{term-II} yields
\begin{equation}\label{est-II2}
(II)\leq \frac{b_k}{b_{k-1}} \widetilde{q}_{k-1}(s)(20\mu_\beta M a_{k-1})\exp\left\{-\widetilde{C}_*M^{0.2}a_{k-1}^{1/2}\right\}.
\end{equation}
Thus it suffices to show that the RHS of \eqref{est-II2} can be bounded above by $\exp\{-c_6 M^{0.25}a_{k-1}^{1/2}\}/2$.

Recall that $s\leq a_{k-1}$ and we assume \eqref{est-pk} and \eqref{est-qk} hold for all $l<k$ and all $s\leq a_{l}$.
Applying \eqref{est-qk} to $\widetilde{q}_{k-1}(s)$ on the RHS of \eqref{est-II2} yields that
    \begin{align*}
      &\text{RHS of \eqref{est-II2}}\label{goal-1-1}\\
      &\leq 40000\mu_\beta^2 M^2 a_{k-1}^2 \exp\left\{M^{0.2}(a_k^{1/2}-a_{k-1}^{1/2})-c_{6,k-1}M^{0.25}a_{k-2}^{1/2}-\widetilde{C}_*M^{0.2}a_{k-1}^{1/2}\right\}\nonumber\\
      &\leq \exp\left\{\lambda M^{0.2}a_{k-1}^{1/2}-c_{6}M^{0.25}a_{k-2}^{1/2}-\frac{1}{2}\widetilde{C}_*M^{0.2}a_{k-1}^{1/2}\right\}\quad \quad \text{by \eqref{def-ai} and $M$ sufficiently large}\nonumber\\
    &\leq \exp\left\{-c_6 M^{0.25}a_{k-1}^{1/2}-a_{k-1}^{1/2}\left(\frac{1}{2}\widetilde{C}_*M^{0.2}-2c_{6,k-1}\lambda M^{0.25}-\lambda M^{0.2}\right)\right\}\quad \text{by }a_{k-1}^{1/2}-a_{k-2}^{1/2}\leq 2\lambda a_{k-1}^{1/2}\nonumber\\
    &\leq \exp\left\{-c_6 M^{0.25}a_{k-1}^{1/2}-a_{k-1}^{1/2}\left(\frac{1}{2}\widetilde{C}_*M^{0.2}-2c_{6,k-1} M^{0.15}- M^{0.1}\right)\right\}\quad \quad \text{by }\lambda =M^{-0.1}\nonumber\\
    &\leq \frac{1}{2}\exp\left\{-c_6 M^{0.25}a_{k-1}^{1/2}\right\} \quad \quad \text{by $M$ sufficiently large}\nonumber.
    \end{align*}
Combining this with \eqref{est-II2} and \eqref{est-I}, we get that \eqref{est-pk} holds for $s\leq a_{k-1}$.

In the final part of the proof, we consider the case where $s>a_{k-1}$. From Lemma~\ref{prob-1-bad} and $s\leq a_k$, it is obvious that
  \begin{equation*}
    \begin{aligned}
      &\mathbb{P}\left[\text{there exists $L\in \mathcal{L}_{k}$ such that }\# L=s\ \text{and }J\sim L\right]\\
      &\leq \exp\left\{-c_7M\max\{s,a_k\}\right\}\\
      &\leq \frac{1}{4}\exp\left\{-c_6 M^{0.25}sa_{k-1}^{-1/2}\right\}\quad \quad \text{by }c_6=\min\{c_7,c_1\}.
    \end{aligned}
  \end{equation*}
  Combining this with Lemmas~\ref{prob-total-explore}, \ref{p-j-k-2-2} and the definition of $\widetilde{\mathcal{L}}_k$ in \eqref{L0ak}, we get the desired result for  $s\geq a_{k-1}$.
\end{proof}

Finally, we present the
\begin{proof}[Proof of Proposition \ref{prop-eventH}]
  From Proposition \ref{lem-prob-CF}, by taking a union bound we have
  \begin{equation*}
  \begin{aligned}
    \mathds{P}[\mathcal{H}^c]&\leq \sum_{k>K_{*}} \frac{n}{b_k}\exp\left\{-c_6  M^{0.25}a_{k-1}^{1/2}\right\}\\
    &=\sum_{k>K_{*}}n\exp\left\{-M^{0.2}a_k^{1/2}-c_6  M^{0.25}a_{k-1}^{1/2}\right\}\\\
    &\leq 2n \exp\left\{-c_6 M^{0.25}a_{K_{*}}^{1/2}\right\}\\
    &\leq 2n\cdot 2^{-1}n^{-3}=n^{-2}\quad \text{by definition of $K_{*}$ in \eqref{def-K}}.
    \end{aligned}
  \end{equation*}
Hence, we complete the proof.
\end{proof}

\subsection{Proof of Proposition \ref{weak-supermult}}\label{sect-proofwm}
Recall that $\mathcal{H}$ is the event that $\bigcup_{k>K_{*}}\widetilde{\mathcal{L}}_{k}=\emptyset$.
For each $k\geq 1$, $\mathcal{G}_k=\widetilde{\mathcal{L}}^{(2)}_{k\to k}\setminus \widetilde{\mathcal{L}}^{(2)}_{k\to k,\mathrm{bad}}$, see Definition \ref{def-expansion} (3).
For simplicity,  we will refer to a red animal in $\mathcal{G}_k$ as a good red $a_k$-animal.

Combining the definition of  $\mathcal{H}$ and  Definition \ref{def-expansion}, it is clear that on the event $\mathcal{H}$, each red component $L\in \mathcal{L}$ can be expanded to a good red $a_k$-animal for some $k\geq 1$. That is, there exists a good red animal $\widetilde{L}\in \cup_{k\leq K_{*}}\mathcal{G}_k$ such that $L\subset \widetilde{L}$, which means both the vertices and edges within $L$ are included in $\widetilde{L}$.
Additionally, by Definition \ref{def-expansion}, we can derive the following properties for each good red $a_k$-animal $L\in \mathcal{G}_k$:
\begin{itemize}

\item[(P1)]  For each black vertex $\varpi(i)\in V_n$ with $\varpi(i)\sim L$, there is no long edge connecting $B_{b_{k-1}}(\varpi(i))$ and $B_{b_{k-1,k}}(\varpi(i))^c$ directly, nor is there any long edge connecting $B_{b_{k-1,k}}(\varpi(i))$ and $B_{b_{k}}(\varpi(i))^c$ directly.

\item[(P2)] There is no other good red $a_k$-animal $\widetilde{L}\in \mathcal{G}_k$ such that
$$\text{dist}(L,\widetilde{L})\leq 2b_k=2\exp\left\{M^{0.2}a_k^{1/2}\right\}.$$

\item[(P3)] For each black vertex $\varpi(i)\in V_n$ with $\varpi(i)\sim L$,
$$
R\left(B_{b_{k-1,k}}(L), B_{b_k}(L)^c;E_n\setminus E_{L\times L}\right)> c_*(20\mu_\beta Ma_k)^{-1}\exp\left\{\delta M^{0.2}a_k^{1/2}/2\right\},
$$
where $\delta\in (0,\delta^*]$, $c_*$ and $\delta^*$ are the constants defined in Lemma \ref{Lem-RNtail-1} and $B_{\cdot}(L)$ is defined in \eqref{def-BL}.
\end{itemize}

In the following, let us assume that the event $\mathcal{H}$ occurs. We fix a unit flow $g$ from $\varpi(0)$ to $\varpi(n-1)$ restricted to $V_n$ and a good red $a_k$-animal $L$.
We will show the energy generated by the flow $g$ through $L$ is significantly smaller than the energy generated by the flow $g$ through the neighboring region of $L$.
Specifically, we denote by $\theta_{g,L}$ the total amount of the flow $g$ into $L$, given by
\begin{equation*}
\theta_{g,L}=\sum_{\varpi(i)\in L}\sum_{\varpi(j): \varpi(j)\text{ is black}}g_{\varpi(j)\varpi(i)}\I_{\{g_{\varpi(j)\varpi(i)}>0\}}.
\end{equation*}
Since $L\in\mathcal{G}_k\subset \widetilde{\mathcal{L}}^{(2)}_{k\rightarrow k}$, it is evident that $\#L\leq a_k$ and $\text{deg}(L)\leq 20\mu_\beta Ma_k$.
Therefore, from this and (P3) we obtain
\begin{equation}\label{g-L}
\sum_{\langle \varpi(i),\varpi(j)\rangle\in E_{L \times L}}g_{\varpi(i)\varpi(j)}^2\leq \theta^2_{g,L} a^2_k
\end{equation}
and
\begin{equation}\label{g-Lint}
\begin{aligned}
&\sum_{\varpi(i)\in V_n \setminus L: \varpi(i)\sim L}\sum_{\langle \varpi(j_1),\varpi(j_2)\rangle\in (E_{B_{b_k}(\varpi(i))^2}\setminus E_{L\times L}) }g_{\varpi(j_1)\varpi(j_2)}^2\\
&\geq \theta^2_{g,L} R\left(B_{b_{k-1,k}}(L), B_{b_k}(L)^c;E_n\setminus E_{L\times L}\right)\\
&\geq c_*\theta^2_{g,L}(20\mu_\beta Ma_k)^{-1}\exp\left\{\delta M^{0.2}a_k^{1/2}/2\right\}.  
\end{aligned}
\end{equation}
In addition, according to the definition of $a_k$ in \eqref{def-ai}, we can see that for sufficiently large $M\in \mathds{N}$,
\begin{equation*}
c_*(20\mu_\beta Ma_k)^{-1}\exp\left\{\delta M^{0.2}a_k^{1/2}/2\right\}\geq a_k^3\quad \text{for all }k\geq 1.
\end{equation*}
Combining this with \eqref{g-L} and \eqref{g-Lint}, we get that
\begin{equation}\label{energy-L-B}
\sum_{\langle \varpi(i),\varpi(j)\rangle\in E_{L\times L} }g_{\varpi(i)\varpi(j)}^2\leq \frac{1}{a_k}\sum_{\varpi(i)\in V_n\setminus L : \varpi(i)\sim L}\sum_{\langle \varpi(j_1),\varpi(j_2)\rangle\in (E_{B_{b_k}(\varpi(i))^2}\setminus E_{L\times L})}g_{\varpi(j_1)\varpi(j_2)}^2.
\end{equation}
This implies that the energy generated by the flow $g$ through $L$ is significantly smaller than the energy generated by the flow $g$ through the region $\cup_{\varpi(i)\in V_n\setminus L:\ \varpi(i)\sim L}B_{b_k}(\varpi(i))$.

Notice that the property (P2) implies that for any $k\geq 1$ and any $L_1,L_2\in \mathcal{G}_k$, the aforementioned neighboring regions for different animals do not intersect. Specifically, we have
\begin{equation*}
\left(\bigcup_{\varpi(i)\in V_n\setminus L_1:\varpi(i)\sim L_1}B_ {b_k}\left(\varpi(i)\right)\right)\bigcap \left(\bigcup_{\varpi(i)\in V_n\setminus L_2:\varpi(i)\sim L_2}B_ {b_k}(\varpi(i))\right)=\emptyset.
\end{equation*}
Therefore, we can apply \eqref{energy-L-B} to the good red animals in $\mathcal{G}_1$ to get
\begin{equation}\label{property-g}
\begin{aligned}
&\sum_{\langle \varpi(i),\varpi(j)\rangle\in E_n}g_{\varpi(i)\varpi(j)}^2\\
&=\sum_{k\geq 1}\sum_{L\in \mathcal{G}_k}\sum_{\langle \varpi(i),\varpi(j)\rangle\in E_{L\times L}}g_{\varpi(i)\varpi(j)}^2+\sum_{\langle \varpi(i),\varpi(j)\rangle\in E_n\setminus(\cup_{k\geq 1}\cup_{L\in \mathcal{G}_k}E_{L\times L})}g_{\varpi(i)\varpi(j)}^2\\
&\leq \left(1+\frac{1}{a_1}\right)\left(\sum_{k\geq 2}\sum_{L\in \mathcal{G}_k}\sum_{\langle \varpi(i),\varpi(j)\rangle\in E_{L\times L}}g_{\varpi(i)\varpi(j)}^2+\sum_{\langle \varpi(i),\varpi(j)\rangle\in E_n\setminus(\cup_{k\geq 1}\cup_{L\in \mathcal{G}_k}E_{L\times L})}g_{\varpi(i)\varpi(j)}^2\right).
\end{aligned}
\end{equation}
By iteratively applying \eqref{energy-L-B} to the good red animals in $\mathcal{G}_k$ for all $k\geq 1$, we obtain that for each $l\geq 1$,
\begin{equation*}
\begin{aligned}
&\sum_{k\geq l}\sum_{L\in \mathcal{G}_k}\sum_{\langle \varpi(i),\varpi(j)\rangle\in E_{L\times L}}g_{\varpi(i)\varpi(j)}^2+\sum_{\langle \varpi(i),\varpi(j)\rangle\in E_n\setminus(\cup_{k\geq 1}\cup_{L\in \mathcal{G}_k}E_{L\times L})}g_{\varpi(i)\varpi(j)}^2\\
&\leq \left(1+\frac{1}{a_l}\right)\left(\sum_{k\geq l+1}\sum_{L\in \mathcal{G}_k}\sum_{\langle \varpi(i),\varpi(j)\rangle\in E_{L\times L}}g_{\varpi(i)\varpi(j)}^2+\sum_{\langle \varpi(i),\varpi(j)\rangle\in E_n\setminus(\cup_{k\geq 1}\cup_{L\in \mathcal{G}_k}E_{L\times L})}g_{\varpi(i)\varpi(j)}^2\right).
\end{aligned}
\end{equation*}
Combining this with \eqref{property-g} yields
\begin{equation}\label{property-g2}
\begin{aligned}
\sum_{\langle \varpi(i),\varpi(j)\rangle\in E_n}g_{\varpi(i)\varpi(j)}^2&\leq \prod_{k\geq 1}\left(1+\frac{1}{a_k}\right)\sum_{\langle \varpi(i),\varpi(j)\rangle\in E_n\setminus(\cup_{k\geq 1}\cup_{L\in \mathcal{G}_k}E_{L\times L})}g_{\varpi(i)r\varpi(j)}^2\\
&\leq C_1\sum_{\langle \varpi(i),\varpi(j)\rangle\in E_n\setminus(\cup_{k\geq 1}\cup_{L\in \mathcal{G}_k}E_{L\times L})}g_{\varpi(i)\varpi(j)}^2
\end{aligned}
\end{equation}
for some constant $C_1<\infty$ (depending only on $\beta$), where the last inequality follows from the definition of $a_k$ in \eqref{def-ai}.

With the above analysis, we can present the
\begin{proof}[Proof of Proposition \ref{weak-supermult}]
We first assume that the event $\mathcal{H}$ occurs.
Let $f$ be the unit electric flow from $0$ to $mn-1$ restricted to $[0,mn)$ in the LRP model. By the definition of flow, for each $i\in [0,n)$,
\begin{equation*}
\sum_{u\in I_i}\sum_{j\in [0,n)\setminus\{i\}}\sum_{v\in I_j}f_{vu}\I_{\{f_{vu}>0\}}
\end{equation*}
represents  the amount of flow entering into the interval $I_i$. In addition,  recall that $R_{I_i}$ denotes the internal (optimal) energy in $I_i$ as defined in \eqref{def-RIi}. Consequently,
\begin{equation}\label{inter-Ii}
\frac{1}{2}\sum_{u\in I_i}\sum_{v\in I_i}f_{uv}^2
\geq \left(\sum_{u\in I_i}\sum_{j\in [0,n)\setminus\{i\}}\sum_{v\in I_j}f_{vu}\I_{\{f_{vu}>0\}}\right)^2R_{I_i}
\end{equation}
for each $i\in [0,n)$.

We now construct a unit flow $g$ from $\varpi(0)$ to $\varpi(n-1)$ in the graph $G_n$ based on $f$, which is defined as
\begin{equation}\label{def-f-g}
g_{\varpi(i)\varpi(j)}=\sum_{u\in I_i}\sum_{v\in I_j}f_{uv}\quad \text{for all } i,j\in [0,n).
\end{equation}
Applying the flow $g$ to \eqref{property-g2} and according to \eqref{inter-Ii} and the definition of black vertices (i.e. $(\delta,\alpha)$-very good vertices, see Definition \ref{def-alphagood-1}), we obtain that on the event $\mathcal{H}$,
\begin{align*}
R_{[0,mn)}(0,mn-1)&=\frac{1}{2}\sum_{u\sim v}f_{uv}^2\geq \sum_{i\in [0,n)}\left(\sum_{u\in I_i}\sum_{v\in I^c_i}f_{vu}\I_{\{f_{vu}>0\}}\right)^2R_{I_i}\\
&= \sum_{i\in[0,n)}\left(\sum_{u\in I_i}\sum_{j\in [0,n)\setminus\{i\}}\sum_{v\in I_j}f_{vu}\I_{\{f_{vu}>0\}}\right)^2R_{I_i}\\
&\geq  \sum_{i\in[0,n)}\sum_{j\in [0,n)\setminus\{i\}}\left(\sum_{u\in I_i}\sum_{v\in I_j}f_{vu}\I_{\{f_{vu}>0\}}\right)^2R_{I_i}\\
&\geq \sum_{i\in[0,n)}\sum_{j\in [0, n)\setminus\{i\}} g^2_{\varpi(j)\varpi(i)}\I_{\{g_{\varpi(j)\varpi(i)}>0\}}R_{I_i} \quad \quad \text{by \eqref{def-f-g}}\\
&\geq \sum_{\langle \varpi(i),\varpi(j)\rangle\in E_n\setminus(\cup_{k\geq 1}\cup_{L\in \mathcal{G}_k}E_{L\times L})}g_{\varpi(j)\varpi(i)}^2\I_{\{g_{\varpi(j)\varpi(i)}>0\}} R_{I_i}\\
&\geq \frac{1}{2}\alpha_2 m^\delta\sum_{\langle \varpi(i),\varpi(j)\rangle\in E_n\setminus(\cup_{k\geq 1}\cup_{L\in \mathcal{G}_k}E_{L\times L})}g_{\varpi(i)\varpi(j)}^2 \quad \quad \text{by Definition \ref{def-alphagood-1}}\\
&\geq \frac{1}{2C_1}\alpha_2 m^\delta\sum_{\langle \varpi(i), \varpi(j)\rangle \in E_n}g_{\varpi(i)\varpi(j)}^2\quad \quad \text{by \eqref{property-g2}}\\
&\geq \frac{1}{2C_1}\alpha_2 m^{\delta}R^{G}_{V_n}(\varpi(0),\varpi(n-1))=:\widetilde{c}_1m^{\delta}R^{G}_{V_n}(\varpi(0),\varpi(n-1)),
\end{align*}
where $\alpha_2>0$ and $C_1$ (depending only on $\beta$) are the constants in Definition \ref{def-alphagood-1} and \eqref{property-g2}, respectively. Consequently, combining this with the scaling invariance of the LRP model and $\mathds{P}[\mathcal{H}^c]\leq n^{-2}$, we get that
\begin{equation}\label{exponent-delta}
\begin{aligned}
&\mathds{E}\left[R_{[0,mn)}(0,mn-1)\right]\\
&\geq \mathds{E}\left[R_{[0,mn)}(0,mn-1)\I_{\mathcal{H}}\right]\\
&\geq \widetilde{c}_1m^{\delta}\mathds{E}\left[R^{G}_{V_n}(\varpi(0),\varpi(n-1))\I_{\mathcal{H}}\right]\\
&=\widetilde{c}_1m^{\delta}\left(\mathds{E}\left[R^{G}_{V_n}(\varpi(0),\varpi(n-1))\right]-\mathds{E}\left[R^{G}_{V_n}(\varpi(0),\varpi(n-1))\I_{\mathcal{H}^c}\right]\right)\\
&\geq \widetilde{c}_1m^{\delta}\left(\mathds{E}\left[R^{G}_{V_n}(\varpi(0),\varpi(n-1))\right]-n\mathds{P}\left[\mathcal{H}^c\right]\right)\quad\quad  \text{by }R^{G}_{V_n}(\varpi(0),\varpi(n-1))\leq n\\
&\geq \widetilde{c}_2m^{\delta}\Lambda(n)\quad\quad  \text{by Proposition \ref{R(0,n)main} and Remark \ref{rem-delta*} (2)}
\end{aligned}
\end{equation}
for some constant $\widetilde{c}_2>0$ depending only on $\beta$.
\end{proof}

\section{Several types of resistances in $[0,n)$}\label{sect-severaltype}
Throughout this section, we fix a sufficiently large $n\in \mathds{N}$.
The objective of this section is to establish the relationships and estimates between different types of resistances, including point-to-point resistance, point-to-box resistance, and box-to-box resistance. To this end, we begin in Section \ref{secondm} by providing a second moment bound for point-to-point resistances.
Next, in Section \ref{sect-pb}, we will use the second moment method to show that the point-to-box resistance $R(0,[-n,n]^c)$ is comparable to the point-to-point resistance $\Lambda(n)$. Finally, in Section \ref{sect-bb}, we want to prove that the box-to-box resistance $R([-n,n],[-2n,2n]^c)$, conditioned on the absence of edges connecting $[-n,n]$ and $[-2n,2n]^c$, is also comparable to the previous two types of resistances.


\subsection{The second moment bound}\label{secondm}

The following proposition is the main output of this subsection, which gives a relationship  between the second moment of the point-to-point resistances and their first moment. We employ a slightly modified technique that has already been used in \cite[Section 8]{Baumler23a} to prove a similar result for the graph distance in the $\beta$-LRP model.

\begin{proposition}\label{prop-2thmoment}
For all $\beta>0$, there exists a constant $C_1=C_1(\beta)<\infty$ {\rm(}depending only on $\beta${\rm)} such that for all $n\in \mathds{N}$ and all $i,j\in [0,n)$,
\begin{equation*}
\mathds{E}\left[R_{[0,n)}(i,j)^2\right]\leq C_1\Lambda(n)^2.
\end{equation*}
\end{proposition}

To prove Proposition \ref{prop-2thmoment}, we need to do some preparations. For $n\in \mathds{N}$ and $\iota\in (0,1)$, let $R_\iota^n=[n-\iota n, n)$ and $L_\iota^n=[0,\iota n]$.

\begin{lemma}\label{lem-point-box}
For all $\beta\geq 0$, there exists sufficiently small $\iota>0$ {\rm(}depending only on $\beta${\rm)} such that for all $n\in \mathds{N}$,
\begin{equation*}
\mathds{E}[R_{[0,n)}(L_\iota^n, R_\iota^n)]\geq \frac{1}{2}\mathds{E}[R_{[0,n)}(0,n-1)].
\end{equation*}
\end{lemma}
\begin{proof}
Without loss of generality, we assume that $\iota n\in \mathds{N}$. Otherwise, we can replace $\iota n$ with $\lfloor \iota n\rfloor$ in the following proof.

First, we claim that for $n\in \mathds{N}$ and $\iota\in (0,1)$,
\begin{equation}\label{point-box}
\mathds{E}[R_{[0,n)}(0,n-1)]\leq 2\Lambda(\iota n)+\mathds{E}[R_{[0,n)}(L_\iota^n, R_\iota^n)].
\end{equation}
To establish this, denote
\begin{equation*}
I_L=\left\{x\in L_\iota^n:\ \text{there exists $y\in (\iota n, n-\iota n)$ such that }\langle x,y\rangle \in \mathcal{E}_{L_\iota^n\times (\iota n, n-\iota n)}\right\}
\end{equation*}
and
\begin{equation*}
I_R=\left\{y\in R_\iota^n:\ \text{there exists $x\in (\iota n, n-\iota n)$ such that }\langle x,y\rangle \in \mathcal{E}_{ (\iota n,n-\iota n)\times R_\iota^n}\right\}.
\end{equation*}
Let $g$ be a unit electric flow from $L_\iota^n$ to $R_\iota^n$ restricted to $[0,n)$. We denote $\{\theta_x\}_{x\in I_L}$ as the amount of flow $g$ leaving each point $x$ in $L_\iota^n$, and $\{\vartheta_y\}_{y\in I_R}$ as the amount of the flow entering each point $y$ in $R_\iota^n$.
We will now construct a unit flow $f$ from 0 to $n-1$ according to the flow $g$ as follows.
\begin{itemize}
\item[(1)] For any fixed $x\in I_L$, denote by $\phi^{(x)}$ the unit electric flow from 0 to $x$ restricted to the interval $L_\iota^n$. Then we define
    \begin{equation*}
    f_{uv}=\sum_{x\in I_L}\theta_x\phi_{uv}^{(x)}\quad \text{for all }\langle u,v\rangle \in \mathcal{E}_{L_\iota^n\times L_\iota^n}.
    \end{equation*}

\item[(2)] For any fixed $y\in I_R$, denote by $\phi^{(y)}$ the unit electric flow from $y$ to $n-1$ restricted to the interval $R_\iota^n$. Then we define
        \begin{equation*}
    f_{uv}=\sum_{y\in I_R}\vartheta_y\phi_{uv}^{(y)}\quad \text{for all }\langle u,v\rangle \in \mathcal{E}_{R_\iota^n\times R_\iota^n}.
    \end{equation*}
\end{itemize}

From the above construction, it can checked that $f$ is a unit flow from 0 to $n-1$ restricted to the interval $[0,n-1)$.
In addition, note that the flow $g$ (accordingly, $R_{[0,n)}(L_\iota^n, R_\iota^n)$), along with $I_L$ and $I_R$, depend only  on the edge set
\begin{equation*}
\widehat{\mathcal{E}}_{\iota}:=\mathcal{E}_{(\iota n,n-\iota n)\times [0,n)}.
\end{equation*}
Given $\widehat{\mathcal{E}}_{\iota}$, the energies generated by $\phi^{(x)}$ for all $x\in I_L$
and $\phi^{(y)}$ for all $y\in I_R$ depend only on the edge sets $\mathcal{E}_{L_\iota^n\times L_\iota^n}$ and $\mathcal{E}_{R_\iota^n\times R_\iota^n}$, respectively. Therefore, according to Cauchy-Schwartz inequality and the fact that $\sum_{x\in I_L}\theta_x=\sum_{y\in I_R}\vartheta_y=1$, we have
\begin{align}
\mathds{E}\left[R_{[0,n)}(0,n-1)\ |\widehat{\mathcal{E}}_\iota\right]
&\leq \frac{1}{2}\mathds{E}\left[\sum_{u\sim v}f_{uv}^2\ |\widehat{\mathcal{E}}_\iota\right]\label{R0n}\\
&=R_{[0,n)}(L_\iota^n, R_\iota^n)
+\frac{1}{2}\mathds{E}\left[\sum_{\langle u,v\rangle\in \mathcal{E}_{L_\iota^n\times L_\iota^n}}f_{uv}^2\ | \widehat{\mathcal{E}}_\iota\right]
+\frac{1}{2}\mathds{E}\left[\sum_{\langle u,v\rangle\in \mathcal{E}_{R_\iota^n\times R_\iota^n}}f_{uv}^2\ | \widehat{\mathcal{E}}_\iota\right]\nonumber\\
&=R_{[0,n)}(L_\iota^n, R_\iota^n)
+\frac{1}{2}\mathds{E}\left[\sum_{\langle u,v\rangle\in \mathcal{E}_{L_\iota^n\times L_\iota^n}}\left(\sum_{x\in I_L}\theta_x\phi^{(x)}_{uv}\right)^2\ | \widehat{\mathcal{E}}_\iota\right]\nonumber\\
&\quad +\frac{1}{2}\mathds{E}\left[\sum_{\langle u,v\rangle\in \mathcal{E}_{R_\iota^n\times R_\iota^n}}\left(\sum_{y\in I_R}\vartheta_y\phi^{(y)}_{uv}\right)^2\ | \widehat{\mathcal{E}}_\iota\right]\nonumber\\
&\leq R_{[0,n)}(L_\iota^n, R_\iota^n)
+\frac{1}{2}\mathds{E}\left[\sum_{\langle u,v\rangle\in \mathcal{E}_{L_\iota^n\times L_\iota^n}}\sum_{x\in I_L}\theta_x\left(\phi^{(x)}_{uv}\right)^2\ | \widehat{\mathcal{E}}_\iota\right]\nonumber\\
&\quad +\frac{1}{2}\mathds{E}\left[\sum_{\langle u,v\rangle\in \mathcal{E}_{R_\iota^n\times R_\iota^n}}\sum_{y\in I_R}\vartheta_y\left(\phi^{(y)}_{uv}\right)^2\ | \widehat{\mathcal{E}}_\iota\right]\nonumber\\
&=R_{[0,n)}(L_\iota^n, R_\iota^n)
+\sum_{x\in I_L}\theta_x\mathds{E}\left[ R_{L_\iota^n}(0,x)\ | \widehat{\mathcal{E}}_\iota\right]
+\sum_{y\in I_R}\vartheta_y\mathds{E}\left[ R_{R_\iota^n}(y,n-1)\ | \widehat{\mathcal{E}}_\iota\right].\nonumber
\end{align}
Additionally, by the independence of edges, the translation invariance of LRP model and the definition of $\Lambda(\cdot)$ in \eqref{def-Lambda}, we get that for all $x\in I_L$ and all $y\in I_R$,
\begin{equation*}
\mathds{E}\left[ R_{L_\iota^n}(0,x)\ | \widehat{\mathcal{E}}_\iota\right]\leq \max_{u,v\in L_\iota^n}\mathds{E}\left[R_{L_\iota^n}(u,v)\right]=\Lambda(\iota n)
\end{equation*}
and
\begin{equation*}
\mathds{E}\left[ R_{R_\iota^n}(y,n-1)\ | \widehat{\mathcal{E}}_\iota\right]
\leq  \max_{u,v\in R_\iota^n}\mathds{E}\left[R_{R_\iota^n}(u,v)\right]=\Lambda(\iota n).
\end{equation*}
Applying this and  $\sum_{x\in I_L}\theta_x=\sum_{y\in I_R}\vartheta_y=1$ into \eqref{R0n} yields \eqref{point-box}.

Moreover, from Propositions \ref{R(0,n)main} and \ref{weak-supermult}, we can see that there exist constants $\widetilde{c}_1>0$ and $\delta>0$ (both depending only on $\beta$) such that
$$
\widetilde{c}_1\iota^{-\delta}\Lambda(\iota n)\leq \Lambda(n).
$$
Applying this to \eqref{point-box}, we arrive at
\begin{equation*}
\begin{aligned}
\mathds{E}\left[R_{[0,n)}(L_\iota^n, R_\iota^n)\right]&\geq \mathds{E}\left[R_{[0,n)}(0,n-1)\right]- 2\Lambda(\iota n)\\
&\geq \mathds{E}\left[R_{[0,n)}(0,n-1)\right]-2\widetilde{c}_1^{-1}\iota^{\delta}\Lambda(n)\\
&\geq \mathds{E}\left[R_{[0,n)}(0,n-1)\right]-\widetilde{C}_1\iota^{\delta}\mathds{E}\left[R_{[0,n)}(0,n-1)\right]
\end{aligned}
\end{equation*}
for some constant $\widetilde{C}_1=\widetilde{C}_1(\beta)<\infty$ depending only on $\beta$, where in the last inequality we used Proposition \ref{R(0,n)main}.
Consequently, by choosing sufficiently small $\iota>0$ such that $\widetilde{C}_1\iota^{\delta}<1/4$, we get the desired result.
\end{proof}

In the following, we will analyze the growth of the resistance $\mathds{E}[R_{[0,n)}(0,n-1)]$. As we mentioned in \cite[Section 1.2]{DFH24+}, for fixed $\beta<1$, we can obtain
 $$
 \mathds{E} \left[R_{[0,n)}(0,n)\right]\geq c(\beta)n^{1-\beta}
 $$
 by finding a sufficient number of cut-edges, where $c(\beta)>0$ is a constant depending only on $\beta$. The next lemma provides a more uniform bound on the growth of $\mathds{E}[R_{[0,n)}(0,n-1)]$ that applies simultaneously for all $\beta\in [0,1]$.

\begin{lemma}\label{cut-beta0-1}
For any $\beta\in [0,1]$, there exists a constant $c_1>0$ {\rm(}depending only on $\beta${\rm)} such that for all $m,n\in \mathds{N}$,
\begin{equation}\label{R-R-box}
\mathds{E}[R_{[0,mn)}(0,mn-1)]\geq c_1m^{1-\beta}\mathds{E}[R_{[0,n)}(0,n-1)].
\end{equation}
\end{lemma}
\begin{proof}
For fixed $\beta\geq 0$, it follows from Lemma \ref{lem-point-box} that there exists a sufficiently small $\iota\in (0,1)$ (depending only on $\beta$) such that for all $n\in \mathds{N}$,
\begin{equation}\label{box-point}
\mathds{E}\left[R_{[0,n)}(L_\iota^n,R_\iota^n)\right]\geq \frac{1}{2}\mathds{E}\left[R_{[0,n)}(0,n-1)\right].
\end{equation}
From this, we claim that there exists a constant $\widetilde{c}_1>0$ (depending only on $\beta$) such that
\begin{equation}\label{box-point-2n}
\mathds{E}\left[R_{[-n,2n)}([-n,0),[n,2n))\right]\geq \widetilde{c}_1 \mathds{E}\left[R_{[0,n)}(0,n-1)\right]
\end{equation}
holds for all  sufficiently large $n\in \mathds{N}$.
Indeed, for fixed $\iota>0$, let $A$ be the event that the rightmost vertex connected to $[-n,0)$ lies inside $L_\iota^n$ and the leftmost vertex connected to $[n,2n)$ lies inside $R_\iota^n$. By a simple calculation, it is clear that $\mathds{P}(A)>0$. In addition, on the event $A$, one has
\begin{equation}\label{R-event}
R_{[-n,2n)}([-n,0),[n,2n))\geq R_{[0,n)}(L_\iota^n,R_\iota^n).
\end{equation}
Note that by the monotonicity of effective resistance with edges, the event in \eqref{R-event} and the event $A$ are both decreasing with the edge set. Therefore, it follows from the FKG inequality and \eqref{box-point} that
\begin{equation*}
\begin{aligned}
\mathds{E}\left[R_{[-n,2n)}([-n,0),[n,2n))\right]&\geq \mathds{E}\left[R_{[-n,2n)}([-n,0),[n,2n))\I_A\right]\\
&\geq \mathds{E}\left[R_{[0,n)}(L_\iota^n,R_\iota^n)\I_A\right]\\
&\geq \mathds{E}\left[R_{[0,n)}(L_\iota^n,R_\iota^n)\right]\mathds{P}[A]\geq \frac{\mathds{P}[A]}{2}\mathds{E}\left[R_{[0,n)}(0,n-1)\right].
\end{aligned}
\end{equation*}
This, along with $\mathds{P}(A)>0$, implies the claim \eqref{box-point-2n} holds.

Now for the long-range percolation on the interval $[0,m)$, we call an odd point $i\in [1,m-2]$ a separation point (see Figure \ref{sep-point} for an illustration) if
\begin{equation*}
i\nsim \{0,1,\cdots,i-2\}\cup \{i+2,\cdots,m-1\}\quad \text{and}\quad \{0,1,\cdots,i-1\}\nsim \{i+1,\cdots,m-1\}.
\end{equation*}
From the calculations in \cite[Proof of Lemma 8.1]{Baumler23a}, we have that
\begin{equation}\label{sepa-i}
\mathds{P}\left[i\text{ is a separation point}\right]\geq 0.1 m^{-\beta}.
\end{equation}
In addition, for odd $i\in [1,m-2]$, we also call the interval $[in,(i+1)n)$ as a separation interval if
\begin{equation*}
[in,(i+1)n)\nsim [0,(i-1)n)\cup [(i+2)n, mn)\quad \text{and}\quad [0,in)\nsim [(i+1)n,mn).
\end{equation*}
By the scaling invariance of LRP model and \eqref{sepa-i}, we can see that
\begin{equation}\label{P-sepint}
\mathds{P}\left[\text{$[in,(i+1)n)$ is a separation interval}\right]= \mathds{P}\left[\text{$i$ is a separation point}\right]\geq 0.1 m^{-\beta}.
\end{equation}

\begin{figure}[htbp]
\centering
\includegraphics[scale=0.8]{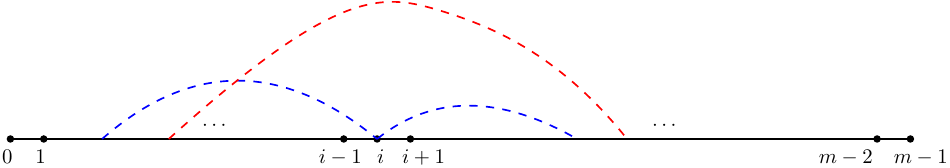}
\caption{The illustration for the definition of separation points. The blue dashed curves in the graph represent the absence of long edges directly connecting $i$ and $\{0,1,\cdots, i-2\}\cup \{i+2,\cdots, m-1\}$, while the red dashed curve represents the absence of long edge directly connecting the intervals $\{0,1,\cdots, i-1\}$ and $\{i+1,\cdots,m-1\}$.
}
\label{sep-point}
\end{figure}

Now let $i_1,\cdots, i_k\in \{1,\cdots, m-2\}$ be the integers such that $[i_ln, (i_l+1)n)$ is a separation interval for all $1\leq l\leq k$.
Then each flow from $0$ to $mn-1$ restricted in $[0,mn)$ needs to cross all separation intervals. Furthermore, when passing through a separation interval $[i_l n, (i_l+1)n)$, the flow must hit the intervals $[(i_l-1)n, i_ln)$ and $[(i_l+1)n, (i_l+2)n)$ before and after crossing, respectively (see Figure \ref{sep-int} for an illustration).
\begin{figure}[htbp]
\centering
\includegraphics[scale=0.8]{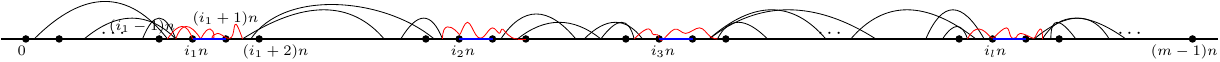}
\caption{The illustration for the flow passing through the separation intervals. The blue lines represent the separation intervals. When the flow travels from 0 to $mn-1$, it must first traverse the left side of the separation interval $[i_ln,(i_l+1)n)$, specifically the interval $[(i_l-1)n,i_ln)$. The flow then enters the separation interval and subsequently exits into the interval $[(i_l+1)n, (i_l+2)n)$. This generates energy lower-bounded by $R_{[(i_l-1)n,(i_l+2)n)}([(i_l-1)n, i_ln), [(i_l+1)n, (i_l+2)n))$ as represented by the red curves in the figure.
}
\label{sep-int}
\end{figure}
Therefore,
\begin{equation*}
R_{[0,mn)}(0,mn-1)\geq \sum_{l=1}^k R_{[(i_l-1)n,(i_l+2)n)}\left([(i_l-1)n, i_ln), [(i_l+1)n, (i_l+2)n)\right).
\end{equation*}
Since the event $\{[i_ln, (i_l+1)n)\ \text{is a separation interval}\}$ and the resistance
$$
R_{[(i_l-1)n,(i_l+2)n)}\left([(i_l-1)n, i_ln), [(i_l+1)n, (i_l+2)n)\right)
$$
 are decreasing with edges, by the FKG inequality and the translation invariance of LRP model, we get that
\begin{equation*}
\begin{aligned}
&\mathds{E}\left[R_{[0,mn)}(0,mn-1)\right]\\
&\geq \mathds{E}\left[\sum_{i=1}^{m-2}\I_{\{[in, (i+1)n)\ \text{is a separation interval}\}}R_{[(i-1)n,(i+2)n)}([(i-1)n, in), [(i+1)n, (i+2)n))\right]\\
&\geq \sum_{i=1}^{m-2}\mathds{E}\left[\I_{\{[in, (i+1)n)\ \text{is a separation interval}\}}\right]\mathds{E}\left[R_{[-n,2n)}([-n,0),[n,2n))\right]\\
&\geq \sum_{i\in \{1,\cdots, m-2\},\ i\text{ odd}} 0.1 m^{-\beta}\widetilde{c}_1 \mathds{E}\left[R_{[0,n)}(0,n-1)\right]\quad \quad \text{by \eqref{box-point-2n} and \eqref{P-sepint} }\\
&\geq \widetilde{c}_2m^{1-\beta} \mathds{E}\left[R_{[0,n)}(0,n-1)\right]
\end{aligned}
\end{equation*}
for some small constant $\widetilde{c}_2>0$ (depending only on $\beta$) and sufficiently large $m$. For small $m\in \mathds{N}$, we can take $\widetilde{c}_2$ small enough such that \eqref{R-R-box} holds for such $m$.
\end{proof}

With this at hand, we are ready to present the

\begin{proof}[Proof of Proposition \ref{prop-2thmoment}]
In this proof, we say that the vertex $i\in [1, m-2]$ is a cut point for the interval $[0,m)$ if there exists no edge of the form $\langle k,l\rangle $ with $0\leq k<i<l\leq m-1$. From \cite[Proof of Lemma 4.5]{Baumler23a} we can see that  for all $i<m/2$ and all $\beta<2$,
\begin{equation*}
\mathds{P}\left[i\ \text{is a cut point for the interval }[0,m)\right]\leq 4i^{-\beta}.
\end{equation*}
Furthermore,
\begin{equation}\label{numb-cut}
\begin{aligned}
f(\beta,m)&:=\mathds{E}\left[\#\{i\in[1,m-2]:\ i\ \text{is a cut point for the interval }[0,m)\}\right]\\
&\leq
\begin{cases}
\frac{20}{1-\beta}m^{1-\beta},\quad &\beta<1,\\
10+8\log m, \quad &1\leq \beta\leq 2, \\
20,\quad &\beta>2.
\end{cases}
\end{aligned}
\end{equation}

In the following, we aim to bound the second moment of $R_{[0,m^{n+1})}(0, m^{n+1}-1)$. The main idea of our method is originated from \cite[Proof of Lemma 4.5 for $d=1$]{Baumler23a} (see also \cite{DS13+}), with minor modifications specific to the resistance.

To this end, we say that an interval $[im^n,(i+1)m^n)$ is unbridged if there exists no edge $\langle k,l\rangle $ with $k\in [0,im^n)$ and $l\in [(i+1)m^n,m^{n+1})$. On the contrary, if there exists such an edge we say that the interval is bridged by the edge $\langle k,l\rangle $.
Obviously, the intervals $[0,m^n)$ and $[(m-1)m^{n},m^{n+1})$ are unbridged by the definition.
Moreover, for any $i\in [1,m-2]$,
\begin{equation*}
\mathds{P}\left[[im^n,(i+1)m^n) \text{ is unbridged}\right]=\mathds{P}\left[i\ \text{is a cut point}\right]\leq 4i^{-\beta}.
\end{equation*}

We now also define a set of edges $\mathcal{B}$ as follows. For $i,j\in [0,m)$ with $i+1<j$, if
\begin{itemize}

\item[(1)]  $[im^n,(i+1)m^n)\sim  [jm^n,(j+1)m^n)$;
\medskip

\item[(2)]  $[(i-l_1)m^n,(i-l_1+1)m^n)\nsim [(j+l_2)m^n,(j+l_2+1)m^n)$ for all $(l_1,l_2)\in \{0,\cdots, i\}\times \{0,\cdots, m-1-j\}\setminus \{(0,0)\}$,
\end{itemize}
then we add one edge connecting $[im^n,(i+1)m^n)$ to $[jm^n,(j+1)m^n)$ directly  into $\mathcal{B}$. Note that if there are multiple such edges, we choose the shortest edge that has the minimum endpoint in $[im^n,(i+1)m^n)$.
With this construction, we can see that $\#\mathcal{B}\leq m$, as each interval $[jm^n,(j+1)m^n)$ can be adjacent to at most two edges in $\mathcal{B}$, and each edge in $\mathcal{B}$ touches two intervals. Furthermore, if an interval $[jm^n,(j+1)m^n)$ is bridged, then there exists an edge $e\in \mathcal{B}$ so that $[jm^n,(j+1)m^n)$ is bridged by $e$.
Let $\mathcal{U}'$ be the set of endpoints of edges in $\mathcal{B}$. We define
\begin{equation}\label{def-u}
\mathcal{U}:=\mathcal{U}'\cup \{0, m^n,\cdots, (m-1)m^n\}\cup \{m^n-1,\cdots, m^{n+1}-1\}.
\end{equation}
 For convenience, we will arrange the elements of $\mathcal{U}$ in the ascending order and denote it as $\mathcal{U}=\{x_0,x_1,\cdots, x_u\}$.
 By the construction, we have $\#\mathcal{U}\leq 4m$ and $|x_{i-1}-x_i|\leq m^n-1$.
 Now for $x_{i-1},x_i$ with $(x_{i-1},x_i)\neq (km^n-1,km^n)$ for all $k$, we say that the interval $[x_{i-1},x_i]$ is bridged, if there exists an edge $\langle k,l\rangle\in \mathcal{B}$ with $k\leq x_{i-1}<x_i\leq l$.
 It is worth noting that for a pair $(x_{i-1},x_i)$ that is not of the form $(km^n-1,km^n)$, there must exist a $j\in [0,m-1]$ such that $[x_{i-1},x_i]\subset [jm^n,(j+1)m^n)$. In this case, if  $[x_{i-1},x_i]$ is not bridged, then $[jm^n,(j+1)m^n)$ is also not bridged.
 In addition, if $[x_{i-1},x_i]$ is bridged, then $ [jm^n,(j+1)m^n)$ is also bridged by some edge in $\mathcal{B}$.
  Within each interval $[jm^n,(j+1)m^n)$, there are at most two points in  $[jm^n,(j+1)m^n)\cap \mathcal{U}$ that correspond to the endpoints of edges in  $\mathcal{B}$, and note that both endpoints of the interval are also included in $[jm^n,(j+1)m^n)\cap \mathcal{U}$. Therefore, we can get that
 \begin{equation*}
\#\left([jm^n,(j+1)m^n)\cap \mathcal{U}\right)\leq 4,
 \end{equation*}
which implies that there are at most three intervals of the form $[x_{i-1},x_i]$ contained within each $[jm^n,(j+1)m^n)$. This leads us to conclude that
\begin{equation}\label{unbridge-1}
\begin{aligned}
&\#\{i\in \{1,\cdots, u\}:\ [x_{i-1},x_i]\ \text{is unbridged}\}\\
&\leq 3\#\{j\in \{0,\cdots, m-1\}:\ [jm^n,(j+1)m^n)\ \text{is unbridged}\}.
\end{aligned}
\end{equation}

From now on, we aim to estimate $R_{[0,m^{n+1})}(0,s)$ for all $s\in (0,m^{n+1})$. To achieve this, denote $u_s=\max\{i\in \{1,\cdots, u \}:\ x_i\leq s\}\leq 4m$, and let
\begin{equation}\label{arg-max}
\tau_s= \text{arg}\left(\max_{i\in \{1,\cdots, u_s-1\}}R_{[x_{i-1},x_i]}(x_{i-1},x_i)\vee R_{[x_{u_s},x_{u_s+1}]}(x_{u_s},s)\right).
\end{equation}
Without loss of generality, we assume that this maximum does not occur at the right endpoint $x_i=km^n$ for some $k$. In fact, if $x_i=km^n$ for some $k$, then by the definition of $\mathcal{U}$ in \eqref{def-u}, we have $[x_{i-1},x_i]=[km^n-1,km^n]$. This leads to $R_{[x_{i-1},x_i]}(x_{i-1},x_i)\leq 1$ for all $i$, which makes the subsequent analysis trivial.
Therefore, if there are multiple maximizers in \eqref{arg-max}, we select one such that $x_i\neq km^n$ for all $k$ and choose the one with the minimal $x_i$ among these maximizers.
This implies that  $[x_{i-1},x_i]$ will always lie within some interval $[jm^n,(j+1)m^n)$.


We will first consider the case where $[x_{\tau_s-1},x_{\tau_s}]$ is bridged  
by some edge $e=\langle x_{\tau_{s1}},x_{\tau_{s2}}\rangle \in \mathcal{B}$ with $\tau_{s1}<\tau_{s2}\leq u_s+1$.
Clearly, in this case, $\tau_s \neq u_s+1$.
This is because if $\tau_s = u_s+1$, then we have $[x_{\tau_s-1},x_{\tau_s}]=[x_{u_s},x_{u_s+1}]$.
Consequently, the edge $e=\langle x_{\tau_{s1}},x_{\tau_{s2}}\rangle$, which bridges the interval $[x_{\tau_s-1},x_{\tau_s}]$,  must satisfy  $x_{\tau_{s2}}>x_{u_s+1}$, implying $\tau_{s2}>u_s+1$.
This contradicts the requirement that $\tau_{s2}\leq u_s+1$.
Now consider the flow that flows from $0=x_0$ to $x_{\tau_{s1}}$, then directly flows to $x_{\tau_{s2}}$ and subsequently flows to $s$. This gives us that
\begin{equation*}
\begin{aligned}
R_{[0,m^{n+1})}(0,s)&\leq \sum_{i=1}^{\tau_{s1}}R_{[x_{i-1},x_i]}(x_{i-1},x_i)+1+\sum_{i=\tau_{s2}+1}^{u_s}R_{[x_{i-1},x_i]}(x_{i-1},x_i)+ R_{[x_{u_s},x_{u_s+1}]}(x_{u_s},s)\\
&\leq u_s\left(\max_{i\neq \tau_s,i\leq u_s+1}R_{[x_{i-1},x_i]}(x_{i-1},x_i\wedge s)\right)
 \leq 4m \left(\max_{i\neq \tau_s,i\leq u_s+1}R_{[x_{i-1},x_i]}(x_{i-1},x_i\wedge s)\right).
\end{aligned}
\end{equation*}

In the case where $[x_{\tau_s-1},x_{\tau_s}]$ is not bridged, or where $[x_{\tau_s-1},x_{\tau_s}]$ is bridged by some edge $e=\langle x_{\tau_{s1}},x_{\tau_{s2}}\rangle\in \mathcal{B}$ with $\tau_{s1}<\tau_{s2}$ and $\tau_{s2}> u_s+1$, we will refer to these two situations as $[x_{\tau_s-1},x_{\tau_s}]$ being $s$-unbridged for simplicity. We will consider the flow that passes through the points $x_0,x_1,\cdots, x_{u_s},s$. Note that even in the presence of the aforementioned long edge $e=\langle x_{\tau_{s1}},x_{\tau_{s2}}\rangle$, we will avoid using it.
Then we have
\begin{align*}
&R_{[0,m^{n+1})}(0,s)\\
&\leq \sum_{i=1}^{\tau_{s}-1}R_{[x_{i-1},x_i]}(x_{i-1},x_i)
+R_{[0,m^{n+1})}(x_{\tau_s-1},x_{\tau_s})+\sum_{i=\tau_{s}+1}^{u_s}R_{[x_{i-1},x_i]}(x_{i-1},x_i)\\
&\quad  + R_{[x_{u_s},x_{u_s+1}]}(x_{u_s},s)\I_{\{\tau_s\neq u_s+1\}}\\
& \leq 4m \left(\max_{i\neq \tau_s,i\leq u_s+1}R_{[x_{i-1},x_i]}(x_{i-1},x_i\wedge s)\right)
 +\max_{[x_{i-1},x_i]\ \text{is $s$-unbridged}, i\leq u_{s}+1}R_{[0,m^{n+1})}(x_{i-1},x_i\wedge s).
\end{align*}
Therefore, in both cases  we have
\begin{equation}\label{iterate-R}
\begin{aligned}
R_{[0,m^{n+1})}(0,s)^2
& \leq 32m^2 \left(\max_{i\neq \tau_s,i\leq u_s+1}R_{[x_{i-1},x_i]}(x_{i-1},x_i\wedge s)\right)^2\\
&\quad +2\left(\max_{[x_{i-1},x_i]\ \text{is $s$-unbridged}, i\leq u_{s}+1}R_{[0,m^{n+1})}(x_{i-1},x_i\wedge s)\right)^2\\
&\leq 32m^2 \left(\max_{i\neq \tau_s,i\leq u_s+1}R_{[x_{i-1},x_i]}(x_{i-1},x_i\wedge s)\right)^2\\
&\quad +2\left(\sum_{[x_{i-1},x_i]\ \text{is $s$-unbridged}, i\leq u_{s}+1}R_{[k_im^n,(k_i+1)m^{n})}(x_{i-1},x_i\wedge s)\right)^2,
\end{aligned}
\end{equation}
where in the last inequality we chose $k_i\in \mathds{Z}$ such that $[x_{i-1},x_i]\subset [k_im^n,(k_i+1)m^n)$ for all $i\leq u_s+1$.

Next, we want to bound both terms in \eqref{iterate-R} in expectation. It is worth emphasizing that conditioned on $\mathcal{U}$,   $\{R_{[x_{i-1},x_i]}(x_{i-1},x_i\wedge s)\}_{1\leq i\leq u_s+1}$ are independent, and their expectations are bounded by $\mathds{E}[R_{[0,m^n)}(0,m^n-1)]$ up to a constant (depending only on $\beta$) from Lemma \ref{R-renorm} and Proposition \ref{R(0,n)main}.
Therefore, from the above fact and \cite[(94)]{Baumler23a} we get that
\begin{equation}\label{1th-term}
\begin{aligned}
&\mathds{E}\left[\left(\max_{i\neq \tau_s,i\leq u_s+1}R_{[x_{i-1},x_i]}(x_{i-1},x_i\wedge s)\right)^2\right]\\
&=\mathds{E}\left[\mathds{E}\left[\left(\max_{i\neq \tau_s,i\leq u_s+1}R_{[x_{i-1},x_i]}(x_{i-1},x_i\wedge s)\right)^2\ |\mathcal{U}\right]\right]\\
&\leq \mathds{E}\left[16m^2\left(\max_{i\neq \tau_s}\mathds{E}\left[R_{[x_{i-1},x_i]}(x_{i-1},x_i)\ |\mathcal{U}\right]\vee \mathds{E}\left[R_{[x_{u_s},x_{u_s+1}]}(x_{u_s},s)\ |\mathcal{U}\right]\right)^2\right]\\
&\leq 16\widetilde{C}_1m^2 \left(\mathds{E}[R_{[0,m^n)}(0,m^n-1)]\right)^2
\end{aligned}
\end{equation}
for some constant $\widetilde{C}_1=\widetilde{C}_1(\beta)<\infty$ depending only on $\beta$.

In addition, define
$$
\Gamma_{n}=\max_{0\leq s<r< m^n}\mathds{E}\left[R_{[0,m^n)}(s,r)^2\right].
$$
By \eqref{unbridge-1}, the translation invariance of the model and the monotonicity of function $f(\beta,\cdot)$, we have
\begin{align*}
&\mathds{E}\left[\sum_{[x_{i-1},x_i]\ \text{is $s$-unbridged}, i\leq u_{s}+1}R_{[k_im^n,(k_i+1)m^n)}(x_{i-1},x_i\wedge s)^2\right]\\
&=\mathds{E}\left[\mathds{E}\left[\sum_{[x_{i-1},x_i]\ \text{is $s$-unbridged}, i\leq u_{s}+1}R_{[k_im^n,(k_i+1)m^n)}(x_{i-1},x_i\wedge s)^2\ |\mathcal{U}\right]\right]\\
&\leq \Gamma_{n} \mathds{E}\left[\sum_{ i\leq u_{s}+1}\I_{\{[x_{i-1},x_i]\ \text{is $s$-unbridged}\}}\right]\\
&\leq \Gamma_{n} \mathds{E}\left[\#\left\{j\in [0,u_s+1):\ [jm^n,(j+1)m^n)\ \text{is $s$-unbridged}\right\}\right]\\
&\leq (2+f(\beta,u_s+1))\Gamma_{n}\leq (2+f(\beta,m))\Gamma_{n}.
\end{align*}
Combining this with \eqref{1th-term} and \eqref{iterate-R}, we obtain
\begin{equation}\label{R0s}
\mathds{E}\left[R_{[0,m^{n+1})}(0,s)^2\right]\leq 2(2+f(\beta,m))\Gamma_{n}+ \widetilde{C}_2m^4 \Upsilon_{n}^2,
\end{equation}
where $\Upsilon_{n}:=\mathds{E}[R_{[0,m^n)}(0,m^n-1)]$ and $\widetilde{C}_2<\infty$ is a constant depending only on $\beta$.
Using the similar argument as in the above, we can also get that the inequality \eqref{R0s} holds for $\mathds{E}\left[R_{[0,m^{n+1})}(r,s)^2\right]$ for all $0\leq s<r< m^{n+1}$.  Consequently,
\begin{equation*}\label{Gamma-iterate}
\Gamma_{n+1}\leq 2(2+f(\beta,m))\Gamma_{n}+ \widetilde{C}_2m^4 \Upsilon_{n}^2=:\widetilde{f}(\beta,m)\Gamma_{n}+ \widetilde{C}_2m^4 \Upsilon_{n}^2.
\end{equation*}
Iterating this inequality over all $k=1,\cdots, n$, we arrive at
\begin{equation}\label{Gamma-iterate-2}
\Gamma_{n+1}\leq \widetilde{C}_2m^4\sum_{k=1}^n \left(\widetilde{f}(\beta,m)\right)^{n+1-k}\Upsilon_{k}^2.
\end{equation}
In addition, by \eqref{numb-cut} we have
\begin{equation*}
\widetilde{f}(\beta,m)=2(2+f(\beta,m))\leq
\begin{cases}
\frac{80}{1-\beta}m^{1-\beta},\quad &\beta<1,\\
40(1+\log m),\quad & 1\leq \beta\leq 2,\\
80,\quad &\beta >2.
\end{cases}
\end{equation*}


In the rest of the proof, we aim to show that for all $\beta>0$, there exists a constant $\widetilde{C}<\infty$ (depending only on $\beta$) such that for all $N\in \mathds{N}$, and all $u,v\in [0,N)$,
\begin{equation}\label{secondmoment}
\mathds{E}\left[R_{[0,N)}(u,v)^2\right]\leq \widetilde{C}\Lambda(N)^2.
\end{equation}

We start with the case $\beta\geq 1$. By Proposition \ref{weak-supermult}, there exists a constant $\delta=\delta(\beta)>0$ (depending only on $\beta$) such that
\begin{equation*}\label{unif-weaksup}
\Upsilon_{k+1}=\mathds{E}\left[R_{[0,m^{k+1})}(0,m^{k+1}-1)\right]
\geq m^{\delta}\mathds{E}\left[R_{[0,m^{k})}(0,m^{k}-1)\right]=m^{\delta}\Upsilon_{k}
\end{equation*}
for all $k\in\mathds{N}$ and all $m\in \mathds{N}$ large enough. Applying this into \eqref{Gamma-iterate-2}, we arrive at
\begin{equation}\label{Gamma-iterate-3}
\Gamma_{n+1}\leq \widetilde{C}_2m^4\sum_{k=1}^n \left(\widetilde{f}(\beta,m)\right)^{n+1-k}(m^{-2\delta})^{n-k}\Upsilon_{n}^2.
\end{equation}
Now when $\beta>1$, we choose $m\in \mathds{N}$ large enough so that
$
\widetilde{f}(\beta, m)m^{-2\delta}<1/2.
$
For $\beta=1$, by the definition of $\widetilde{f}$, it is clear that for sufficiently large $m\in \mathds{N}$, we also have
$\widetilde{f}(1, m)m^{-2\delta(1)}\leq 1/2$.
Applying this into \eqref{Gamma-iterate-3} gives that
\begin{equation*}\label{Gamma-iterate-4}
\begin{aligned}
\Gamma_{n+1}&=\max_{0\leq s<r< m^{n+1}}\mathds{E}[R_{[0,m^{n+1})}(s,r)^2]\\
&\leq  \widetilde{C}_2m^4\widetilde{f}(\beta,m)
\sum_{k=1}^n \left(\frac{1}{2}\right)^{n-k}\Upsilon_{n}^2\leq \widetilde{C}_3m^4\widetilde{f}(\beta,m)\Lambda(m^n)^2
\end{aligned}
\end{equation*}
for some $\widetilde{C}_3<\infty$. According this, Lemma \ref{R-renorm} and Proposition \ref{R(0,n)main}, we obtain that \eqref{secondmoment} holds for $N=m,m^2,\cdots$.

In addition, for general $N\in \mathds{N}$, let $p$ be the number such that $m^p\leq N<m^{p+1}$. Then for any $u,v\in [0,N)$, we can find a sequence $\{i_k\}_{k=1}^K \subset \{1, 2,\cdots, \lceil Nm^{-p}\rceil\}$ such that $|i_k-i_{k+1}|=1$, and $u\leq i_1m^p<\cdots<i_Km^p\leq v$.
From this and the triangle inequality of the resistance, we get that
$$
\begin{aligned}
\mathds{E}[R_{[0,N)}(u,v)^2]&\leq
2m\bigg[\mathds{E}\left[R_{[(i_1-1)m^p,i_1m^p)}(u,i_1m^p-1)^2\right]+1\\
&\quad+\sum_{k=1}^{K-1}\mathds{E}\left[R_{[i_km^p,(i_k+1)m^p)}(i_km^p,(i_k+1)m^p-1)^2\right]+1\\
&\quad +\mathds{E}\left[R_{[i_Km^p,(i_K+1)m^p)}(i_Km^p,v)^2\right]
\bigg].
\end{aligned}
$$
Combining this with \eqref{secondmoment} for $N=m^p$  and Lemma \ref{R-renorm} again we obtain that
\begin{equation*}
\mathds{E}[R_{[0,N)}(u,v)^2]\leq \widetilde{C}_4 m^5\widetilde{f}(\beta,m)\Lambda(m^p)^2\leq \widetilde{C}_5 m^5\widetilde{f}(\beta,m)\Lambda(N)^2
\end{equation*}
for some constants $\widetilde{C}_4, \widetilde{C}_5<\infty$ depending only on $\beta$. That is, \eqref{secondmoment} holds for all $N\in \mathds{N}$.

Next, we turn to the case that $\beta\in (0,1)$. From Lemma \ref{cut-beta0-1}, we get that there exists a constant $c_1\in (0,1)$ (depending only on $\beta$) such that
\begin{equation*}
\begin{aligned}
\mathds{E}\left[R_{[0,m^n)}(0,m^n-1)\right]&\geq c_1 m^{(n-k)(1-\beta)}\mathds{E}\left[R_{[0,m^k)}(0,m^k-1)\right]\\
&\geq \left(c_1 m^{1-\beta}\right)^{n-k}\mathds{E}\left[R_{[0,m^k)}(0,m^k-1)\right]
\end{aligned}
\end{equation*}
for all $n\geq k$ and $m\in \mathds{N}$. Now take $m$ large enough so that $\frac{80m^{\beta-1}}{c_1(1-\beta)}<1/2$. Combining this with \eqref{Gamma-iterate-2} we obtain that 
\begin{align*}
\Gamma_{n+1}&=\max_{0\leq s<r< m^{n+1}}\mathds{E}[R_{[0,m^{n+1})}(s,r)^2]\\
&\leq \widetilde{C}_2m^4\sum_{k=1}^n \left(\widetilde{f}(\beta,m)\right)^{n+1-k}\left(\mathds{E}\left[R_{[0,m^k)}(0,m^k-1)\right]\right)^2\\
&\leq \widetilde{C}_2m^4\sum_{k=1}^n \left(\widetilde{f}(\beta,m)\right)^{n+1-k}
\left(c_1m^{1-\beta}\right)^{-2(n-k)}
\left(\mathds{E}\left[R_{[0,m^n)}(0,m^n-1)\right]\right)^2\\
&\leq \widetilde{C}_2m^4\widetilde{f}(\beta,m)
\sum_{k=1}^n\left(\frac{80m^{\beta-1}}{c_1(1-\beta)}\right)^{n-k}
\left(\mathds{E}\left[R_{[0,m^n)}(0,m^n-1)\right]\right)^2\\
&\leq \widetilde{C}_2m^4\widetilde{f}(\beta,m)
\sum_{k=1}^n2^{-(n-k)}
\left(\mathds{E}\left[R_{[0,m^n)}(0,m^n-1)\right]\right)^2\\
&\leq \widetilde{C}_6m^5\left(\mathds{E}\left[R_{[0,m^n)}(0,m^n-1)\right]\right)^2
\end{align*}
for some constant $\widetilde{C}_6<\infty$ depending only on $\beta$. This, along with Lemma \ref{R-renorm} and Proposition \ref{R(0,n)main}, implies that \eqref{secondmoment} hods for $N=m,m^2,\cdots$. With the similar arguments for extending to general $N$ in the case $\beta\geq 1$, we can obtain \eqref{secondmoment} for all $N\in \mathds{N}$.
\end{proof}

\subsection{Resistances between points and boxes}\label{sect-pb}
The goal of this section is to establish the equivalence between point-to-point and point-to-box resistances as follows.

\begin{proposition}\label{R-Lambda}
For all $\beta>0$ and $\varepsilon\in (0,1)$, there exist constants $0<c<C<\infty$ {\rm(}depending only on $\beta$ and $\varepsilon${\rm)} such that for all $n\in \mathds{N}$,
\begin{equation*}
\mathds{P}\left[c\Lambda(n)\leq R(0,[-n,n]^c)\leq C\Lambda(n)\right]\geq 1-\varepsilon.
\end{equation*}
\end{proposition}

To prove Proposition \ref{R-Lambda}, we first do some preparations. For $n\in \mathds{N}$ and $\iota\in (0,1)$, recall that $R_\iota^n=[n-\iota n, n)$ and $L_\iota^n=[0,\iota n]$.

\begin{lemma}\label{lem-pointbox}
For all $\beta>0$, there exist  constants $c_2=c_2(\beta)>0$ and sufficiently small $\iota>0$ {\rm(}both depending only on $\beta${\rm)} such that for all $n\in \mathds{N}$,
\begin{equation*}
\mathds{P}\left[R_{[0,n)}(L_\iota^n,R_\iota^n)\geq \frac{1}{4}\mathds{E}\left[R_{[0,n)}(0,n-1)\right]\right]\geq c_2.
\end{equation*}
\end{lemma}
\begin{proof}
For convenience, let $A$ be the event that $R_{[0,n)}(L_\iota^n,R_\iota^n)\geq \frac{1}{4}\mathds{E}[R_{[0,n)}(0,n-1)]$. Then according to Lemma \ref{lem-point-box}, the Cauchy-Schwarz inequality and Proposition \ref{prop-2thmoment}, we get that for sufficiently small $\iota>0$,
\begin{equation*}
\begin{aligned}
\mathds{E}[R_{[0,n)}(0,n-1)]&\leq 2\mathds{E}\left[R_{[0,n)}(L_\iota^n, R_\iota^n)\right]\\
&= 2\mathds{E}[R_{[0,n)}(L_\iota^n, R_\iota^n)\I_{A^c}]+2\mathds{E}\left[R_{[0,n)}(L_\iota^n, R_\iota^n)\I_{A}\right]\\
&\leq \frac{1}{2}\mathds{E}\left[R_{[0,n)}(0,n-1)\right]
+2\left(\mathds{E}\left[R_{[0,n)}(0,n-1)^2\right]\right)^{1/2}\mathds{P}[A]^{1/2}\\
&\leq \frac{1}{2}\mathds{E}\left[R_{[0,n)}(0,n-1)\right]
+C_1^{1/2}\mathds{E}\left[R_{[0,n)}(0,n-1)\right]\mathds{P}[A]^{1/2},
\end{aligned}
\end{equation*}
where $C_1<\infty$ (depending only on $\beta$) is the constant defined in Proposition \ref{prop-2thmoment}. This implies that $\mathds{P}[A]\geq (4C_1)^{-1}$.
\end{proof}

\begin{lemma}\label{point-point}
For all $\beta>0$, there exists a constant $c_3=c_3(\beta)>0$ {\rm(}depending only on $\beta${\rm)} such that for all $n\in \mathds{N}$,
\begin{equation}\label{box-point-2}
\mathds{E}\left[R_{[-n,n]}(0,-n)\right]=\mathds{E}\left[R_{[-n,n]}(0,n)\right]\geq c_3\mathds{E}\left[R_{[0,n)}(0,n-1)\right].
\end{equation}
\end{lemma}
\begin{proof}
By the translation invariance of the LRP model, we can see the first equality in \eqref{box-point-2} holds directly.
Therefore, we only need to prove the inequality in \eqref{box-point-2}.
To this end, by the monotonicity of resistance, we get that
\begin{equation*}
\mathds{E}\left[R_{[-n,2n]}(0,n)\right]\leq \mathds{E}\left[R_{[-n,n]}(0,n)\right].
\end{equation*}
Using this and the translation invariance of the LRP model again yields
\begin{equation*}
\mathds{E}\left[R_{[0,3n]}(n,2n)\right]\leq \mathds{E}\left[R_{[-n,n]}(0,n)\right].
\end{equation*}
Combining this with the triangle inequality of the resistance, we obtain that for all $k\geq 1$,
\begin{equation}\label{Rk-triangle}
\mathds{E}\left[R_{[0,(k+2)n]}(n,(k+1)n)\right]
\leq\sum_{l=1}^k\mathds{E}\left[R_{[(l-1)n,(l+2)n]}(ln,(l+1)n)\right]\leq k\mathds{E}\left[R_{[-n,n]}(0,n)\right].
\end{equation}

We now denote $w_1=\lfloor\frac{n}{k+2}\rfloor$ and $w_2=\lfloor\frac{(k+1)n}{k+2}\rfloor$. Then from Lemma \ref{R-renorm} and \eqref{Rk-triangle}, we can see that
\begin{equation}\label{Rw1w2}
\mathds{E}\left[R_{[0,n]}(w_1,w_2)\right]\leq C_{6,*}\mathds{E}\left[R_{[0,(k+2)n]}(n,(k+1)n)\right]\leq C_{6,*}k\mathds{E}\left[R_{[-n,n]}(0,n)\right],
\end{equation}
where $C_{6,*}<\infty$ (depending only on $\beta$) is the constant defined in Lemma \ref{R-renorm}.

Finally, recall that $\iota>0$ is the constant defined in Lemma \ref{lem-point-box}, and $R_\iota^n=[n-\iota n, n)$ and $L_\iota^n=[0,\iota n]$. Take $k\in \mathds{N}$ large enough such that $w_1\in L_\iota^n$ and $w_2\in R_\iota^n$. It is worth emphasizing that $k$ depends only on $\iota$, and thus depends only on $\beta$.
Then by the monotonicity of resistance, we get that
$$
\mathds{E}\left[R_{[0,n]}(w_1,w_2)\right]\geq \mathds{E}\left[R_{[0,n]}(L_\iota^n,R_\iota^n)\right].
$$
Applying this into \eqref{Rw1w2}, we obtain that
\begin{equation*}
\begin{aligned}
\mathds{E}\left[R_{[-n,n]}(0,n)\right]&\geq \frac{1}{C_{6,*}k}\mathds{E}\left[R_{[0,n]}(w_1,w_2)\right]\\
&\geq\frac{1}{C_{6,*}k}\mathds{E}\left[R_{[0,n]}(L_\iota^n,R_\iota^n)\right]\\
&\geq \frac{1}{2C_{6,*}k}\mathds{E}\left[R_{[0,n]}(0,n)\right]\geq \frac{1}{2C_{6,*}^2k}\mathds{E}\left[R_{[0,n)}(0,n-1)\right],
\end{aligned}
\end{equation*}
where we used Lemmas \ref{lem-point-box} and \ref{R-renorm} for the third and fourth inequalities, respectively.
\end{proof}

\begin{lemma}\label{box-box-lem}
For $\beta>0$, there exist constants $\eta=\eta(\beta)>0$ and $c_4=c_4(\beta)>0$ {\rm(}both depending only on $\beta${\rm)} such that for all $n\in \mathds{N}$,
\begin{equation}\label{box-point-3}
\mathds{E}\left[R_{[-n,n]}([-\eta n, \eta n], [n-\eta n, n+\eta n])\right]\geq c_4\Lambda(n).
\end{equation}
In addition, there exists a constant $c_5=c_5(\beta)>0$ {\rm(}depending only on $\beta${\rm)} such that
\begin{equation}\label{P-box-point}
\mathds{P}\left[R_{[-n,n]}([-\eta n, \eta n], [n-\eta n, n+\eta n])\geq \frac{ c_4}{2}\Lambda(n)\right]\geq c_5
\end{equation}
and
\begin{equation}\label{box-box}
\mathds{P}\left[R([-n,n], [-2n,2n]^c)\geq  \frac{ c_4}{8}\Lambda(n)\right]\geq c_5.
\end{equation}
\end{lemma}

\begin{proof}
By Proposition \ref{R(0,n)main} and Lemma \ref{point-point}, we can see that there exists a constant $\widetilde{C}_1=\widetilde{C}_1(\beta)<\infty$ (depending only on $\beta$) such that
\begin{equation*}
\Lambda(n)\leq C'_{1,*}\mathds{E}\left[R_{[0,n)}(0,n-1)\right]\leq \widetilde{C}_1(\beta)\mathds{E}\left[R_{[-n,n]}(0,n)\right],
\end{equation*}
where $C'_{1,*}<\infty$ (depending only on $\beta$) is the constant defined in Proposition \ref{R(0,n)main}.

In the following, without loss of generality, we assume that $\eta n\in \mathds{N}$. Otherwise, we can replace $\eta n$ by $\lfloor \eta n\rfloor$. Using the similar argument in the proof of Lemma \ref{lem-point-box}, we arrive at
\begin{equation*}
\begin{aligned}
\mathds{E}\left[R_{[-n,n]}(0,n)\right]
&\leq \mathds{E}\left[R_{[-n,n]}([-\eta n, \eta n], [n-\eta n, n+\eta n])\right]+ 2\Lambda(\lfloor\eta n\rfloor)\\
&\leq  \mathds{E}\left[R_{[-n,n]}([-\eta n, \eta n], [n-\eta n, n+\eta n])\right]+ \widetilde{C}_2\eta^\delta \Lambda(n)
\end{aligned}
\end{equation*}
for some constants $\widetilde{C}_2=\widetilde{C}_2(\beta)<\infty$  and $\delta>0$ both depending only on $\beta$.  Combining this with Lemma \ref{point-point} gives us that
\begin{equation*}
\begin{aligned}
\mathds{E}\left[R_{[-n,n]}([-\eta n, \eta n], [n-\eta n, n+\eta n])\right]&\geq \mathds{E}\left[R_{[-n,n]}(0,n)\right]- \widetilde{C}_2\eta^\delta \Lambda(n)\\
&\geq c_3\mathds{E}\left[R_{[0,n)}(0,n-1)\right]- \widetilde{C}_2\eta^\delta \Lambda(n),
\end{aligned}
\end{equation*}
where $c_3>0$ (depending only on $\beta$) is the constant defined in Lemma \ref{point-point}. Consequently, we can obtain \eqref{box-point-3} by Proposition \ref{R(0,n)main} and taking $\eta>0$ (depending only on $\beta$) small enough.
In addition, using similar arguments in the proof of Lemma \ref{lem-pointbox}, we can get \eqref{P-box-point}.

To prove \eqref{box-box}, we start by proving that there exists a constant $\widetilde{c}_1=\widetilde{c}_1(\beta)>0$ (depending only on $\beta$) such that
\begin{equation}\label{box-box-2}
\mathds{P}\left[R([-\eta n,\eta n], [-n,n]^c)\geq \frac{c_4}{8}\Lambda(n)\right]\geq \widetilde{c}_1.
\end{equation}
For that, let $A_\eta$ be the event that $R_{[-n,n]}([-\eta n, \eta n], [x-\eta n, x+\eta n])\geq \frac{ c_4}{2}\Lambda(n)$ for all $x=\pm n$.
Since $\{R_{[-n,n]}([-\eta n, \eta n], [x-\eta n, x+\eta n])\geq \frac{ c_4}{2}\Lambda(n)\}$ is decreasing with the edge set, from the FKG inequality and  \eqref{P-box-point} we get that
\begin{equation}\label{P(A)}
\mathds{P}[A_\eta]\geq c_5^2,
\end{equation}
where $c_5>0$ is the constant defined in \eqref{P-box-point}.
Additionally, denote by $B_\eta$ the event that $[-(n-\eta n), n-\eta n]\nsim [-n,n]$.  Then by a simple calculation, one has
\begin{equation}\label{P(B)}
\begin{aligned}
\mathds{P}[B_\eta]&=\exp\left\{-\int_{-(n-\eta n)}^{n-\eta n}\left(\int_{-\infty}^{-n}+\int_{n}^\infty\right)\frac{\beta}{|u-v|^2}\d u\d v\right\}\\
&=\exp\left\{-\int_{-(1-\eta )}^{1-\eta }\left(\int_{-\infty}^{-1}+\int_{1}^\infty\right)\frac{\beta}{|u-v|^2}\d u\d v\right\}=:b>0.
\end{aligned}
\end{equation}

Note that on the event $B_\eta\cap A_\eta$, for any flow $f$ from $[-\eta n,\eta n]$ to $[-n,n]^c$, it must pass through $\cup_{x=\pm n}[x-\eta n, x+\eta n]$, which means
\begin{equation*}
\begin{aligned}
&R([-\eta n,\eta n],[-n,n]^c)\\
&\geq R([-\eta n,\eta n], \cup_{x=\pm n}[x-\eta n, x+\eta n])\\
&\geq \frac{1}{4}\min\left\{R([-\eta n,\eta n],[-n-\eta n, -n+\eta n]),\ R([-\eta n,\eta n],[n-\eta n, n+\eta n])\right\}\\
&\geq \frac{c_4}{8}\Lambda(n).
\end{aligned}
\end{equation*}
Combining this with \eqref{P(A)}, \eqref{P(B)} and the fact that $A_\eta, B_\eta$ are independent, we obtain that
\begin{equation*}
\mathds{P}\left[R([-\eta n,\eta n],[-n,n]^c)\geq \frac{c_4}{8}\Lambda(n)\right]\geq bc_5^2,
\end{equation*}
which shows \eqref{box-box-2}.

Let $D_{\eta,n}=\{-n+\eta n, -n+2\eta n,\cdots, n-\eta n\}$ . It is obvious that  $\cup_{x\in D_{\eta,n}}[x-\eta n,x+\eta n]=[-n,n]$ and $\#D_{\eta,n}=\lfloor 1/\eta\rfloor +1=:q$. From the monotonicity of the resistance, we have
\begin{equation*}
\begin{aligned}
R([-\eta n,\eta n],[-2n,2n]^c)&=R(\cup_{x\in D_{\eta,n}}[x-\eta n,x+\eta n], [-2n,2n]^c)\\
&\geq \min_{\{\theta_x\}_{x\in D_{\eta,n}}}\sum_{x\in D_{\eta,n}}\theta_x^2 R([x-\eta n,x+\eta n], [-2n,2n]^c)\\
&\geq  \min_{\{\theta_x\}_{x\in D_{\eta,n}}}\sum_{x\in D_{\eta,n}} \theta_x^2 R([x-\eta n,x+\eta n], [x-n,x+n]^c),
\end{aligned}
\end{equation*}
where the minmum is over all $\{\theta_x\}_{x\in D_{\eta,n}}$ with $\theta_x\geq 0$ and $\sum_{x\in D_{\eta,n}}\theta_x=1$.
Therefore, by \eqref{box-box-2} and the FKG inequality again, we arrive at
\begin{equation*}
\begin{aligned}
&\mathds{P}\left[R([-n,n], [-2n,2n]^c)\geq  \frac{ c_4}{8}\Lambda(n)\right]\\
&\geq \mathds{P}\left[R([x-\eta n,x+\eta n], [x-n,x+n]^c)\geq  \frac{ c_4}{8}\Lambda(n)\ \text{for all } x\in D_{\eta,n}\right] \geq \widetilde{c}_1^q,
\end{aligned}
\end{equation*}
which completes the proof.
\end{proof}

\begin{proof}[Proof of Proposition \ref{R-Lambda}]

Fix $\varepsilon\in (0,1)$ and $n\in \mathds{N}$. From the monotonicity of the resistance, we can see that $R(0,[-n,n]^c)\leq R_{[0,n+1]}(0,n+1)$. Consequently,
$$
\mathds{E}\left[R(0,[-n,n]^c)\right]\leq \Lambda(n+1)\leq \Lambda(n)+1.
$$
Combining this with Markov's inequality, we get that
\begin{equation}\label{uppertail}
\mathds{P}\left[R(0,[-n,n]^c)>C\Lambda(n)\right]\leq \frac{\Lambda(n)+1}{C\Lambda(n)}.
\end{equation}
We now choose the constant $C$ sufficiently large such that $2/C<\varepsilon/2$. Then by \eqref{uppertail} one has the upper tail
\begin{equation}\label{uppertail-1}
\mathds{P}\left[R(0,[-n,n]^c)>C\Lambda(n)\right]\leq \varepsilon/2.
\end{equation}

We now turn to the lower tail of $R(0,[-n,n]^c)$.
Fix $K,N\in \mathds{N}$ such that the function $i\rightarrow \Lambda(K^{2i}N)$ is increasing in $i$.
This is possible by Proposition \ref{weak-supermult}. In the following, we consider intervals of the form $[-K^{2(i-1)}N, K^{2(i-1)}N]$ for $i\geq 1$. By a simple calculation, it is clear that
\begin{equation*}
\mathds{P}\left[[-K^{2(i-1)}N, K^{2(i-1)}N]\sim [-K^{2i}N, K^{2i}N]^c\right]=\mathds{P}\left[0\sim [-K^2,K^2]^c\right]\leq 50\beta K^{-2}.
\end{equation*}
Therefore,
\begin{equation}\label{exist-edge}
\mathds{P}\left[\exists i\in [1,K]\ \text{such that }[-K^{2(i-1)}N, K^{2(i-1)}N]\sim [-K^{2i}N, K^{2i}N]^c\right]\leq 50\beta K^{-1}.
\end{equation}
For convenience, denote by $A$ the complement of the event in \eqref{exist-edge}.  So on the event $A$, each unit flow from $[-N,N]$ to $[-K^{2K}N,K^{2K}N]$ needs to flow from $[-K^{2(i-1)}N, K^{2(i-1)}N]$ to $[-K^{2i}N, K^{2i}N]^c$ for all $1\leq i\leq K$.
For all odd $i$, the resistances
$$
R([-K^{2(i-1)}N, K^{2(i-1)}N],[-K^{2i}N, K^{2i}N]^c)
$$
are independent of each other.
Since  the function $i\rightarrow \Lambda(K^{2i}N)$ is increasing in $i$, conditioned on $A$ we have
\begin{equation*}
\begin{aligned}
&\mathds{P}\left[R(0, [-K^{2K}N, K^{2K}N]^c)<\frac{c_4}{8}\Lambda(N)\ |A\right]\\
&\leq \mathds{P}\left[R([-K^{2(i-1)}N, K^{2(i-1)}N], [-K^{2i}N, K^{2i}N]^c)<\frac{c_4}{8}\Lambda(N)\ \text{for all odd }i\in[1,K]_\mathds{Z}\ |A\right]\\
&=\prod_{i=1,i\text{ odd}}^K\mathds{P}\left[R([-K^{2(i-1)}N, K^{2(i-1)}N], [-K^{2i}N, K^{2i}N]^c)<\frac{c_4}{8}\Lambda(N)\ |A\right]\\
&\leq (1-c_5)^{\lfloor K/2\rfloor},
\end{aligned}
\end{equation*}
where $c_4,c_5$ are the constants defined in Lemma \ref{box-box-lem}, and the last inequality is from \eqref{box-box} in Lemma \ref{box-box-lem} and the monotonicity of the resistance. Consequently, we arrive at
\begin{equation}\label{resist-KN}
\begin{aligned}
\mathds{P}\left[R(0, [-K^{2K}N, K^{2K}N]^c)<\frac{c_4}{8}\Lambda(N)\right]&\leq \mathds{P}\left[R(0, [-K^{2K}N, K^{2K}N]^c)<\frac{c_4}{8}\Lambda(N)\ |A\right]+\mathds{P}[A^c]\\
&\leq (1-c_5)^{\lfloor K/2\rfloor}+50\beta K^{-1}.
\end{aligned}
\end{equation}

In addition, by Proposition \ref{submult}, we can see that there exists a constant $C_{1,*}=C_{1,*}(\beta)<\infty$ (depending only on $\beta$) such that
\begin{equation*}
\Lambda(N)\leq \Lambda(K^{2K}N)\leq C_{1,*}\Lambda(K^{2K})\Lambda(N)\leq C_{1,*}K^{2K}\Lambda(N).
\end{equation*}
Combining this with \eqref{resist-KN}, we obtain
\begin{equation}\label{R2KN}
\begin{aligned}
&\mathds{P}\left[R(0, [-K^{2K}N, K^{2K}N]^c)<\frac{c_4}{8C_{1,*}K^{2K}}\Lambda(K^{2K}N)\right]\\
&\leq \mathds{P}\left[R(0, [-K^{2K}N, K^{2K}N]^c)<\frac{c_4}{8}\Lambda(N)\right]\\
&\leq (1-c_5)^{\lfloor K/2\rfloor}+50\beta K^{-1}.
\end{aligned}
\end{equation}

Now for fixed $\varepsilon \in (0,1)$, take $K$ sufficiently large so that $(1-c_5)^{\lfloor K/2\rfloor}+50\beta K^{-1}<\varepsilon/2$.
For sufficiently large $n\in \mathds{N}$ with $n>K^{2K}$, let $N$ be the largest integer for which $K^{2K}N\leq n$. Thus we have $K^{2K}N\leq n\leq 2K^{2K}N$. Then by Lemma \ref{R-renorm} and Proposition \ref{R(0,n)main}, there exists a constant $\widetilde{C}_2=\widetilde{C}_2(\beta)<\infty$ (depending only on $\beta$) such that
\begin{equation*}
\Lambda(n)\leq  \widetilde{C}_2\Lambda(2K^{2K}N)\leq 2\widetilde{C}_2\Lambda(K^{2K}N).
\end{equation*}
Combining this with \eqref{R2KN} and the selection of $K$, we have
\begin{equation*}
\begin{aligned}
&\mathds{P}\left[R(0,[-n,n]^c)<\frac{c_4}{16\widetilde{C}_1\widetilde{C}_2K^{2K}}\Lambda(n)\right]\\
&\leq \mathds{P}\left[R(0,[-n,n]^c)<\frac{c_4}{8\widetilde{C}_1K^{2K}}\Lambda(K^{2K}N)\right]\\
&\leq \mathds{P}\left[R(0,[-K^{2K}N,K^{2K}N]^c)<\frac{c_4}{8\widetilde{C}_1K^{2K}}\Lambda(K^{2K}N)\right]< \varepsilon/2.
\end{aligned}
\end{equation*}
Hence, we obtain the desired result by combining this with \eqref{uppertail-1}.
\end{proof}

\subsection{Resistances between boxes and boxes}\label{sect-bb}
In the section, we mainly establish the following equivalence between point-to-point and box-to-box resistances.

\begin{proposition}\label{box-box-tail}
For $\beta>0$ and $\varepsilon \in (0,1)$, there exist constants $0<c'<C'<\infty$ {\rm(}depending only on $\beta$ and $\varepsilon${\rm)} such that for all $n\in \mathds{N}$,
\begin{equation}\label{P-boxbox-tail}
\mathds{P}\left[c'\Lambda(n)\leq R([-n,n],[-2n,2n]^c )\leq C'\Lambda(n)\ |[-n,n]\nsim [-2n,2n]^c \right]\geq 1-\varepsilon.
\end{equation}
\end{proposition}

To prove Proposition \ref{box-box-tail}, we need some preparations. For that,  fix $\beta>0$ and $\varepsilon \in (0,1)$.  For convenience, we denote by $A_n$ the event that $[-n,n]\nsim [-2n,2n]^c$ for $n\in \mathds{N}$ throughout this subsection.
For any $m_1,m_2\in \mathds{N}$ with $m_1<m_2$, denote by
$$
\mathbb{A}_{m_2,m_1}=[-m_2,m_2]\setminus [-m_1,m_1]
$$
the annulus between $[-m_2,m_2]$ and $[-m_1,m_1]$.
We also fix two constants $N$ and $M$ (depending only on $\beta$), which will be determined later. Let $\mu_M=4\beta M$ and let $\xi$ be the number of $x\in \mathbb{A}_{2MN,MN}$ such that there exists an edge connecting $x$ and $(x-N, x+N)^c$ directly. That is,
\begin{equation}\label{def-xi}
\xi=\#\left\{x\in \mathbb{A}_{2MN,MN}: x\sim [x-(N-1), x+(N-1)]^c\right\}.
\end{equation}

We now define two events as follows.

\begin{definition}\label{event-Bgamma}
For $\gamma>0$, let $B_{\gamma,1}$ be the event that $\xi\geq (1+\gamma^{-1})\mu_M$. Additionally, let $B_{\gamma,2}$ be the event that there exists a point $x\in \mathbb{A}_{2MN,MN}$ such that there exists an edge $e=\langle x,y\rangle$ with $|x-y|\geq N$ and
\begin{equation*}
R\left(x, \mathbb{A}_{2MN,MN}^c;\mathcal{E}\setminus\{e\}\right)< (4\mu_M)^2\gamma \Lambda(MN).
\end{equation*}
Here $R(\cdot, \cdot;\mathcal{E}\setminus\{e\})$ denotes the effective resistance after removing the edge $e$.
\end{definition}

\begin{lemma}\label{prob-Bgamma}
For all $\beta>0$ and $\gamma\in (0,1/2)$, there exists  a function $f_1(\gamma)$ satisfying $f_1(\gamma)\rightarrow 0$ as $\gamma\rightarrow 0$ such that
$$\mathds{P}[B_{\gamma,1}]\leq {\rm e}^{-\mu_M/(2\gamma)}\quad  \text{and}\quad
\mathds{P}[B_{\gamma,2}]\leq f_1(\gamma).$$
\end{lemma}

To prove Lemma \ref{prob-Bgamma}, we need the following event.
\begin{definition}\label{Bgamma'}
For $\gamma'>0$, let $D_{\gamma'}$ be the event that there exists $x\in \partial^{\gamma'}\mathbb{A}_{2MN,MN}$ such that $x\sim (x-N,x+N)^c$, where $$\partial^{\gamma'}\mathbb{A}_{2MN,MN}:=\mathbb{A}_{MN+\gamma'N,MN}\cup \mathbb{A}_{2MN,2MN-\gamma'N}. $$
\end{definition}

\begin{lemma}\label{prob-Dgamma'}
For all $\beta>0$ and all $\gamma'\in(0,1)$, we have $\mathds{P}[D_{\gamma'}]\leq 10 \beta \gamma'$.
\end{lemma}
\begin{proof}
According to $|\partial^{\gamma'}\mathbb{A}_{2MN,MN}|\leq 4\gamma'N$ and the independence of edges in the model, we get that
\begin{equation*}
\begin{aligned}
\mathds{P}[D_{\gamma'}]\leq 4\gamma'N \mathds{P}\left[0\sim (N,N)^c\right]&=8\gamma'N\sum_{k\geq N}\left[1-\exp\left\{-\int_0^1\int_k^{k+1}\frac{\beta}{|u-v|}\d u\d v\right\}\right]\\
&\leq 8\gamma'\beta N\sum_{k\geq N}\frac{1}{k^2}\leq 10\beta\gamma'.
\end{aligned}
\end{equation*}
Hence the proof is complete.
\end{proof}
Using Lemma \ref{prob-Dgamma'}, we can present the
\begin{proof}[Proof of Lemma \ref{Bgamma'}]
For $\gamma\in (0,1/2)$, we start with the event $B_{\gamma,1}$. By the definition of $\xi$ in \eqref{def-xi}, we can see that
\begin{equation*}
\mathds{E}[\xi]\leq \sum_{x\in \mathbb{A}_{2MN,MN}}\sum_{y\in (x-N,x+N)^c}\mathds{P}[x\sim y]\leq \#\mathbb{A}_{2MN,MN}\sum_{y\in (-N,N)^c}\mathds{P}[0\sim y]\leq 4\beta M=:\mu_M.
\end{equation*}
Combining this with the Chernoff bound, we get that
\begin{equation*}
\mathds{P}\left[B_{\gamma,1}\right]=\mathds{P}\left[\xi\geq (1+\gamma^{-1})\mu_M\right]\leq \exp\left\{-\frac{\gamma^{-2}\mu_M}{2+\gamma^{-1}}\right\}\leq \exp\left\{-\mu_M/(2\gamma)\right\}.
\end{equation*}

Additionally, let $\gamma'$ be a constant depending on $\gamma$ such that $\gamma'\rightarrow 0$ as $\gamma\rightarrow 0$. Then from Lemma \ref{prob-Dgamma'} and Proposition \ref{R-Lambda}, we have
\begin{equation*}
\begin{aligned}
\mathds{P}\left[B_{\gamma,2}\right]&=\mathds{P}\left[D_{\gamma'}\right]+ \mathds{P}\left[D_{\gamma'}^c\cap B_{\gamma,2}\right]\\
&\leq 10\beta\gamma'+\mathds{P}\left[\exists x\in \mathbb{A}_{2MN,MN}\setminus \partial^{\gamma'}\mathbb{A}_{2MN,MN}\ \text{such that there exists an edge $e=\langle x,y\rangle$}\right.\\
 &\quad \quad\quad\text{with $|x-y|\geq N$ and }R(x,[x-\gamma'N,x+\gamma'N]^c;\ \mathcal{E}\setminus\{e\})<(4\mu_M)^2\gamma \Lambda(MN) \Big]\\
 &\leq 10\beta\gamma'+ \mu_M\mathds{P}\left[R(0,[-\gamma'N,\gamma' N]^c)<(4\mu_M)^2\gamma \Lambda(MN)\right]=:f_1(\gamma)\rightarrow 0
\end{aligned}
\end{equation*}
as $\gamma\rightarrow 0 $.
\end{proof}

With the above lemmas at hand, we can present the

\begin{proof}[Proof of Proposition \ref{box-box-tail}]
Fix $\beta>0$ and $\varepsilon \in (0,1)$.
By the monotonicity and the translation invariance of the resistance, we first get that for all $n\in\mathds{N}$,
\begin{equation*}
\mathds{E}\left[R([-n,n],[-2n,2n]^c )\ |A_n\right]\leq \mathds{E}\left[R(n,2n+1)\right]\leq \mathds{E}\left[R_{[0,n]}(0,n)\right]+1\leq \Lambda(n)+1.
\end{equation*}
Combining this with the Markov's inequality we obtain that for sufficiently large $\widetilde{C}_1<\infty$ (depending only on $\beta$ and $\varepsilon$) with $\widetilde{C}_1\geq 4/\varepsilon$,
\begin{equation}\label{uppertail-2}
\mathds{P}\left[R([-n,n],[-2n,2n]^c )\geq \widetilde{C}_1\Lambda(n)\ |A_n\right]\leq \frac{\Lambda(n)+1}{\widetilde{C}_1\Lambda(n)+1}\leq \frac{2}{\widetilde{C}_1}\leq \varepsilon /2.
\end{equation}

We now turn to the lower tail of $R([-n,n],[-2n,2n]^c )$. Fix sufficiently large $M,N\in \mathds{N}$, which will be determined later.
For $\gamma\in (0,1/2)$ and $\gamma'>0$ depending on $\gamma$ such that $\gamma'\rightarrow 0$ as $\gamma\rightarrow 0$, recall that $B_{\gamma,1}, B_{\gamma,2}$ and $D_{\gamma'}$ are the events defined in Definitions \ref{event-Bgamma} and \ref{Bgamma'}, respectively.
In the rest of the proof, we condition on $A_{MN}$ and assume that $(B_{\gamma,1}\cup B_{\gamma,2}\cup D_{\gamma'})^c$ occurs. It is worth emphasizing that by Lemmas \ref{prob-Bgamma} and \ref{prob-Dgamma'}, we have
\begin{equation}\label{P-BBD}
\mathds{P}\left[B_{\gamma,1}\cup B_{\gamma,2}\cup D_{\gamma'}\right]\leq \e^{-\mu_M/(2\gamma)}+ f_1(\gamma)+10 \beta \gamma'=:\widetilde{f}(\gamma)\rightarrow 0\quad \text{as} \gamma\rightarrow 0.
\end{equation}
In addition, on the event $(B_{\gamma,1}\cup B_{\gamma,2}\cup D_{\gamma'})^c$, the following holds:
\begin{itemize}
\item[(1)] In $\mathbb{A}_{2MN,MN}$, the number of long edges with a Euclidean length of at least $N$, denoted as $\xi$, satisfies $\xi\leq (1+\gamma^{-1})\mu_M < 2\mu_M/\gamma$. For simplicity, we will denote the set of such edges as $\mathcal{C}_{\geq N}$.
\medskip

\item[(2)]For each $e=\langle x,y\rangle\in  \mathcal{C}_{\geq N}$ with $x\in \mathbb{A}_{2MN,MN}$, we have $x\in \mathbb{A}_{2MN,MN}\setminus \partial^{\gamma'}\mathbb{A}_{2MN,MN}$.
\medskip

\item[(3)] For each $e=\langle x,y\rangle\in  \mathcal{C}_{\geq N}$ with $x\in \mathbb{A}_{2MN,MN}\setminus \partial^{\gamma'}\mathbb{A}_{2MN,MN}$, one has
    $$
    R(x,[x-\gamma'N,x+\gamma'N]^c; \mathcal{E}\setminus\{e\})\geq(4\mu_M)^2\gamma \Lambda(MN).
    $$
\end{itemize}
For convenience, we will also denote $\mathcal{C}_{\geq N}=\{e_1=\langle x_1,y_1\rangle,\cdots,  e_\xi=\langle x_\xi,y_\xi\rangle\}$.

Let $f$ be the unit electric  flow from $[-MN,MN]$ to $[-2MN,2MN]^c$.
Denote $\theta_{\geq N}=\sum_{i=1}^\xi |f_{x_iy_i}|$, which represents the total flow through the long edges in $\mathcal{C}_{\geq N}$.

(a) We first consider the case that $\theta_{\geq N}\geq 1/2$.  Since $\xi=\#\mathcal{C}_{\geq N}<2\mu_M/\gamma$ by (1), we can see that there must exist $e_i=\langle x_i,y_i \rangle \in \mathcal{C}_{\geq N}$ such that
\begin{equation*}
|f_{x_iy_i}|\geq \frac{\theta}{2\mu_M/\gamma}\geq \gamma/(4\mu_M).
\end{equation*}
Combining this with the above (3), we obtain that
\begin{equation*}
R([-MN,MN],[-2MN,2MN]^c)\geq \gamma^3\Lambda(MN).
\end{equation*}

(b) For the case where $\theta_{\geq N}<1/2$, there must exist a subflow $g$ of the unit flow $f$ that satisfies $|g|\geq 1/2$ and does not pass through any long edge in $\mathcal{C}_{\geq N}$. This implies that the subflow $g$ starting from $[-MN,MN]$, must across each annulus $\mathbb{A}_{MN+kN,MN+(k-1)N}$ for $1\leq k\leq M-1$. Consequently,
\begin{equation}\label{R-kN}
\begin{aligned}
&R([-MN,MN],[-2MN,2MN]^c)\geq\frac{1}{2}\sum_{x\sim y}g_{xy}^2\\
&\geq \frac{1}{4}\sum_{k:\ k\leq M-2\ \text{and odd}} R(\mathbb{A}_{MN+kN,MN+(k-1)N}, \mathbb{A}_{MN+(k+2)N,MN+(k+1)N};\mathcal{E}\setminus \mathcal{C}_{\geq N})
\end{aligned}
\end{equation}
(see Figure \ref{CMN} for an illustration).
\begin{figure}[htbp]
\centering
\includegraphics[scale=0.8]{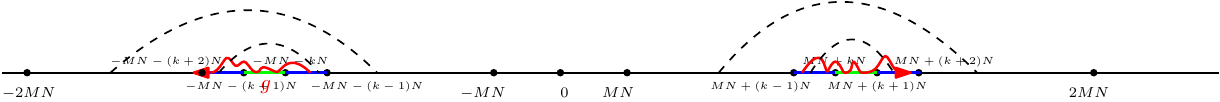}
\caption{The illustration for \eqref{R-kN}. The red curves represent the subflow $g$, which does not pass through any long edges in $\mathcal{C}_{\geq N}$ (shown as dashed curves). Therefore, when the subflow $g$ starting from $[-MN,MN]$, it must enter each annulus of width $N$, denoted as $\mathbb{A}_{MN+kN,MN+(k-1)N}$ for $1\leq k\leq M-1$. When the flow transitions from the $k$-th annulus to the $(k+2)$-th annulus (represented by blue lines), it will generate energy in the intermediate annulus $\mathbb{A}_{MN+(k+1)N,MN+kN}$ (as green lines) given by $R(\mathbb{A}_ {MN+kN,MN+(k-1)N}, \mathbb{A}_{MN+(k+2)N,MN+(k+1)N};\mathcal{E}\setminus \mathcal{C}_{\geq N})$.
}
\label{CMN}
\end{figure}
Now if we assume  that
\begin{equation*}\label{assume-R-MN}
R([-MN,MN],[-2MN,2MN]^c)<\gamma^3\Lambda(MN),
\end{equation*}
then from \eqref{R-kN} and Proposition \ref{submult}, there exists a constant $C_{1,*}=C_{1,*}(\beta)<\infty$ (depending only on $\beta$) such that
\begin{equation}\label{sum-RkN}
\begin{aligned}
&\frac{1}{4}\sum_{k:\ k\leq M-2\ \text{and odd}} R(\mathbb{A}_{MN+kN,MN+(k-1)N}, \mathbb{A}_{MN+(k+2)N,MN+(k+1)N};\mathcal{E}\setminus \mathcal{C}_{\geq N})\\
&\leq \gamma^3\Lambda(MN)\leq C_{1,*} \gamma^3\Lambda(M)\Lambda(N).
\end{aligned}
\end{equation}
Take $M>0$ sufficiently large such that
\begin{equation}\label{cond-M}
\Lambda(M)\leq M/(100C_{1,*}).
\end{equation}
 Then the left-hand side in \eqref{sum-RkN} is upper-bounded by $M\gamma^3 \Lambda(N)/ 100$. This implies that there are at leat two odd $k\in \{1,2,\cdots, M-2\}$  such that
\begin{equation*}
R(\mathbb{A}_{MN+kN,MN+(k-1)N}, \mathbb{A}_{MN+(k+2)N,MN+(k+1)N};\mathcal{E}\setminus \mathcal{C}_{\geq N})\leq \gamma^3 \Lambda(N).
\end{equation*}
Combining this with the independence of
$$R(\mathbb{A}_{MN+kN,MN+(k-1)N}, \mathbb{A}_{MN+(k+2)N,MN+(k+1)N};\mathcal{E}\setminus \mathcal{C}_{\geq N})\quad \text{for all odd } k$$
and the monotonicity of resistance, we get that
\begin{equation*}
\begin{aligned}
&\mathds{P}\left[\frac{1}{4}\sum_{k:\ k\leq M-2\ \text{and odd}} R(\mathbb{A}_{MN+kN,MN+(k-1)N}, \mathbb{A}_{MN+(k+2)N,MN+(k+1)N};\mathcal{E}\setminus \mathcal{C}_{\geq N})\leq M\gamma^3 \Lambda(N)/ 100\right]\\
&\leq M^2\left[\max_{i}\mathds{P}\left[ R(\mathbb{A}_{MN+kN,MN+(k-1)N}, \mathbb{A}_{MN+(k+2)N,MN+(k+1)N};\mathcal{E}\setminus \mathcal{C}_{\geq N})<\gamma^3 \Lambda(N)\right]\right]^2\\
&\leq M^2 \mathds{P}\left[R([-N,N],[-2N,2N]^c)<\gamma^3 \Lambda(N)\ |A_{MN}\right]^2.
\end{aligned}
\end{equation*}

From (a) and (b), we obtain that
\begin{equation*}
\begin{aligned}
&\mathds{P}\left[R([-MN,MN], [-2MN,2MN]^c)<\gamma^3\Lambda(MN),\ (B_{\gamma,1}\cup B_{\gamma,2}\cup D_{\gamma'})^c\ |A_{MN}\right]\\
&\leq M^2 \mathds{P}\left[R([-N,N],[-2N,2N]^c)<\gamma^3 \Lambda(N)\ |A_{MN}\right]^2.
\end{aligned}
\end{equation*}
Additionally, combining this with \eqref{P-BBD}, one has
\begin{equation*}
\begin{aligned}
&\mathds{P}\left[R([-MN,MN], [-2MN,2MN]^c)<\gamma^3\Lambda(MN)\ |A_{MN}\right]\\
&\leq M^2 \mathds{P}\left[R([-N,N],[-2N,2N]^c)<\gamma^3 \Lambda(N)\ |A_{MN}\right]^2+\widetilde{f}(\gamma).
\end{aligned}
\end{equation*}
Define a sequence $\{a_n\}_{n\geq 0}$ by
$$
a_0=\mathds{P}\left[R([-N,N],[-2N,2N]^c)<\gamma^3 \Lambda(N)\ |A_{MN}\right]\quad \text{and}\quad a_{n+1}=M^2 a_n^2+\widetilde{f}(\gamma).
$$
Then inductively it follows that
$$
\mathds{P}\left[R([-M^kN,M^kN],[-2M^kN,2M^kN]^c)<\gamma^3 \Lambda(N)\ |A_{MN}\right]\leq a_k\quad \text{for all }k\geq 0.
$$
For $\widetilde{f}(\gamma)<1/(4M^2)$, the equation $a=M^2a^2+\widetilde{f}(\gamma)$ has the following two solutions:
\begin{equation*}
a_-=\frac{1-\sqrt{1-4M^2\widetilde{f}(\gamma)}}{2M^2}\leq \widetilde{c}_1\widetilde{f}(\gamma)\quad \text{and}\quad a_+=\frac{1+\sqrt{1-4M^2\widetilde{f}(\gamma)}}{2M^2}>\frac{1}{2M^2}
\end{equation*}
for some constant $\widetilde{c}_1>0$ (independent with $M$ and $\gamma$).
Note that for $a_0\in [0,a_+)$, in particular for $a_0\in [0,1/(2M^2)]$, the sequence $a_n$ converges to $a_-\leq \widetilde{c}_1\widetilde{f}(\gamma)$. Hence,
\begin{equation}\label{P-RMN}
\limsup_{k\rightarrow \infty}\mathds{P}\left[R([-M^kN,M^kN],[-2M^kN,2M^kN]^c)<\gamma^3 \Lambda(N)\ |A_{MN}\right]\leq a_-\leq \widetilde{c}_1\widetilde{f}(\gamma).
\end{equation}

We now first choose $\gamma\in (0,1/2)$ small enough such that $\widetilde{c}_1\widetilde{f}(\gamma)<\varepsilon/2$. Then choose $M\in \mathds{N}$ sufficiently large such that \eqref{cond-M} holds. Finally, we choose $N\in\mathds{N}$ large enough such that
\begin{equation*}
a_0=\mathds{P}\left[R([-N,N],[-2N,2N]^c)<\gamma^3 \Lambda(N)\ |A_{MN}\right]\leq \frac{1}{2M^2}.
\end{equation*}
Then combining \eqref{uppertail-2} and \eqref{P-RMN}, we can see that \eqref{P-boxbox-tail} holds for $n=MN,M^2N,\cdots$. To get the statement for all $n\in \mathds{N}$, one can use Lemma \ref{R-renorm} and the fact that $\widetilde{c}_2\Lambda(n)\leq \Lambda(mn)\leq \widetilde{C}_2m\Lambda(n)$ for all $m, n\in \mathds{N}$. Here $0<\widetilde{c}_2<\widetilde{C}_2<\infty$ are constants depending only on $\beta$.
\end{proof}

\section{Supermultiplicativity}\label{sect-supermult}

The aim of this section is to establish the supermultiplicativity of the effective resistance as follows.

\begin{proposition}\label{supermult}
For all $\beta>0$, there exists a constant $c_1=c_1(\beta)>0$ {\rm(}depending only on $\beta${\rm)} such that for all $m,n\in \mathds{N}$,
\begin{equation*}
\Lambda(mn)\geq c_1\Lambda(m)\Lambda(n).
\end{equation*}
\end{proposition}

The primary tools for proving Proposition \ref{supermult} are  the coarse-graining argument and the resistance estimates from Section \ref{sect-severaltype}. Since the proof closely resembles that of Proposition \ref{weak-supermult}, we will omit certain details to focus on the difference between the two proofs.

In the remainder of this section, we fix sufficiently large $m,n\in \mathds{N}$.
For convenience, we denote the interval $[im,(i+1)m)$ in $\mathds{Z}$ as $I_i$.
Recall that $\mathcal{K}_\cdot$ and $\widehat{R}(\cdot,\cdot)$ are defined in \eqref{def-Ki-1} and \eqref{def-Rhat}, respectively.
Similar with Definition \ref{def-alphagood-1}, we now define the very good intervals as follows.
\begin{definition}\label{def-alphagood-2}
For $i\in \mathds{Z}$ and $\alpha=(\alpha_1,\alpha_2)\in (0,1)^2$, we say the interval $I_i$ is \textit{$\alpha$-very good} if it satisfies the following conditions.
\begin{itemize}

\item[(1)]For any two different edges $\langle u_1,v_1\rangle,\langle u_2,v_2\rangle\in \mathcal{E}$ with $v_1,u_2\in I_i$, $u_1\in I_i^c$ and $v_2\in [(i-1)m,(i+2)m)^c$, we have $|v_1-u_2|\geq \alpha_1 m$.
  \medskip

\item[(2)] The internal energy
\begin{equation}\label{def-RIi-2}
R_{I_i}:=\inf_{\theta}\left(\sum_{u\in \mathcal{K}_i}\theta_u^2R_{I_i}\left(u,B_{\alpha_1m}(u)^c\right)+\left(1-\sum_{u\in \mathcal{K}_i}\theta_u\right)^2\widehat{R}(I_{i-1},I_{i+1})\right) \geq \alpha_2 \Lambda(m),
\end{equation}
where the infimum is taken over $\theta=\{\theta_u\}_{u\in \mathcal{K}_i}$ with $\theta_u\geq 0$ and $\sum_{u\in \mathcal{K}_i}\theta_u\leq 1$.
\end{itemize}
\end{definition}

By replacing  Lemma \ref{Lem-RNtail-1} in the proof of Proposition \ref{prop-alphagood-1} with Propositions \ref{R-Lambda} and \ref{box-box-tail}, we can similarly establish the following lower bound for the probability of very good intervals.

\begin{proposition}\label{prop-alphagood}
For all $\beta>0$ and $\varepsilon\in (0,1)$, there exists $\alpha=(\alpha_1,\alpha_2)\in (0,1)^2$ {\rm(}both depending only on $\beta$ and $\varepsilon${\rm)} such that for all $i\in \mathds{N}$,
\begin{equation*}
\mathds{P}\left[I_i\text{ is $\alpha$-very good}\right]\geq 1-\varepsilon.
\end{equation*}
\end{proposition}

\begin{proof}
For any $\varepsilon \in (0,1)$, it follows from Proposition \ref{R-Lambda} that, there exists a constant $\widetilde{c}_1>0$ (depending only on $\beta$ and $\varepsilon$) such that for each $m\in \mathds{N}$,
\begin{equation}\label{pb-Lambda}
\mathds{P}\left[R(0,[-m,m]^c)\geq \widetilde{c}_1\Lambda(m)\right]\geq 1-\varepsilon.
\end{equation}
In addition, recall that $\widehat{R}(\cdot,\cdot)$ is the resistance defined in \eqref{def-Rhat}. According to  the monotonicity of the resistance, we have
\begin{equation*}
\widehat{R}(I_{-1}, I_{1})\geq \widehat{R}([-2m,0), (-\infty,-3m)\cup [m,\infty))\overset{\text{law}}{=}\widehat{R}([-m,m],[-2m,2m]^c).
\end{equation*}
Combining this with Proposition \ref{box-box-tail}, we obtain
\begin{equation}\label{bb-Lambda}
\mathds{P}\left[\widehat{R}(I_{-1}, I_{1})\geq \widetilde{c}_2\Lambda(m)\right]\geq 1-\varepsilon
\end{equation}
for some constant $\widetilde{c}_2>0$ depending only on $\beta$ and $\varepsilon$.
Replacing \eqref{est-RNhat} in the proof of Proposition \ref{prop-alphagood-1} with \eqref{pb-Lambda} and \eqref{bb-Lambda}, we can complete the proof.
\end{proof}

Now recall that $G=(V,E)$ is the renormalization from the $\beta$-LRP by identifying the intervals $I_i$  to vertices $\varpi(i)$, and $G_n$ is the subgraph of $G$ restricted to $V_n=\{\varpi(0),\varpi(1),\cdots,\varpi(n-1)\}$.
We will refer to a vertex $\varpi(i)\in V_n$ as red if it is not $\alpha$-very good, and a vertex set $L\subset V_n$ is referred  to as a red component if it forms a connected subset in $G_n$ consisting only of red vertices and there are no other red vertices directly connected to it.
Additionally, we recall the parameters $\lambda, a_k,b_k$ and $K_{*}$ are defined in \eqref{def-Alambda}-\eqref{def-K}.
We also define the animal sets $\mathcal{L}_{k}, \widetilde{\mathcal{L}}_{k}, \widetilde{\mathcal{L}}^{(1)}_{k\rightarrow l}, \widetilde{\mathcal{L}}^{(2)}_{k\rightarrow l}, \widetilde{\mathcal{L}}^{(2)}_{k\to k,\mathrm{bad}}$ and $\mathcal{G}_k$ from Subsection \ref{sect-cg}, replacing $\alpha$-red animal with red animal.

Now let $\mathcal{H}$ be the event that $\cup_{k>K_{*}}\widetilde{\mathcal{L}}_k=\emptyset$. Using the similar coarse-graining argument in the proof of Proposition \ref{prop-eventH}, we can obtain  that
\begin{equation}\label{sect6-H}
\mathds{P}[\mathcal{H}]\geq 1-n^{-2}.
\end{equation}
Furthermore, on the event $\mathcal{H}$, for any $k\geq 1$ and any animal $L\in \mathcal{G}_k$, it possesses the properties (P1)-(P3) in Section \ref{sect-proofwm}.
Consequently, we also have \eqref{property-g2} for any unit flow from $\varpi(0)$ to $\varpi(n-1)$ restricted to $V_n$.

Based on the above analysis, we can present the
\begin{proof}[Proof of Proposition \ref{supermult}]
Let us first assume that the event $\mathcal{H}$ occurs.
Let $f$ be the unit electric flow from $0$ to $mn-1$ restricted to $[0,mn)$ in the LRP model. We can then construct a unit flow $g$ from $\varpi(0)$ to $\varpi(n-1)$ in the graph $G_n$ based on $f$, defined as
\begin{equation}\label{def-f-g2}
g_{\varpi(i)\varpi(j)}=\sum_{u\in I_i}\sum_{v\in I_j}f_{uv}\quad \text{for all } i,j\in [0,n).
\end{equation}

Applying the flow $g$ to \eqref{property-g2} and according to the definitions of $R_{I_i}$ in \eqref{def-RIi-2} and black vertices (i.e. $\alpha$-very good vertices, see Definition \ref{def-alphagood-2}), we find that on the event $\mathcal{H}$,
\begin{align*}
R_{[0,mn)}(0,mn-1)&=\frac{1}{2}\sum_{u\sim v}f_{uv}^2\geq \sum_{i\in [0,n)}\left(\sum_{u\in I_i}\sum_{v\in I^c_i}f_{vu}\I_{\{f_{vu}>0\}}\right)^2R_{I_i}\\
&= \sum_{i\in[0,n)}\left(\sum_{u\in I_i}\sum_{j\in [0,n)\setminus\{i\}}\sum_{v\in I_j}f_{vu}\I_{\{f_{vu}>0\}}\right)^2R_{I_i}\\
&\geq  \sum_{i\in[0,n)}\sum_{j\in [0,n)\setminus\{i\}}\left(\sum_{u\in I_i}\sum_{v\in I_j}f_{vu}\I_{\{f_{vu}>0\}}\right)^2R_{I_i}\\
&\geq \sum_{i\in[0,n)}\sum_{j\in [0, n)\setminus\{i\}} g^2_{\varpi(j)\varpi(i)}\I_{\{g_{\varpi(j)\varpi(i)}>0\}}R_{I_i} \quad \quad \text{by \eqref{def-f-g2}}\\
&\geq \sum_{\langle \varpi(i),\varpi(j)\rangle\in E_n\setminus(\cup_{k\geq 1}\cup_{L\in \mathcal{G}_k}E_{L\times L})}g_{\varpi(j)\varpi(i)}^2 \I_{\{g_{\varpi(j)\varpi(i)}>0\}}R_{I_i}\\
&\geq \frac{1}{2}\alpha_2 \Lambda(m)\sum_{\langle \varpi(i),\varpi(j)\rangle\in E_n\setminus(\cup_{k\geq 1}\cup_{L\in \mathcal{G}_k}E_{L\times L})}g_{\varpi(i)\varpi(j)}^2 \quad \quad \text{by Definition \ref{def-alphagood-2}}\\
&\geq \frac{1}{2C_1}\alpha_2 \Lambda(m)\sum_{\langle \varpi(i), \varpi(j)\rangle \in E_n}g_{\varpi(i)\varpi(j)}^2\quad \quad \text{by \eqref{property-g2}}\\
&\geq \frac{1}{2C_1}\alpha_2 \Lambda(m)R^{G}_{V_n}(\varpi(0),\varpi(n-1))=:\widetilde{c}_1\Lambda(m)R^{G}_{V_n}(\varpi(0),\varpi(n-1)),
\end{align*}
where $\alpha_2>0$ and $C_1<\infty$ (both depending only on $\beta$) are the constants in Definition \ref{def-alphagood-2} and \eqref{property-g2}, respectively. Consequently, combining this with the scaling invariance of the LRP model and $\mathds{P}[\mathcal{H}^c]\leq n^{-2}$ from \eqref{sect6-H}, we get that
\begin{align*}
&\mathds{E}\left[R_{[0,mn)}(0,mn-1)\right]\\
&\geq \mathds{E}\left[R_{[0,mn)}(0,mn-1)\I_{\mathcal{H}}\right]\\
&\geq \widetilde{c}_1\Lambda(m)\mathds{E}\left[R^{G}_{V_n}(\varpi(0),\varpi(n-1))\I_{\mathcal{H}}\right]\\
&=\widetilde{c}_1\Lambda(m)\left(\mathds{E}\left[R^{G}_{V_n}(\varpi(0),\varpi(n-1))\right]-\mathds{E}\left[R^{G}_{V_n}(\varpi(0),\varpi(n-1))\I_{\mathcal{H}^c}\right]\right)\\
&\geq \widetilde{c}_1\Lambda(m)\left(\mathds{E}\left[R^{G}_{V_n}(\varpi(0),\varpi(n-1))\right]-n\mathds{P}\left[\mathcal{H}^c\right]\right)\quad\quad  \text{by }R^{G}_{V_n}(\varpi(0),\varpi(n-1))\leq n\\
&\geq \widetilde{c}_2\Lambda(m)\Lambda(n)\quad\quad  \text{by Proposition \ref{R(0,n)main} and Remark \ref{rem-delta*} (2)}
\end{align*}

for some constant $\widetilde{c}_2>0$ depending only on $\beta$.
\end{proof}

\section{The proof of Theorem \ref{mainr1} and Corollary \ref{cor-lt}}\label{sect-proofmr}

We first present the proof of Theorem \ref{mainr1}. Although proving that a sequence exhibits polynomial growth via submultiplicativity and supermultiplicativity is a well-established method, we will provide a detailed proof here for the sake of completeness.

\begin{proof}[Proof of Theorem \ref{mainr1}]
We first combine Propositions \ref{submult} (for submultiplicativity) and \ref{supermult} (for supermultiplicativity) to get that there exist constants $0<\widetilde{c}_1<\widetilde{C}_1<\infty$ (depending only on $\beta$) such that for all $n,m\in \mathds{N}$,
\begin{equation}\label{sub-super}
\widetilde{c}_1\Lambda(m,\beta)\Lambda(n,\beta)\leq \Lambda(mn,\beta)\leq \widetilde{C}_1\Lambda(m,\beta)\Lambda(n,\beta).
\end{equation}
Now denote $a_k=\log(\widetilde{C}_1\Lambda (2^k,\beta))$ and $b_k=\log(\widetilde{c}_1\Lambda (2^k,\beta))$ for all $k\in \mathds{N}$. Then from the submultiplicativity in \eqref{sub-super} and Fekete's lemma, we obtain that the following limit exists:
\begin{equation}\label{sub-inf}
\delta(\beta)=\lim_{k\rightarrow \infty}\frac{\log(\Lambda (2^k,\beta))}{\log(2^k)}=\lim_{k\rightarrow \infty}\frac{a_k}{k\log2}=\inf_{k\in \mathds{N}}\frac{a_k}{k\log2}.
\end{equation}
In addition, by the supermultiplicativity in \eqref{sub-super}, we arrive at
\begin{equation*}
b_{k+l}=\log(\widetilde{c}_1\Lambda (2^{k+l},\beta))\geq \log(\widetilde{c}_1\Lambda (2^{k},\beta)\widetilde{c}_1\Lambda (2^{l},\beta))=b_k+b_l.
\end{equation*}
Hence, we also have
\begin{equation*}
\delta(\beta)=\lim_{k\rightarrow \infty}\frac{\log(\Lambda (2^k,\beta))}{\log(2^k)}=\lim_{k\rightarrow \infty}\frac{b_k}{k\log2}=\sup_{k\in \mathds{N}}\frac{b_k}{k\log2}.
\end{equation*}
Combining this with \eqref{sub-inf} yields that
\begin{equation}\label{ineq-2k}
\widetilde{C}_1^{-1}2^{k\delta(\beta)}\leq \Lambda(2^k,\beta)\leq \widetilde{c}_1^{-1}2^{k\delta(\beta)}
\end{equation}
for all $k\in \mathds{N}$. Furthermore, by Lemma \ref{R-renorm} we can extend inequalities in \eqref{ineq-2k} to all integers and get that there exists a constant $\widetilde{C}_2=\widetilde{C}_2(\beta)\in (0,\infty)$ (depending only on $\beta$) such that for all $n\in \mathds{N}$,
\begin{equation}\label{exp-Lambda}
\widetilde{C}_2^{-1}n^{\delta(\beta)}\leq \Lambda(n,\beta)\leq \widetilde{C}_2n^{\delta(\beta)}.
\end{equation}
Therefore, the proof is complete by applying this to Propositions \ref{R(0,n)main}, \ref{R-Lambda} and \ref{box-box-tail}.
\end{proof}

Finally, we will present the proof of Corollary \ref{cor-lt} by using similar arguments in Section \ref{sect-supermult}.

\begin{proof}[Proof of Corollary \ref{cor-lt}]
Fix $N\in \mathds{N}$ and sufficiently small $\varepsilon\in (0,1)$.
For $\alpha=(\alpha_1,\alpha_2)\in(0,1)^2$, recall that the $\alpha$-very good interval and the event $\mathcal{H}$ are  defined in Definition \ref{def-alphagood-2} and in the preceding \eqref{sect6-H}, respectively.

Let $m= (\varepsilon/\alpha_2)^{1/\delta}N$ and $n=N/m=(\alpha_2/\varepsilon)^{1/\delta}$.
Without loss of generality, we assume that $m,n\in \mathds{N}$.
Otherwise, we can replace $m,n$ with $\lfloor m\rfloor$ and $\lfloor n\rfloor$, respectively.
Now applying $n=(\alpha_2/\varepsilon)^{1/\delta}$ to \eqref{sect6-H}, we get that
\begin{equation}\label{prob-H2}
\mathds{P}[\mathcal{H}]\geq 1-(\varepsilon/\alpha_2)^{2/\delta}.
\end{equation}

Let $f$ be the unit electric flow from 0 to $[-N,N]^c$ in the LRP model. Similar to \eqref{def-f-g2}, we construct a unit flow from $\varpi(0)$ to $W^c_n:=\{\varpi(-n),\varpi(-n+1),\cdots, \varpi(n-1)\}^c$ in the graph $G$, defined as
\begin{equation*}\label{def-f-g2}
g_{\varpi(i)\varpi(j)}=\sum_{u\in I_i}\sum_{v\in I_j}f_{uv}\quad \text{for all } i,j\in \mathds{Z}.
\end{equation*}
Applying the similar arguments in the proof of Proposition \ref{supermult} to the flow $g$, we can show that on the event $\mathcal{H}$, there exists a constant $\widetilde{c}_1>0$ (depending only on $\beta$) such that
\begin{align}
R(0,[-N,N]^c)=\frac{1}{2}\sum_{u\sim v}f_{uv}^2
&\geq \frac{1}{2C_1}\alpha_2 \Lambda(m)R^{G}(\varpi(0),W_n^c)\label{RN2}\\
&\geq \widetilde{c}_1\alpha_2m^{\delta}R^{G}(\varpi(0),W_n^c)\quad \quad \text{by \eqref{exp-Lambda}}\nonumber\\
&=\widetilde{c}_1\varepsilon N^{\delta}R^{G}(\varpi(0),W_n^c) \quad \quad \text{by }m=(\varepsilon/\alpha_2)^{1/\delta}N,\nonumber
\end{align}
where $\delta>0$ (depending only on $\beta$) is the exponent defined in \eqref{sub-inf}.

In addition, according to the scaling invariance of the LRP model, Lemma \ref{Lem-RNtail-1} and the fact that $n=(\alpha_2/\varepsilon)^{1/\delta}$, we have
\begin{equation*}\label{Rnn}
\begin{aligned}
\mathds{P}\left[R^{G}(\varpi(0),W_n^c)\geq c_*\right]
&\geq \mathds{P}\left[R(0,[-n,n]^c)\geq c_*+1\right]\geq 1-(\varepsilon/\alpha_2)^{\frac{1}{C_*\delta}},
\end{aligned}
\end{equation*}
where $0<c_*<C_*<\infty$ (both depending only on $\beta$) are the constants defined in Lemma \ref{Lem-RNtail-1}.
Applying this and \eqref{prob-H2} to \eqref{RN2}, we get that
\begin{equation*}
\mathds{P}\left[R(0,[-N,N]^c)\geq \widetilde{c}_2 \varepsilon N^\delta\right]\geq 1-(\varepsilon/\alpha_2)^{2/\delta}-(\varepsilon/\alpha_2)^{\frac{1}{C_*\delta}}\geq 1-\varepsilon^q
\end{equation*}
for some $\widetilde{c}_2, q>0$ depending only on $\beta$. Hence, the proof is complete.
\end{proof}

\bigskip

\noindent{\bf Acknowledgement.} \rm
We warmly thank Takashi Kumagai for stimulating discussions at an early stage of the project.
J.\ Ding is partially supported by NSFC Key Program Project No.\ 12231002 and the Xplorer prize. L.-J.\ Huang is partially supported by National Key R\&D Program of China No.\ 2023YFA1010400, and the National Natural Science Foundation of China No.\ 12471136.

\noindent\textbf{Data Availability} Data sharing is not applicable to this article as no datasets were generated or analyzed during the current study.

\noindent\textbf{Conflict of interest} On behalf of all authors, the corresponding author states that there is no conflict of interest.

\bibliographystyle{plain}
\bibliography{resistance-ref2}

\end{document}